\RequirePackage{fix-cm}
\documentclass[smallcondensed]{svjour3}     
\smartqed  
 \usepackage[utf8]{inputenc}
 \usepackage[T1]{fontenc}
 \usepackage{lmodern}
 \usepackage{amsmath}
 \usepackage{amssymb}
 \usepackage{dsfont}
\usepackage{hyperref}
\usepackage{relsize}
\usepackage{exscale}
\usepackage{ulem}
\newcommand{\tr}{\rm Tr}

\newcommand{\ra}{\rightarrow}
\newcommand{\E}{\mathbb E}

\newcommand{\Pp}{\mathbb P}

\newcommand{\R}{\mathbb R}
\newcommand{\N}{\mathbb N}

\newcommand\Car{\mathds{1}}
\renewcommand\phi\varphi 
\newcommand\eps{\varepsilon}
\usepackage[svgnames]{xcolor}

\spnewtheorem{assum}{Assumptions}{\bf}{\it}

 \journalname{}
\begin{document}

\title{Large deviations  for the largest eigenvalue of  sub-Gaussian matrices
\thanks{This project  was partially supported by Labex MILYON and has received funding from the European Research Council (ERC) under the European Union
Horizon 2020 research and innovation program (grant agreements No. 692452 and No. 884584).}}


\author{Fanny Augeri         \and
        Alice Guionnet  \and \\  Jonathan Husson}

\authorrunning{Augeri, Guionnet and Husson} 

\institute{Fanny  Augeri \at Weizmann Institute of Science, 234 Herzl Street, Rehovot, 7610001, Israel. \\
              \email{fanny.augeri@weizmann.ac.il}           
           \and
           Alice Guionnet \at \'Ecole Normale Supérieure de Lyon, 46 all\'ee d'Italie, 69364 Lyon, France. \\
               \email{Alice.Guionnet@ens-lyon.fr}
		\and
           Jonathan Husson \at \'Ecole Normale Supérieure de Lyon, 46 all\'ee d'Italie, 69364 Lyon, France.\\
		\email{Jonathan.Husson@ens-lyon.fr}
 }

\date{}

\maketitle

\begin{abstract}
We establish large deviations estimates for the largest eigenvalue of Wigner matrices with sub-Gaussian entries. 
Under technical assumptions, we show that the large deviation behavior of the largest eigenvalue is universal for small deviations, in the sense that the speed and the rate function are the same as in the case of the GOE. In contrast, in the regime of very large deviations, we obtain a non-universal rate function,
thus establishing the existence of a transition between two different large deviation mechanisms.

\end{abstract}

\section{Introduction}
 In a breakthrough paper \cite{Wig58}, Wigner showed  that the empirical distribution of the eigenvalues of a Wigner matrix converges to the semi-circle law provided the off-diagonal entries have a finite second moment. Following the pioneering work of  K\'omlos and F\H{u}redi  \cite{FuKo}, it was proved in  \cite{BY} that assuming the Wigner matrix has centered entries, the largest eigenvalue converges to the right edge of the support of the semi-circle law  if and only if  the fourth moment of the off-diagonal entries is finite. But what is the probability that  the empirical measure or the largest eigenvalue have a different  behavior ? Analyzing the probability that the largest eigenvalue of a random matrix takes an unexpected value  is a challenging question, with many applications in statistics \cite{fey},  mobile communications systems \cite{fey,BDMN11} or the energy landscape of disordered systems \cite{MontaBA,Biro}.   This turns out to be a much more challenging question than to analyze  the typical behavior which could be only answered so far for very specific models. It was first solved in the case of Gaussian ensembles, such as the Gaussian Unitary Ensemble (GUE) and Gaussian Orthogonal Ensemble (GOE),   where the joint law of the eigenvalues is explicit. In these cases, large deviations principles were derived for the empirical distribution of the eigenvalues and the largest eigenvalue
 in  \cite{BAG97} and \cite{BADG01} merely by Laplace's method, up to taking care of the singularity of the interaction. The question was revived 
 in a breakthrough paper by Bordenave and Caputo \cite{BordCap} who considered Wigner matrices with entries with tails  heavier than in the Gaussian case. They proved a large deviations principle for the law of the empirical measure by a completely different argument based on the fact that deviations are created by  a relatively small number of  large entries.
 These large deviations have a smaller speed than in the   Gaussian case. This phenomenon was shown to hold as well for the largest eigenvalue by one of the authors \cite{fanny}.  

Yet, the case of sub-Gaussian entries remained open and the general mechanism which creates large  deviations  mysterious. Last year, two of the authors  showed in \cite{HuGu} that if the Laplace transform of the entries is pointwise bounded  from above by the one of the GUE or GOE, then a large deviations principle holds with the same rate function as in the Gaussian case. This special case of entries,  which was said to have \textit{sharp sub-Gaussian tails}, includes  Rademacher variables  and uniform variables. Yet, many entries with sub-Gaussian tails are not sharp, as for instance sparse entries which are obtained by multiplying  a Gaussian variable  with  an independent  Bernoulli random variable.  In this article, we investigate this general setting. We derive large deviations estimates for the largest eigenvalue of Wigner matrices with sub-Gaussian entries. In particular, we show that the rate function of this large deviations estimates is different from the one of the GOE. 

This article is restricted to  matrices with entries which  are  centered   and covariance of order of the inverse of the  dimension. Non centered models, such as the adjacency matrix of Erd\H{o}s-R\'enyi matrices, may have different deviations properties  as the mean of the entries can be seen as a rank one deformation of the later.
The large deviation behavior of the extreme eigenvalues of Erd\H{o}s-Rényi graphs has attracted some attention recently. Inside the regime where the average degree goes to infinity, that is $np \gg 1$ where $n$ is the number of vertices and $p$ is the edge-probability,  Cook and Dembo \cite{CoDe}  have computed the tail distribution of the operator norm for $p \gg  n^{-1/2}$, and proved that it is governed by a certain mean-field variational problem. Very recently, the joint large deviations of the extreme eigenvalues were established in \cite{BBG} for the ``localized'' regime where $1 \ll np \ll \sqrt{\log n}$. The case were $np$ is of order one is still open and might be studied   by our techniques.

 We will consider hereafter a $N\times N$ symmetric random matrix $X_N$ with independent entries $(X_{ij})_{i\le j}$ above the diagonal so that $\sqrt{N}X_{ij}$
has law $\mu$ for all  $i<j$ and $\sqrt{N/2}X_{ii}$ has law $\mu$ for all $i$.  In particular, the variance profile is the same as the one of the GOE. We assume that $\mu$ is centered and has a variance equal to $1$. For a real number $x$, let 
$$  \psi(x) = \frac{1}{x^2} \log \int e^{xt} d\mu(t)\,.$$
 $\psi$ is a continuous function on the real line such that $\psi(0) =1/2$.
 Assume that $\mu$ is sub-Gaussian so that
\begin{equation}\label{subGauss} \frac{A}{2}:=\sup_{x\in\mathbb R}\psi(x)<+\infty.\end{equation}
The case where $A=1$ is the case of \textit{sharp sub-Gaussian tails}  studied in \cite{HuGu}. We investigate here the case where $A>1$ and we show the following result.  
\begin{theorem}\label{main}Denote by  $\lambda_{X_N}$ the largest eigenvalue of $X_N$.
Under some technical assumptions, there exist a good rate function $I_\mu : \R \to [0,+\infty]$ and a  set $\mathcal{O}_\mu\subset \R $ such that $(-\infty,2] \cup [x_\mu,+\infty) \subset \mathcal{O}_\mu$ for some $x_\mu\in (2,+\infty)$ and such that for any $x\in \mathcal{O}_\mu$,
$$\lim_{\delta\rightarrow 0}\liminf_{N\ra+\infty}\frac{1}{N}\log \Pp\left(|\lambda_{X_N}-x|\le\delta\right)=\lim_{\delta\rightarrow 0}\limsup_{N\ra+ \infty}\frac{1}{N}\log \Pp\left(|\lambda_{X_N}-x|\le\delta\right)=-I_\mu(x).$$
The rate function $I_\mu$ is infinite on $(-\infty,2)$  and satisfies
 $$ I_\mu(x)\underset{x\rightarrow +\infty}{\sim} \frac{1}{4A}x^2.$$
If $A\in (1,2)$, then $[2, \sqrt{A-1}+1/\sqrt{A-1}]\subset \mathcal O_\mu$ and $I_\mu$ coincides on this interval with the rate function of the GOE, that is, 
\begin{equation} \label{GOE} I_\mu(x)=\frac{1}{2}\int_2^x\sqrt{y^2-4} dy=:I_{GOE}(x).\end{equation}
Moreover, for all $x\ge 2$, $I_\mu(x)\le I_{GOE}(x)$.
\end{theorem}
The technical assumptions include the case where $\psi$ is increasing (which holds in the case of sparse Gaussian entries) and the case where the maximum of $\psi$ is achieved on $\R$ at a unique point in a  neighborhood of which it is strictly concave.  In the later case, $I_{\mu}(x)$ only depends on $A$ for $x$ large enough.

The method introduced in \cite{HuGu} is based on a tilt of the measure by spherical integrals and therefore the estimation of the annealed spherical integrals (given by the average of spherical integrals over the entries of the matrix $X_{N}$). In the case of  sharp  sub-Gaussian entries, the annealed spherical integrals are easy to estimate
because they concentrate on delocalized vectors. In the general case we deal with in this article, estimating the annealed spherical integral becomes much more complicated and interesting because it can concentrate on localized vectors, at least in {a regime corresponding to sufficiently large deviations}.  This new phenomenon comes  with a  ``phase'' transition  at least when $A<2$, since then for small deviations the annealed spherical integral concentrates on delocalized vectors whereas for large deviations it concentrates on more localized vectors. This fact is reflected in the  eigenvector of the largest eigenvalue when the later is conditioned to be large, see section \ref{section:loc}. We show that it is delocalized for sharp sub-Gaussian entries, whereas otherwise it localizes for large enough deviations. The existence of such a transition makes it difficult to compute large deviations  on the whole real line, but in fact our formulas may  just be wrong then. For instance, our formulas would  predict a convex rate function, which may not always be the case.

\subsection{Assumptions}  We now describe more precisely our assumptions.

\begin{assum}\label{AG} Let $\mu \in\mathcal{P}(\R)$ be a symmetric probability measure with unit variance. We denote by $L$ its log-Laplace transform,
 $$ \forall x \in \R, \ L(x)=\log \int e^{x t}d\mu(t),$$
and $\psi(x) = L(x)/x^2$. We assume that $\mu$ is sub-Gaussian in the sense that
$$\frac{A}{2} : = \sup_{x\in \R} \psi(x)<
+\infty,$$
and we define $B\geq 0$ by
$$\frac{B}{2} :=\lim_{|x|\ra+\infty}\psi (x).$$
We assume moreover that $L(\sqrt{.})$ is a Lipschitz function and that $\mu$ does not have sharp sub-Gaussian tails, meaning that $A>1$.
\end{assum}
We describe below a few examples of probability measures $\mu$ which satisfy the above assumptions. In each of these cases, the fact that $L(\sqrt{.})$ is Lipschitz is clear and left to the reader.
\begin{example}\label{example}
 \begin{itemize}
\item (Combination of Gaussian and Rademacher laws). Let
$$\mu(dx)=a \frac{e^{-\frac{1}{ 2B} x^{2} }}{\sqrt{2\pi B}} dx+(1-a)\frac{1}{2}(\delta_{-b}+\delta_{+b})$$
where $a,b,B$ are non negative real numbers such that $a \in (0,1)$ and $a B+(1-a)b^{2}=1$. Then, for all $x \in \R$,
$$L_{\mu}(x)=\log\Big(a e^{\frac{B}{2} x^{2}}+(1-a) \cosh (b x)\Big)\,.$$
If $B>1$ and $b\in(0,1)$ we see that our conditions are fulfilled and $A=B$.
\item (Sparse Gaussian case). Let $\mu$ be the law of  $\zeta \Gamma$ with $\zeta$ a Bernoulli variable of parameter $p\in(0,1)$ and $\Gamma$ a centered Gaussian variable with variance $1/p$. For any $x\in\R$,
$$L_\mu(x)=\log\big(p e^{\frac{x^2}{2p}}+1-p\big)$$
so that $A=B=\frac{1}{p}$.
\item (Combination of Rademacher laws). Let $$\mu=\sum_{i=1}^{p}\frac{\alpha_{i}}{2}(\delta_{\beta_{i}}+\delta_{-\beta_{i}})$$ with $\alpha_{i}\ge 0$ , $\beta_i \in \R$ and $p\in\N$ 
so that $\sum\alpha_{i}=1, \sum \alpha_{i}\beta_{i}^{2}=1$. Since $\mu$ is compactly supported $B=0$. The fact that $\mu$ does not have sharp sub-Gaussian tails means that there exist some $t$ and $A>1$ such that
$$\sum_{i=1}^p \alpha_{i}\cosh(\beta_{i }t)\ge e^{A\frac{t^{2}}{2}}\,.$$
The latter is equivalent to
$$\sum_{i=1}^p\frac{\alpha_{i}}{2}e^{\frac{\beta_{i}^{2}}{2 A}}\Big( e^{-\frac{A}{2}(t-\frac{\beta_{i}}{A})^{2}}+ e^{-\frac{A}{2}(t+\frac{\beta_{i}}{A})^{2}}\Big)\ge 1\,.$$
This inequality holds as soon as $\alpha_{i}e^{\frac{\beta_{i}^{2}}{2 A}}\ge 2$ for some $i\in\{1,\ldots,p\} $ by taking
$t=\frac{\beta_{i}}{A}$. This can be fulfilled if $\beta_{i}$ is large enough while $\alpha_i \beta_i^2 <1$. We also see with this family of examples that $A$ can be taken arbitrarily large even if $B=0$ (take e.g $p=2$, $A=\beta_{1}, t=1,\alpha_{1}=(2\beta_{1}^{2})^{-1}$, $e^{\beta_{1}/2}\ge 4\beta_{1}^{2}, \beta_{2}^{2}=(2-\beta_{1}^{-2})^{-1}, \alpha_{2}=1-\alpha_{1}$).
\end{itemize}
\end{example}
 Let $\mathcal{H}_N$ be the set of real symmetric matrices of size $N$. We denote for any $A\in\mathcal{H}_N$ by $\lambda_A$ its largest eigenvalue, $||A||$ is spectral radius and by $\hat \mu_A$ the empirical distribution of its eigenvalues, that is
$$ \hat \mu_A = \frac{1}{N}\sum_{i=1}^N \delta_{\lambda_i},$$
where $\lambda_1,\ldots,\lambda_N$ are the eigenvalues of $A$.
We make the following assumption of concentration of the empirical distribution of the eigenvalues at the scale $N$.
 \begin{assum} \label{ass} 
The empirical distribution of the eigenvalues $\hat \mu_{X_N}$ concentrates at the scale $N$:
\begin{equation}\label{concspmeasure} \limsup_{N \to +\infty} \frac{1}{N} \log \Pp\left( d( \hat \mu_{X_N}, \sigma) > N^{-\kappa} \right) = - \infty, \end{equation}
for some $\kappa>0$, where $d$ is a distance compatible with the weak topology and $\sigma$ is the semi-circle law, defined by
$$ \sigma(dx) = \frac{1}{2\pi}\sqrt{4-x^2} \Car_{|x|\leq 2}dx.$$

\end{assum}
\begin{remark}
\begin{enumerate}
\item From \cite[Lemmas 1.8, 1.11]{HuGu}, we know that Assumption \ref{ass} is fulfilled  if  $\mu$ is either compactly supported, or if $\mu$ satisfies a logarithmic  Sobolev inequality in the sense that there exists $c>0$ so that for any smooth function $f : \R \to \R$, such that $\int f^2 d\mu =1$,
$$\int f^{2}\log f^{2} d\mu\le c\int \|\nabla f\|_2^{2} d\mu\,.$$
\item If  $\mu$ is a symmetric sub-Gaussian probability measure on $\R$ with log-concave tails in the sense that $ t\mapsto \mu(|x|\ge t)$  is a log-concave function, then the Wigner matrix $X_N$ satisfies Assumption \ref{ass}.  In particular, if $B$ is a Wigner matrix with Bernoulli entries with parameter $p$ and $\Gamma$ is a GOE matrix, then the sparse Gaussian matrix $B \circ \Gamma/\sqrt{p}$, where $\circ$ is the Hadamard product, satisfies  Assumption \ref{ass}. We refer the reader to section \ref{concappendix} of the appendix for more details.

\end{enumerate}
\end{remark}
{ Let us remark  that when $\mu$ is sub-Gaussian,  \eqref{subGauss}, the spectral radius of $X_{N}$ is exponentially  tight
 \cite[Lemma 5.1]{HuGu}  and that this fact remains true under any tilted measure 
\begin{equation}\label{deftilt}\Pp^{(e,\theta)} = \frac{e^{\theta N\langle e, X_N e\rangle}}{\E_X(e^{\theta N\langle e, X_N e\rangle})} d\Pp(X),\end{equation} where $e\in\mathbb{S}^{N-1}$ and $\theta \geq 0$. More precisely, we can make the following remark.}
\begin{remark}\label{tensionexporadius}{If  \eqref{subGauss} holds, then for any $\theta \geq 0$,
$$\lim_{K\to +\infty} \limsup_{N\to +\infty}\sup_{e \in \mathbb{S}^{N-1}} \frac{1}{N} \log \Pp^{(e,\theta)}\big( || X_N|| \geq K \big) = - \infty.$$}
\end{remark}
\subsection{Statement of the results and scheme of the proof}
As in \cite{HuGu}, our approach to derive  large deviations estimates is based on a tilting  of the law of the Wigner matrix $X_N$ by spherical integrals.
Let  us recall the definition of spherical integrals. For any $\theta \geq 0$, we define
$$I_N(X_N, \theta)=\E_e[ e^{\theta N\langle e, X_N e\rangle}]$$
where $e$ is uniformly sampled on the sphere $\mathbb S^{N-1}$ with radius one. 
The asymptotics of 
$$J_N(X_N,\theta)=\frac{1}{N}\log I_N(X_N,\theta)$$ 
were studied in \cite{GuMa05} where  the following result was proved.
\begin{theorem}\cite[Theorem 6]{GuMa05}\label{myl}
Let $(E_N)_{N \in \N}$ be a sequence of $N \times N$ real symmetric matrices  such that:
\begin{itemize}
\item The sequence of empirical measures $\hat\mu_{E_N}$ converges weakly to a compactly supported measure $\mu$.
\item There is a real number $\lambda_{E}$ such that  the sequence of the largest eigenvalues $\lambda_{E_N}$ converges to $\lambda_{E}$.
\item $ \sup_N || E_N||<+\infty$.
\end{itemize}
For any $\theta \geq 0$, 
\[ \lim_{N \to + \infty} J_N(E_N,\theta) = J(\mu, \lambda_{E},\theta) \]
\end{theorem} 
The limit $J$ is defined as follows.  For a compactly supported probability measure $\mu \in\mathcal{P}(\R)$ we define its Stieltjes transform $G_\mu$ by 
\[ \forall z \notin \mathrm{supp}(\mu), \ G_{\mu}(z) := \int_{\R} \frac{1}{z-t} d \mu(t), \]
where $\mathrm{supp}(\mu)$ is the support of $\mu$.
Let $r_\mu$ denote the right edge of the support of $\mu$. Then  $G_{\mu}$ is a bijection from $(r_\mu, +\infty)$ to $\big(0, G_{\mu}(r_\mu)\big)$ where
$$G_{\mu}(r_\mu) = \lim_{t \downarrow r_\mu} G_{\mu}(t).$$
Let $K_{\mu}$ be the inverse of $G_\mu$ on $(0, G_{\mu}(r_\mu))$ and let
$$\forall z \in (0, G_{\mu}(r_\mu)), \ R_{\mu}(z) := K_{\mu}(z) - 1/z,$$
be the $R$-transform of $\mu$ as defined by Voiculescu in \cite{Vo5}. 
Then, the limit of spherical integrals is defined for any $\theta \ge 0$ and $x \ge r_\mu$ by,
$$
J( \mu,  x,\theta):= \theta v(\mu, x,\theta) -\frac{1}{2} \int \log\big(1+2\theta v(\mu,x, \theta) - 2\theta y\big) d\mu( y),
$$
with 
$$
 v(\mu, x, \theta) :=\begin{cases}
                                 R_\mu(2 \theta) & \text{ if } 0 \le 2 \theta \le G_{\mu}( x),\\
                                x - \frac{1 }{2\theta} & \text{ if } 2 \theta >  G_{\mu}( x).
                                \end{cases}
$$
In the case of the semi-circle law, we have
$$ G_\sigma(x)=\frac{1}{2}(x-\sqrt{x^2-4}), \ R_{\sigma}(x) =x.$$
We denote   by $J(x,\theta)$ as a short-hand for $J(\sigma, x,\theta)$. In the next lemma we compute explicitly $J(x,\theta)$, whose proof is left to the reader.

\begin{lemma}\label{calculJ} Let $\theta \geq 0$ and $x \geq 2$.
For $\theta\le \frac{1}{2} G_\sigma(x)$,
$$J(x,\theta)= \theta^2.$$
Whereas for $\theta\ge \frac{1}{2} G_\sigma(x)$,
$$J(x,\theta)=\theta x-\frac{1}{2} -\frac{1}{2} \log 2\theta -\frac{1}{2}\int \log(x-y)d\sigma(y).$$

\end{lemma}
To derive large deviations estimates using a tilt by spherical integrals, it is central to obtain the asymptotics of the annealed spherical integral $F_N(\theta)$ defined as, 
$$ F_N( \theta) = \frac{1}{N} \log \E_{X_{N}} \E_{e}[ \exp(N \theta \langle e, X_{N} e \rangle) ] \,.$$
In the following lemma, we obtain the limit of $F_N$ as the solution of a certain variational problem. We denote by $\overline{F}(\theta)$ and $\underline{F}(\theta)$ its upper and lower limits:
$$\overline{F}(\theta) = \limsup_{N\to +\infty} F_N(\theta),$$
$$\underline{F}(\theta) = \liminf_{N\to +\infty} F_N(\theta).$$
For any measurable subset $I \subset \R$, we denote by $\mathcal{M}(I)$ and $\mathcal{P}(I)$ respectively  the set of measures and the set of probability measures supported on $I$.
\begin{proposition}\label{simpleFint} Assume $X_N$ satisfies Assumptions \ref{AG} and \ref{ass}.
$$\overline{F}(\theta)=  \limsup_{ \delta \to 0, K\to +\infty \atop \delta K \to 0}\sup_{ \alpha_1 + \alpha_2 +\alpha_3 = 1\atop \alpha_i\geq 0} \limsup_{N\ra+\infty} \mathcal{F}^N_{\alpha_1,\alpha_2,\alpha_3}(\delta, K),$$
$$\underline{F}(\theta)=\sup_{\alpha_1 + \alpha_2 +\alpha_3 =1 \atop \alpha_i \geq  0} \liminf_{ \delta \to 0, K\to +\infty \atop \delta K \to 0} \limsup_{N\ra+\infty} \mathcal{F}^N_{\alpha_1,\alpha_2,\alpha_3}(\delta, K)\,.$$
$ \mathcal{F}^N_{\alpha_1,\alpha_2,\alpha_3}(\delta, K)$ is the function given by:
\begin{align*} \mathcal{F}^N_{\alpha_1,\alpha_2,\alpha_3}&(\delta, K)  =  \theta^2 \big( \alpha_1^2 +2\alpha_1 \alpha_2 + B \alpha_3^2\big) \\
&+ \sup_{t_i  \in I_2 , i\leq l \atop |\sum_i t_i^2 - N\alpha_2|\leq \delta N}
\sup_{ s_i  \in I_3 , i\leq k\atop |\sum_i s_i^2 - N\alpha_3|\leq \delta N} \Big\{\frac{1}{N}\sum_{i=1}^k\sum_{j=1}^{l} L\Big( \frac{2\theta s_i t_{j}}{\sqrt{N}}\Big)  + \frac{1}{2N} \sum_{i,j=1}^{l} L\Big(\frac{2\theta t_{i}t_{j}}{\sqrt{N}} \Big) \\
&+ \sup_{ \nu_1 \in \mathcal{P}(I_1) \atop \int x^2 d\nu_1(x) = \alpha_1} \Big\{\sum_{i=1}^k \int L\Big(\frac{2\theta s_ix}{\sqrt{N}}\Big) d\nu_1(x)-H(\nu_1)\Big\} - \frac{1}{2} \log (2\pi) - \frac{1}{2}\Big\},
\end{align*}
where $I_1 = \{ x : |x|\leq  \delta^{1/2}N^{1/4} \}$, $I_2 = \{ x : \delta^{1/2} N^{1/4} < |x|\leq K^{1/2}N^{1/4} \}$, $I_3 = \{x :  K^{1/2}N^{1/4} < |x| \leq \sqrt{N\alpha_3} \}$, and
$$ H(\nu) = \int \log \frac{d\nu}{dx} d\nu(x),$$
if $\nu$ is absolutely continuous with respect to the Lebesgue measure, whereas $H(\nu)$ is infinite otherwise. 
\end{proposition}
 \begin{remark}\label{remF}
Note that $ \underline{F}$ and $\overline{F}$ are convex by H\"older inequality. Since the entries of $X_N$ are sub-Gaussian, $\overline{ F}(\theta)\le A\theta^2$. In particular $\overline{F}$ and $\underline{F}$ are finite convex functions and therefore  are continuous on $\R_+$. 
 \end{remark}

The above proposition gives quite an intricate definition for the limit of the annealed spherical integrals. Yet, for small enough $\theta$ it can be computed explicitly.
\begin{lemma}\label{F0}
For any $\theta \leq \frac{1}{2\sqrt{A-1}}$,
$$ \underline{F}(\theta) = \overline{F}(\theta) =\theta^2.$$

\end{lemma}
Note that for large $\theta$ this formula is not valid anymore when $A>1$ since $\underline F$ grows like $A\theta^2$ at infinity (see the proof of Proposition \ref{proprate}).

\begin{proof}
Using the bound $L(x) \leq A x^2/2$ for any $x\geq 0$ and the notation of Proposition \ref{simpleFint}, we have
\begin{align*}
\overline{F}(\theta) \leq \sup_{\alpha_1+\alpha_2+\alpha_3= 1}\Big\{  & \theta^2 \big( \alpha_1^2 +2\alpha_1 \alpha_2 +B\alpha_3^2 + 2A \alpha_3\alpha_2+A\alpha_2^2+2A\alpha_1\alpha_3 \big) +\frac{1}{2} \log \alpha_1 \Big\}\,.
\end{align*}
Here we used the fact that 
$$ \inf\{ H(\nu_1) : \int x^2 d \nu_1 =\alpha_1, \nu_1 \in \mathcal{P}(I_{1})\}\ge \inf\{ H(\nu_{1}): \int x^2 d \nu_1 =\alpha_1, \nu_1 \in \mathcal{P}(\mathbb R)\},$$
where the infimum in the RHS is achieved at $\nu_{1}(dx)=(2\pi\alpha_{1})^{-1/2} e^{-\frac{x^{2}}{2\alpha_{1}} }dx$ and hence equals
$-1/2(1+\log(2\pi\alpha_{1}))$.

As $A\geq 1$ and $B\leq A$, we deduce the upper bound,
\begin{align*}
\overline{F}(\theta)
 &\leq \sup_{\alpha \in [0,1]}\Big\{   \theta^2 \big( \alpha^2 +2A\alpha (1-\alpha) +A(1-\alpha)^2 \big) +\frac{1}{2} \log \alpha \Big\}\\
& =\sup_{\alpha \in [0,1]}\Big\{   \theta^2\big( A -(A-1) \alpha^2\big) +\frac{1}{2} \log \alpha \Big\}\,.
\end{align*}
Hence for all $\theta\ge 0$, (and as we could have seen directly from the uniform  upper bound $L(\theta)\le \frac{A}{2} \theta^{2}$)
\begin{equation}\label{upperboundF}
\overline{F}(\theta)\le A\theta^{2}\,.
\end{equation}
We see that if $2\theta\sqrt{A-1} \leq 1$ then the function 
$$ \alpha \mapsto \theta^2\big( A -(A-1) \alpha^2\big) +\frac{1}{2} \log \alpha,$$
is increasing on $[0,1]$. Thus
the supremum is achieved at $\alpha =1$, and $\overline{F}(\theta) \leq \theta^2$.  Moreover, taking $\alpha_{1}=1,\alpha_{2}=\alpha_{3}=0$, and $\nu_{1}$ the standard Gaussian restricted to $I_{1}$,
$\nu_{1}(dx)= \Car_{I_{1}}e^{-\frac{x^{2}}{2}} dx/Z$, we find that
\begin{equation}\label{lowerboundF} \underline{F}(\theta) \geq \theta^2.\end{equation}
Thus, if $2\theta \sqrt{A-1}\leq 1$, we get that $\overline{F}(\theta) = \underline{F}(\theta)=\theta^2$.

\end{proof}
Although the limit of the annealed spherical integrals may not be explicit for all $\theta$, we can still use it to obtain large deviations upper bounds as we describe now in the following theorem.

 \begin{theorem}\label{upperbound}Under  Assumptions \ref{AG} and \ref{ass},  the law of the largest eigenvalue $\lambda_{X_N}$ satisfies a large deviation upper bound with good rate function  $\bar I$ which is infinite on $(-\infty,2)$ and otherwise given by:
 \begin{equation} \label{defbarI} \forall y \geq 2, \ \bar{I}(y)=\sup_{\theta\ge 0} \{J( y,\theta)- \overline F(\theta)\}\,.\end{equation}
 Moreover, $\bar{I}(y)\le I_{GOE}(y)$ for all $y\ge 2$.
 \end{theorem}
 \begin{proof} 
From Remark \ref{tensionexporadius}, we know that the law of  the largest eigenvalue is exponentially tight at the scale $N$. Therefore,  it is sufficient to prove a weak large deviations upper bound by \cite[Lemma 1.2.18]{DZ}. Let $\delta>0$. We have,
$$ \Pp(\lambda_{X_N} <2-\delta ) \leq \Pp( \hat \mu_{X_N}(f) =0),$$
where $f$ is a smooth compactly supported function with support in $(2-\delta,2)$. Since $\mathrm{supp}(\sigma) = [-2,2]$, we deduce that,
$$ \Pp(\lambda_{X_N} <2-\delta ) \leq \Pp( d( \hat \mu_{X_N}, \sigma)>\eps),$$
for some $\eps>0$. As the empirical distribution of the eigenvalues concentrates at the scale $N$ according to \eqref{concspmeasure}, we conclude that
 $$ \lim_{N\to +\infty} \frac{1}{N} \log \Pp(\lambda_{X_N} <2-\delta ) =-\infty.$$
 Let now $x\geq 2$ and $\delta>0$. Recall from \eqref{lowerboundF} that  $\underline{F}(\theta) \geq \theta^2$ for any $\theta \geq 0$. Therefore,
$$ \bar{I}(x) \leq \sup_{\theta \geq  0}\{ J(x,\theta) - \theta^2\}.$$
From \cite[Section 4.1]{HuGu}, we know that 
$$\sup_{\theta \geq  0}\{ J(x,\theta) - \theta^2\}= I_{GOE}(x),$$
where $I_{GOE}$ is the rate function of the largest eigenvalue of a GOE matrix. Therefore we have proved that
$$\bar{I}(x)\le I_{GOE}(x),\forall x\ge 2\,.$$
 In particular   $\bar{I}(2) = 0$ since $I_{GOE}(2)=0$. 
Therefore we only need to estimate small ball probabilities around $x\neq 2$. As $\hat \mu_{X_N}$ concentrates at the scale $N$ by Assumption \ref{ass}, and $||X_N||$ is exponentially tight at the scale $N$ by Remark \ref{tensionexporadius}    it is enough to show that for any $K>0$,
$$ \limsup_{\delta \to 0} \limsup_{N\to +\infty} \frac{1}{N} \log \Pp(X_N \in V^K_{\delta,x}) \leq -\bar{I}(x),$$
where $V_{\delta,x}^K = \{  Y\in \mathcal H_N: |\lambda_{Y} -x|<\delta, d(\hat\mu_{Y},\sigma) <N^{-\kappa}, ||Y||\leq K \}$, for some $\kappa>0$.
Let $\theta\geq 0$. From \cite[Proposition 2.1]{Ma07}, we know that the spherical integral is continuous, more precisely, for $N$ large enough and any $X_N \in V_{\delta,x}^K$, 
$$ |J_N(X_N,\theta) - J(x,\theta)| < g(\delta),$$
for some function $g(\delta)$ going to $0$ as $\delta\to 0$.
Therefore,
$$\Pp(X_N \in V_{\delta,x}^K) = \E\Big( \Car_{X_N \in V_{\delta,x}^K} \frac{I_N(X_N,\theta)}{I_N(X_N,\theta)}\Big) \leq \E [I_N(X_N,\theta) ]e^{-NJ(x,\theta) +Ng(\delta)}.$$
Taking the limsup as $N \to 0$ and $\delta \to 0$ at the logarithmic scale, we deduce
$$ \limsup_{\delta \to 0} \limsup_{N\to +\infty} \frac{1}{N} \log \Pp(X_N \in V_{\delta,x}^K) \leq \overline{F}(\theta) - J(x,\theta).$$
Opimizing over $\theta\geq 0$, we get the claim.
\end{proof}

\begin{proposition} \label{proprate} Under Assumption \ref{AG},  the rate function $\bar{I}$ defined in Theorem  \ref{upperbound} is lower semi-continuous, and growing at infinity like $ x^2/4 A$. In particular, $\bar I$ is a good rate function.

\end{proposition}

\begin{proof}
$\bar I$ is lower semi-continuous as a supremum of continuous functions (recall here that $J(\theta,.)$ is continuous by Lemma \ref{calculJ} and $\overline {F}$ is continuous by Remark \ref{calculJ}). It remains to show that its level sets are compact, for which it is sufficient to  prove that $\bar{I}$ goes to infinity at infinity. 
Let $x>2$. Let $C>0$   be a constant to be chosen later such that $Cx \geq 1/2$. We have by taking $\theta=Cx$ and using \eqref{upperboundF}, that
 \begin{align}
\bar I(x)&\ge J(x,Cx)-\overline{F}(Cx)\nonumber \\
&\ge Cx^2 -\frac{1}{2} -\frac{1}{2} \log (2 Cx) - \frac{1}{2} \log x- AC^2 x^2.\label{lowerboundI}
\end{align}
Taking $ C = 1/2A$, and assuming that $x>A$, we obtain that 
\begin{align} \bar I(x) & \geq \frac{x^2}{4A} - o(x^2). \label{borneinfI}
\end{align}
 To get the converse bound, we show that as $\theta$ goes to infinity, $\overline{F}$ goes to infinity like $A\theta^{2}$. We distinguish two cases. First, we consider the case $A=B$. Using Proposition \ref{simpleFint}, we get the lower bound for $\theta \geq 1$,
$$ \underline{F}(\theta) \geq A\theta^2 \Big( 1 - \frac{1}{\theta^2}\Big) -\frac{1}{4} \log \theta ,$$ 
by taking $\alpha_2 =0$, $\alpha_3=1-\theta^{-2}$, $\alpha_1=\theta^{-2}$ and $\nu_1$ the Gaussian law restricted to $I_1$ with variance $\alpha_1$.
In the case $A>B$, we define $m_{*}$ such that  $\psi(m_{*}) = A/2$. Taking $\alpha_3 = 0$, $\alpha_2 = 1-\theta^{-2}$, $\alpha_1=\theta^{-2}$, 
 $t_i = \frac{\sqrt{m_{*}} N^{1/4}}{\sqrt{2 \theta}}$, $l =  \lfloor \frac{2 \theta  \alpha_2 \sqrt{N} }{m_{*}} \rfloor$ 
 , and $\nu_1$ the Gaussian law  {restricted to $I_{1}$ with variance $\alpha_1$}, we obtain,
\begin{equation} \label{lbFcompact} \underline{F}(\theta) \geq A\theta^2 \Big( 1 - \frac{1}{\theta^2}\Big) -\frac{1}{4}\log \theta.\end{equation}
It follows that for any $\eps>0$, there exists $M<\infty$ such that for $\theta \geq M$,
$$\underline{F}(\theta)\ge (1-\eps) A\theta^{2}.$$
Therefore
$${\bar I}(x)\le \max\left\{ \sup_{\theta \ge M}\{J(x,\theta)-(1-\eps)A\theta^{2}\}, \sup_{\theta\le M}\{J(x,\theta)-\overline{F}(\theta)\}\right\}\,.$$
But from Lemma \ref{calculJ} one can see that the second term in the above right-hand side is bounded by $Mx+C$ where $C$ is a numerical constant. Besides, using the same argument as in \eqref{lowerboundI}, we get
$$ \sup_{\theta \ge M}\{J(x,\theta)-(1-\eps)A\theta^{2}\}\geq \frac{x^2}{4(1-\eps)A} -o(x^2).$$
 Hence, for $x$ large enough, 
$$\bar I(x)\leq  \sup_{\theta \ge M}\{J(x,\theta)-(1-\eps)A\theta^{2}\}.$$
But, for $x$ large enough and $\theta\geq 1/2$, $J(\theta,x)\leq \theta x$. Thus,
$$ \sup_{\theta \ge M}\{J(x,\theta)-(1-\eps)A\theta^{2}\} \leq \sup_{\theta \ge 0}\{\theta x -(1-\eps)A\theta^{2}\} = \frac{x^2}{4(1-\eps)A},$$
which ends the proof.

%
 
 \end{proof}
 \begin{proposition} \label{convexrate} For any $\theta\geq 0$, $J(.,\theta)$ is a convex function. Therefore, $\bar I$ is also convex.\end{proposition}
 \begin{proof}

Let $x,y \geq 2 $ and $t\in (0,1)$. Let $E_N$ be a sequence of diagonal matrices such that $|| E_N|| \leq 2$ and such that $\hat \mu_{E_N}$ converges weakly to $\sigma$. Let $E_N^x$ and $E_N^y$ be such that $(E_N^x)_{i,i} = (E_N^y)_{i,i} = (E_N)_{i,i}$ for any $i\in \{1,\ldots, N-1\}$, and
$$ (E_N^x)_{N,N} = x \quad (E_N^y)_{N,N} = y.$$ 
We have  $\lambda_{E_N^x} = x$ and $\lambda_{E_N^y} = y$. 
Then, $H_N = tE_N^x + (1-t) E_N^y$ is such that its empirical distribution of eigenvalues converges to $\sigma$, and $\lambda_{H_N} = tx+(1-t)y$.
By Hölder's inequality we have,
$$\log I_N(H_N,\theta) \leq t\log I_N(E_N,\theta) +(1-t) \log I_N(D_N,\theta).$$ 
Taking the limit as $N \to +\infty$, we get,
$$ J(tx + (1-t)y,\theta) \leq tJ(x,\theta) + (1-t) J(y,\theta).$$
Therefore, $J(\theta,.)$ is convex and $I$ is convex as a supremum of convex functions.
 \end{proof}

 To derive the large deviation lower bound, we denote
by $\mathcal C_\mu$  the set of $\theta\in\mathbb R^+$ such that
$$\underline{F}(\theta) = \overline F(\theta)=:F(\theta)\,.$$
 By Lemma  \ref{F0},  $\mathcal C_\mu$ is not empty. We observe also that by continuity of both $\underline{F}$ and $\overline F$ (see Remark \ref{remF}), $\mathcal C_\mu$ is closed. 
Let 
$$\forall x \geq 2, \ I(x)=\sup_{\theta\in \mathcal C_\mu}\{ J(x,\theta)- {F}(\theta)\}.$$
\begin{theorem}\label{maintheo} For any $x\geq 2$, denote by 
$$\Theta_x = \{ \theta \ge 0 :\bar I(x) = J(x,\theta) - F(\theta) \},$$
where $\bar I$ is defined in \eqref{defbarI}.
Let $x\geq 2$ such that there exists $\theta \in \Theta_x\cap {\mathcal C}_\mu$ and $\theta \notin \Theta_y$ for any $y\neq x$. Then, $I(x) = \bar I(x)$ and
$$\lim_{\delta\ra 0}\liminf_{N\ra+\infty}\frac{1}{N}\log \Pp\left(|\lambda_{X_N}-x|\leq \delta\right)\ge -I(x).$$
\end{theorem}
We apply this general theorem in two cases.
We first investigate the case  where the function $\psi$ is increasing, case  for which we can check that our hypotheses on the sets $\Theta_{x}$ holds for $x$ large enough. This includes the case where $\mu$ is the sparse Gaussian law, see Example \ref{example}.
\begin{proposition}\label{increasingprop}
Suppose that Assumptions \ref{AG} and \ref{ass} hold. If $\psi$ is increasing on $\R_+$, then $\mathcal C_\mu=\mathbb R^{+}$. Moreover, there exists $x_\mu\geq 2$ such that for any $x\geq x_\mu$, the large deviation lower bound holds with rate function $I$. 
\end{proposition}
We then consider the case where $\mu$ is such that $B <A$. This includes any compactly supported measure $\mu$ since then $B=0$. We prove in this case the following result.
\begin{proposition}\label{compactprop}
Suppose that Assumptions \ref{AG} and \ref{ass} hold. If $\mu$ is such that $B < A$ and such that the maximum of $\psi$ is attained on $\R^+$ for a unique $m_{*}$ such that $\psi''(m_{*}) < 0$, then there exists a positive finite real number $\theta_0$ such that $[\theta_0,+\infty [\subset \mathcal C_\mu$.
Therefore, there exists a finite constant $x_\mu$ such that for $x\ge x_\mu$,
the large deviation lower bound holds with rate function $I$. Furthermore, on the interval $[x_\mu,+\infty)$ the rate function $I$ depends only on $A$.
\end{proposition}

In the case where $A$ is sufficiently small, we can show without any additional assumption that the large deviation lower bound holds in a vicinity of $2$ and   the rate function $I$ is equal to the one of the GOE. This contrasts with Proposition \ref{proprate} which shows that the rate function $\bar I$ goes to infinity like $x^2/4A$ at infinity and therefore depends on $A$. In other words the ``heavy tails'' only kicks in above a certain threshold.
\begin{proposition}\label{Propsmall}Assume $A <2$. The large deviation lower bound holds with rate function $\bar I$ on $ [2, 
 1/\sqrt{A-1} + \sqrt{A-1}]$. Moreover, $\bar I$ coincides on this interval with the rate function in the GOE case $I_{GOE}$, defined in \eqref{GOE}. As a consequence,
 for all $x\in [2, 1/\sqrt{A-1} + \sqrt{A-1}]$,
 $$\lim_{\delta\ra 0}\liminf_{N\ra+\infty}\frac{1}{N}\log \Pp\left(|\lambda_{X_N}-x|\leq \delta\right)=\lim_{\delta\ra 0}\limsup_{N\ra+\infty}\frac{1}{N}\log \Pp\left(|\lambda_{X_N}-x|\leq \delta\right)=-I_{GOE}(x).$$
\end{proposition}

\subsection*{Organization of the paper}
In the next section \ref{general}, we detail our approach  to prove large deviations lower bounds. Since Proposition  \ref{simpleFint} is crucial to all our results, we prove it in the next section \ref{asymptspheannealed}. Then, we will apply these results to prove the large deviations lower bounds close to the bulk in section \ref{closebulk}, that is, we give a proof of Proposition \ref{Propsmall}. To prove the large deviations lower bounds for large $x$, we consider first the case of increasing $\psi$ in section \ref{incsec} and then the case of $B < A$ in section \ref{seccompact}. Indeed, the variational formulas for the limiting annealed spherical integrals differ in these two cases, as $B=A$ in the first case whereas $B<A$ in the second.

\section{A general large deviation lower bound}\label{general}
 
We first prove  Theorem \ref{maintheo} and will then give more practical descriptions of the sets $\Theta_x$ in order to apply it. 
\begin{proof}[ Proof of  Theorem \ref{maintheo}] By assumption, there exists  $\theta \in \Theta_x\cap {\mathcal C_\mu}$ such that $\theta \notin \Theta_y$ for $y\neq x$. In particular, it entails that $I(x) = \bar I(x)$.
Introducing the spherical integral with parameter $\theta \geq 0$, we have 
$$ \Pp\left(|\lambda_{X_N}-x|\leq \delta\right) \geq \E \Big(\Car_{X_N \in V_{\delta,x}^K} \frac{I_N(X_N,\theta)}{I_N(X_N,\theta)}\Big),$$
where $V_{\delta,x}^K = \{ Y\in \mathcal H_{N}:|\lambda_{Y}-x|\leq \delta, d(\hat \mu_{Y}, \sigma )<N^{-\kappa}, ||X_N||\leq K\}$ for some $K>0$ and $\kappa>0$.
Using the continuity of the spherical integral (see \cite[Proposition 2.1]{Ma07}), we get
\begin{equation}\label{lbtilt}\Pp\left(|\lambda_{X_N}-x|\leq \delta\right) \geq   \frac{ \E \big(\Car_{X_N \in V_{\delta,x}^K} I_N(X_N,\theta))}{\E I_N(X_N,\theta)}e^{N \underline{F}(\theta)-NJ(x,\theta) -Ng(\delta)- o(N)},\end{equation}
where $g$ is a function such that $g(\delta) \to 0$ as $\delta \to 0$. We claim that
$$\liminf_{N \to +\infty}\frac{1}{N}\log \frac{ \E \big(\Car_{X_N \in V_{\delta,x}^K} I_N(X_N,\theta_x))}{\E I_N(X_N,\theta_x)} \geq 0.$$
To this end we will use our large deviation upper bound. {Since $\hat \mu_{X_N}$ concentrates at  scales faster than $N$ by Assumption \ref{ass}, and by Remark \ref{tensionexporadius}, $||X_N||$ and $\hat\mu_{X_{N}}$ are exponentially tight uniformly under any tilted measures $\Pp^{(e,\theta)}$, as defined in \eqref{deftilt}, and therefore under the measure tilted by spherical integrals}. Hence, 
 it suffices to prove that for all $y\neq x$, for $\delta$ small enough, and $K$ large enough,
$$\limsup_{N}\frac{1}{N}\log \frac{\E[ \Car_{X_N \in V_{\delta,y}^K }I_N(X_N,\theta)] }{\E I_N(X_N,\theta)}<0.$$
By assumption, there exists  $\theta \in \Theta_x\cap {\mathcal C_\mu}$ such that $\theta \notin \Theta_y$ for $y\neq x$.
We introduce a new spherical integral with argument $\theta'$ and use again the continuity of $J_N$ to show that:
\begin{eqnarray*}
\frac{\mathbb E[ \Car_{X_N \in V_{\delta,y}^K} I_N(X_N,\theta)]}{\E I_N(X_N,\theta) }&=&
\frac{\mathbb E[ \Car_{X_N \in V_{\delta,y}^K}\frac{ I_N(X_N,\theta')}{ I_N(X_N,\theta')} I_N(X_N,\theta)]}{\E I_N(X_N,\theta)}\\
&\leq& e^{-N J(y,\theta')-N\underline{F}(\theta)+NJ(y,\theta)+N\eps(\delta)}\mathbb E[ 1_{X_N \in V_{\delta,y}^K} I_N(X_N,\theta') ]\\
&\le&e^{-N J(y,\theta')-N\underline{F}(\theta)+NJ(y,\theta)+N\overline{F}(\theta')+N\eps(\delta)},\\
\end{eqnarray*}
where $\eps(\delta) \to 0$ as $\delta \to 0$. We can conclude that
\begin{eqnarray}
\limsup_{\underset{\delta \to 0}{N\to +\infty}}\frac{1}{N}\log \frac{\mathbb E[ \Car_{X_N \in V_{\delta,y}^K} I_N(X_N,\theta)]}{\E I_N(X_N,\theta)} &\le& -\sup_{\theta'}\{J(y,\theta')-\overline{F}(\theta')\}+J(y,\theta)-\underline{F}(\theta)
\nonumber\\
&=&-\overline{I}(y)+J(y,\theta)-\underline{F}(\theta)\label{kl}
\end{eqnarray}
 By assumption, $\theta \notin \Theta_y$, and $\theta\in\mathcal C_\mu$ so that $\underline F(\theta)=\overline F(\theta)$.  Hence,
$$-\overline{I}(y)+J(y,\theta)-\overline{F}(\theta)<0$$
and the conclusion follows from \eqref{kl}. Therefore, coming back to \eqref{lbtilt}, we obtain since $\theta_x \in \Theta_x $ and $I(x) = \bar I(x)$,
$$ \liminf_{N\to+\infty} \frac{1}{N}\log \Pp\left(|\lambda_{X_N}-x|\leq \delta\right) \geq - I(x).$$
\end{proof}

In a first step, we  identify a subset defined in terms of the subdifferential sets of $F$ at the points of non-differentiability where the large deviation lower bound holds. 
 Let $\mathcal{D}$ be the set of $\theta\ge 0$ such that $\overline{F}$ is differentiable at $\theta$.
\begin{lemma}\label{lowerb}
The lower bound holds for any $x>2$ such that $I(x)=\bar I(x)>0$ and 
\begin{equation}\label{defE} x \notin E:=\bigcup_{\theta \in \mathcal D^c} \Big( \frac{1}{2\theta} +\partial \overline F(\theta)\Big),\end{equation}
where $\partial \overline F(\theta)$ denotes the subdifferential of $\overline{F}$ at $\theta$.
\end{lemma} 
Note that since $\overline{F}$ is a convex function, its subdifferentials are well defined. Moreover, by Lemma \ref{F0}, $\mathcal D^{c}\subset [\frac{1}{2\sqrt{A-1}},+\infty)$.
\begin{proof}Let $x>2$ such that $I(x) = \bar I(x) >0$ and $x\notin E$. Since $\overline{F}(\theta) \geq \theta^2$ for any $\theta \geq 0$ by \eqref{lowerboundF} and $\overline{F}$ is continuous by Remark \ref{remF}, we deduce from Lemma \ref{calculJ} that $\theta\mapsto J(x,\theta) - \overline{F}(\theta)$ is continuous and goes to  $ - \infty$ as $\theta $ goes to $+\infty$. Since $\mathcal{C}_\mu$ is closed, the supremum 
\begin{equation}\label{hypox}  \sup_{\theta \in \mathcal{C}_\mu} \big\{ J(\theta,x) - \overline F(\theta)\big\}\end{equation}
is achieved at some $\theta \in \mathcal{C}_\mu$. We will show that 
$\theta \in \mathcal{D}$.

As $\bar I(x) \neq 0$ we must have $\theta > \frac{1}{2} G_\sigma(x)$. Indeed, $\overline{F}(\theta) \geq \theta^2$ and $J(x,\theta) = \theta^2$ for $\theta \leq \frac{1}{2} G_\sigma(x)$ by Lemma \ref{calculJ} so that
$$ \sup_{ \theta \leq \frac{1}{2} G_\sigma(x)} \{ J(x,\theta) - \overline{F}(\theta)\} =0.$$
Since $\bar I(x) = I(x)$, we deduce by Fermat's rule that $\theta$ is a critical point of $J(x,.)-\overline{F}$ and therefore satisfies the condition:
$$ 0 \in \frac{\partial J}{\partial \theta}(x,\theta) - \partial \overline{F}(\theta) = x - \frac{1}{2\theta} - \partial \overline{F}(\theta)\,.$$
Since $x \notin E$, we deduce that $\overline{F}$ is differentiable at $\theta$.

 According to Theorem \ref{maintheo}, to prove that the lower bound holds at $x$, it suffices to show that $\theta \notin \Theta_y$ for any $y\neq x$. Let us proceed by contradiction and assume that there exists  $y \geq 2$, $y\neq x$,  such that $\theta \in \Theta_y$. As $\overline{F}$ is differentiable at $\theta$, it should be a critical point of both $J(y,.)-\overline{F}$ and $J(x,.)-\overline{F}$.  Therefore, we should have
$$ \frac{\partial}{\partial \theta} J(y,\theta) = \frac{\partial}{\partial \theta} J(x,\theta) .$$
If $G_\sigma(y) < 2\theta$, then we obtain by Lemma \ref{calculJ} and the fact that $G_\sigma(x) <2\theta$ that
 $x=y$. If $G_\sigma (y) \geq 2\theta$, then we have
$$ x - \frac{1}{2\theta} = 2\theta.$$
On the other hand, $2\theta\le G_{\sigma}(y)\le 1$ and therefore 
we get  the unique solution $ 2\theta=G_\sigma (x) $. As we assumed that $2\theta > G_\sigma (x)$, we get a contradiction and  conclude that $\theta \notin \Theta_y$ for any $y\neq x$ such that $G_\sigma (y) \geq 2\theta$,  which completes the proof.
\end{proof}
We are now ready to prove the following result:
\begin{proposition}\label{lb}
Assume that there exists $\theta_0>0$ such that $[\theta_0,+\infty)\subset \mathcal C_\mu$ and such that $\overline{F}$ is differentiable on $(\theta_0,+\infty)$. There exists $x_\mu\in [2,+\infty)$ such that for any $x\geq x_\mu$, $I(x) = \bar I(x)$ and the large deviation lower bound holds for any $x \geq x_\mu$ with rate function $I(x)$.
\end{proposition}
\begin{proof} On one hand,
\begin{equation}\label{bornetheta0} \sup_{\theta \le \theta_0} \{J(\theta, x)-\overline{F}(\theta)\}\le \theta_0 x +C,\end{equation}
where $C$ is some positive constant. Since $\bar I(x) \geq x^2/4A -o(x^2)$ by \eqref{borneinfI}, we deduce that there exists $x_\mu\in [2,+\infty)$ such that for $x\geq x_\mu$, $\bar I(x) >0$ and together with \eqref{bornetheta0} that  the supremum of $J(.,x)-\overline{F}$ is achieved in $[\theta_0,+\infty)$, and therefore in $\mathcal{C}_\mu$ by our assumption. By definition of $\mathcal C_{\mu}$,  we deduce that for any $x\geq x_\mu$, $\bar I(x) = I(x)>0$. In view of Lemma \ref{lowerb}, it remains to show that $E$, defined in \eqref{defE}, is a bounded set.
From our assumption  that $\overline{F}$ is differentiable on $(\theta_0,+\infty)$ and Lemma \ref{F0}, we deduce that
$$ \mathcal{D}^c \subset \Big[ \frac{1}{2\sqrt{A-1}}, \theta_0\Big].$$
We observe that since $0\leq \overline{F}(\theta) \leq A\theta^2$, we have for any $\zeta \in \partial \overline{F}(\theta)$,
$$\zeta \theta \leq  \overline{F}(2\theta) - \overline{F}(\theta) \leq 4A\theta^2,$$
and thus $\zeta \leq 4A\theta$. Therefore, the set $E$ defined in Lemma \ref{lowerb} is  bounded, which ends the proof.

\end{proof}

%
%
%

 \section{Asymptotics of the annealed spherical integral}\label{asymptspheannealed}
 In this section we prove Proposition \ref{simpleFint}. 
Taking the expectation first with respect to $X_N$, the annealed spherical integral is given by
 \begin{eqnarray*}
  F_N( \theta) &=& \frac{1}{N} \log \E_{X_{N}} \E_{e}[ \exp(N \theta \langle e, X_{N} e \rangle) ] \\
  &=& \frac{1}{N} \log \E_e \exp\big( f(e)\big),\end{eqnarray*}
where
$$f(e) = \sum_{i<j} L\big( 2\sqrt{N} \theta e_i e_j\big) + \sum_{i=1}^N L(\sqrt{2N}  \theta e_i^2\big).$$
In a first step, we will prove the following variational representation of the upper and lower limits of $F_N(\theta)$.

\begin{lemma}\label{WigGen2} Let $X_N$ be a Wigner matrix satisfying Assumptions \ref{AG} and \ref{ass}.
Then for any $\theta>0$,
$$\underline F(\theta)\le\liminf_{N\ra+\infty} F_N(\theta)\le \limsup_{N\ra+\infty}  F_N( \theta) \le \overline F(\theta)$$
with
$$\overline{F}(\theta)= \limsup_{ \delta \to 0, K\to +\infty \atop \delta K \to 0}  \sup_{ c=  c_1 + c_2 +c_3 \atop c_i \geq 0 } \limsup_{N\ra+\infty} F^N_{c_1,c_2,c_3}(\delta, K),$$
$$\underline{F}(\theta)= \sup_{c = c_1 + c_2 +c_3\atop c_i \geq 0 } \liminf_{ \delta \to 0, K\to+\infty \atop \delta K \to 0}  \ \liminf_{N\ra+\infty} F^N_{c_1,c_2,c_3}(\delta,K),$$
where
\begin{align*} F^N_{c_1,c_2,c_3}(\delta,K) & = \sup_{s_i\ge \sqrt{c K} N^{1/4}  \atop
|\sum s_i^2- c_3N|\leq \delta N} \sup_{\sqrt{c\delta} \le t_i N^{-1/4}\le \sqrt{cK}   \atop |\sum t_i^2-N c_2|\leq \delta N } \Big\{  \frac{\theta^2}{c^2}\big( c_1^2 +2c_1 c_2 +Bc_3^2 \big) - \frac{1}{2}\big( c_2+c_3\big) \\
&+ \frac{1}{N}\sum_{i,j} L\Big(\frac{2 \theta  s_i t_j}{\sqrt{N}c}\Big)  
 +\frac{1}{2N}\sum_{i, j} L\Big(\frac{ 2\theta  t_i t_j}{\sqrt{N}c}\Big) + \sup_{\nu \in \mathcal{P}(I_1) \atop \int x^2 d\nu(x) = c_1 }\big\{ \Phi(\nu, s) -H(\nu | \gamma) \big\}\Big\},
\end{align*}
and
$$ \Phi(\nu,s) =  \sum_{i}\int L\Big(\frac{ 2\theta x s_{i}}{\sqrt{N}c}\Big) d\nu(x),$$
  with $I_1 = [-\sqrt{c \delta} N^{1/4},\sqrt{c\delta} N^{1/4}]$. Here, we have set $\gamma$ to be the standard Gaussian law and
  $$H(\nu | \gamma)=\int\log \frac{d\nu}{d\gamma}(x) d\nu(x)\,.$$
\end{lemma}

{Hereafter, $o_{\delta}(1)$ is a function which goes to zero as $\delta$ goes to zero (or infinity depending on the context).  $\varepsilon_{K}(\delta)$ denotes a function which goes to zero as $\delta$ goes to zero or infinity while $K$ is fixed.
 $O(\delta)$ is a function such that there exists a finite constant $C$ such that the modulus of $O(\delta)$ is bounded by $C\delta$. These functions may change from line to line.}

\begin{proof}
We use the representation of the law of the vector $e$ uniformly distributed on the sphere
as a renormalized Gaussian vector $g/\|g\|_{2}$ where $g$ is a standard Gaussian vector in $\R^N$, to write
$$\E_e \exp\big(f(e)\big)
=\E \left [\exp (\Sigma(g))\right ],$$
where  $g=(g_{1},\ldots,g_{N})$ and 
$$\Sigma(g) =  \sum_{ i<j} L\Big(  2 \sqrt{N}\theta \frac{g_i g_j}{\sum_{i=1}^N g_i^2}\Big) + \sum_i L\Big(\sqrt{2N}  \theta \frac{g_i^2}{\sum_{i=1}^N g_i^2}\Big).$$
To study the large deviation of $\Sigma(g)$, we split the entries of $g$ into three possible regime: the regime where $g_i \ll N^{1/4}$, an intermediate regime where $g_i \simeq N^{1/4}$ and finally $g_i \gg N^{1/4}$. 
Fix some $K\geq 1,\delta>0$ such that $0 < 2\delta <K^{-1}$.  Let $c_1, c_2,c_3>0$ and $c = c_1+c_2+c_3$. We assume that ${0<K^{-1} \leq c_1\leq c \leq K}$. Define $I_1,I_2,I_3$ as 
\begin{eqnarray*}
I_{1}&=&[0, \sqrt{c\delta} N^{\frac{1}{4}}]\\
I_{2}&=&(  \sqrt{c\delta} N^{\frac{1}{4}}, \sqrt{cK }N^{\frac{1}{4}}]\\
 I_{3}&=& ( \sqrt{cK }N^{\frac{1}{4}}, \sqrt{N(c+3\delta)} ].\end{eqnarray*}
Let  for $i=1,2,3$, $J_{i}=\{ j : |g_{j}|\in I_{i}\}$ and  $\hat c^N_i = \sum_{j\in J_i} g_j^2 / N $. In a first step, we will fix the empirical variances $\hat c^N_i$ and compute the asymptotics of 
\[ {F}^N_{c_{1},c_{2},c_{3}}( \theta,\delta) = \E[ \exp(\Sigma(g)) \mathds{1}_{\mathcal{A}_{c_1,c_2,c_3}^{\delta}}  ].\]
where
\[ \mathcal{A}_{c_1,c_2,c_3}^{\delta} := \bigcap_{1\le i\le 3}\{ |\hat c^N_i - c_i| \leq \delta\}. \]
Let
$$\Sigma_c(g) =  \sum_{ i<j} L\Big(  2 \theta \frac{g_i g_j}{\sqrt{N} c}\Big) + \sum_i L\Big(\sqrt{2}  \theta \frac{g_i^2}{\sqrt{N}c}\Big).$$
Using the fact the $L(\sqrt{.})$ is Lipschitz, we prove in the next lemma that on the event $\mathcal{A}_{c_1,c_2,c_3}^{\delta}$, $\Sigma_c(g)$ is a good approximation of $\Sigma(g)$.
 
\begin{lemma}\label{fixnorm}   On the event $\mathcal{A}_{c_1,c_2,c_3}^{\delta} $, 
\begin{equation} \label{errnorm} \Sigma(g) - \Sigma_c(g) = N o_{\delta K}(1), \quad \text{as } \delta K \to0.\end{equation}
 Moreover
$$|J_3|\le \frac{3 \sqrt{N}}{K},\qquad |J_2\cup J_3|\le \frac{3\sqrt{N}}{\delta}.$$

\end{lemma}
\begin{proof}Note that since $\mu$ is symmetric, $L(x)=L(|x|)$ and 
since we assumed  $L(\sqrt{.})$ Lipschitz, for any $x,y\in\R$, $|L(x)-L(y)|\le L|x^{2}-y^{2}|$ for some finite constant $L$. Therefore,
\begin{align*} \sum_{1 \leq i \neq j \leq N}\big| L\Big( 2 \sqrt{N} \theta \frac{g_i g_j}{\sum_{i=1}^N g_i^2}) - L\Big( 2\theta \frac{g_i g_j}{\sqrt{N}c }\Big)\big| &
\le \frac{4L\theta^{2}}{N}\sum_{i,j} g_i^2 g_j^2  \left| \frac{1}{(\hat c_1^N+\hat c_2^N+\hat c_3^N)^2} -\frac{1}{c^2}\right|\\
& \le  C{N L \theta^2 (c+\hat c_1^N+\hat c_2^N+\hat c_3^N) {\frac{\delta}{c^2}}\le C' N L \theta^2 \delta K},
\end{align*}
where $C,C'$ are numerical constants and we used ${K^{-1} < c < K}$, and $2\delta < K^{-1}$.  We get a similar estimate for the diagonal terms.
The estimates on $|J_3|$ and $|J_2|$ are straightforward consequences of  Tchebychev's inequality. 

\end{proof}

We next fix the set of indices $J_1,J_2,J_3$. Using the invariance under permutation of the distribution of $g$, we can write
$${F}^N_{c_{1},c_{2},c_{3}}( \theta,\delta) =\sum_{0\le k\le 3\sqrt{N\delta} \atop 0\le l\le 3\sqrt{NK }} \binom{N}{k}\binom{N-k}{l} {F}_{c_1,c_2,c_3}^{k,l},$$
where
$$ { F}_{c_1,c_2,c_3}^{k,l}=
\E\Big[ \exp (\Sigma_c(g) ) \mathds{1}_{\mathcal{A}_{c_1,c_2,c_3}^{\delta}\cap \mathcal{I}_{k,l}}\Big ]$$
and
$$\mathcal{I}_{k,l} = \big\{J_{3}=\{1,\ldots ,k\}, J_{2}=\{k+1,\ldots, k+l\}, J_{1}=\{k+l+1,\ldots, N\} \big\}.$$
As the number of all the possible configurations of $J_2$ and $J_3$ are sub-exponential in $N$ by Lemma \ref{fixnorm}, that is, for any $k\leq 3\sqrt{N}/K$ and $l\leq 3\sqrt{N}/ \delta$,
$$ \max \Big(\binom{N}{k},\binom{N-k}{l}\Big) = e^{O(\frac{\sqrt{N}}{\delta} \log N)},$$
we are reduced to compute $F_{c_1,c_2,c_3}^{k,l}$ for fixed $k,l$. More precisely, we have the following result.

\begin{lemma}
\label{fixconfig}
$$\log {F}^N_{c_{1},c_{2},c_{3}}( \theta,\delta) = \max_{ k \leq 3\sqrt{N}/K \atop l \leq 3\sqrt{N}/\delta}  \log{F}_{c_1,c_2,c_3}^{k,l} + O\Big(\frac{\sqrt{N}}{\delta} \log N\Big).$$
\end{lemma}
To simplify the notations, we denote for any $a,b \in \{1,2,3\}$,
\[\forall x, y\in \R^{N}, \  \Sigma_{a,b}(x,y) = \frac{1}{2N}\sum_{i \in J_a,j \in J_b }  L\left( 2 \theta \frac{x_i y_j}{\sqrt{N}c } \right),\]
if $a\neq b$, and 
\[\forall x, y\in \R^{N}, \  \Sigma_{a,a}(x,y) = \frac{1}{2N}\sum_{i\neq j \in J_a }  L\left( 2 \theta \frac{x_i y_j}{\sqrt{N}c } \right) + \frac{1}{N}\sum_{i\in J_a} L\Big( \sqrt{2} \theta \frac{x_i y_i}{\sqrt{N}c }\Big),\]
where now $J_{3}=\{1,\ldots ,k\}, J_{2}=\{k+1,\ldots, k+l\}, J_{1}=\{k+l+1,\ldots, N\}$.

Next, we single out the interaction terms which involves the quadratic behavior of $L$ at $0$ or at $+\infty$. 

\begin{lemma}\label{expan}
On the event $ \mathcal{A}_{c_1,c_2,c_3}^{\delta}$, 
$$ \Sigma_{1,1}(g,g) +2\Sigma_{1,2}(g,g) +\Sigma_{3,3}(g,g) = \frac{\theta^2}{c^2}\big( c_1^2 +2c_1 c_2 +Bc_3^2 \big) +o_{\delta K}(1) + o_K(1),$$\end{lemma}
as $\delta K \to 0$ and $K\to +\infty$.
\begin{proof}
 Observe that for $i\in J_{1}$, $j\in J_{1}\cup J_{2}$, $|g_{i}g_{j}|\le \sqrt{NK\delta}c$. Since $L(x) \sim_0 x^2/2$, we get,
\begin{equation} \label{sigma11}\Sigma_{1,1}(g,g)+2\Sigma_{1,2}(g,g)=\theta^{2}\frac{(\hat c^{N}_{1})^{2}}{c^2}
+2\theta^{2}\frac{{\hat c}^N_{1}{\hat c}^N_{2}}{c^{2}}+o_{\delta K}(1),\end{equation}
as $\delta K \to 0$.
On the event $\mathcal{A}_{c_1,c_2,c_3}^{\delta}$, we have $|(\hat c^{N}_i)^2 - c_i^2| = O( \delta c)$ for any $i\in\{1,2,3\}$. But $c\geq K^{-1}$, therefore 
$$ \Sigma_{1,1}(g,g)+2\Sigma_{1,2}(g,g)=\theta^{2}\frac{ c_{1}^{2}}{c^2}
+2\theta^{2}\frac{c_{1} c_{2}}{c^{2}}+o_{\delta K}(1).$$
For $i,j \in J_3$,  $|g_{i}g_{j}|\geq K\sqrt{cN}$. Since $L(x)\sim_{ +\infty} \frac{B}{2} x^{2}$, we deduce similarly that
\begin{equation} \label{sigma33}\Sigma_{3,3}(g,g)= \Big(\frac{B}{2}+o_{K}(1)\Big) \frac{2\theta^{2}} {N^2 c^2} \Big(\sum_{i \in J_3} g_i^2\Big)^2 = B\theta^2\frac{c_3^2}{c^2}+o_K(1),\end{equation}
as $K\to +\infty$,
which gives the claim.

\end{proof}

From the Lemmas \ref{fixnorm} and \ref{expan}, we have on the event $\mathcal{A}_{c_1,c_2,c_3}^{\delta}$,
\begin{equation} \label{decomp} \Sigma(g)  = \frac{\theta^2}{c^2}\big( c_1^2 +2c_1 c_2 +Bc_3^2 \big) + \Sigma_{1,3}(g,g) + 2\Sigma_{2,3}(g,g) + \Sigma_{2,2}(g,g)+o_{\delta K}(1) + o_K(1).\end{equation}
{We now show that the deviations of the variables $g_i, i \in J_2\cup J_3$ do not lead to any entropic terms, which yields the following lemma.}  
\begin{lemma}\label{explem}{Let $k, l\in \N$ such that $k+l \leq N$. Define 
$$ S_{c_1,c_2,c_3}(\delta) = \max_{ t_i \in I_2, i\leq l \atop  |\sum_i t_i^2 -c_2N|\leq \delta N}  \max_{s_i \in I_3, i\leq k \atop |\sum_i s_i^2 -c_3N|\leq \delta N  } \log  \E\Big( \exp \Big\{{N}(\Sigma_{1,3}(g,s) +2\Sigma_{2,3}(t,s) + \Sigma_{2,2}(t,t) )\Big\} \mathds{1}_{\mathcal{A}_{c_1}^\delta } \Big),$$
where
\begin{equation} \label{defAc1} \mathcal{A}_{c_1}^\delta = \big\{ (g_{i})_{i=k+l+1}^{N}\in I_{1}^{N-k-l}:|  \sum_{i= k+l+1 }^{N}  g_i^2 - N c_1 | \leq N \delta\big\}.\end{equation}
Then,
\begin{align*}\label{tot}
S_{c_1,c_2,c_3}(\delta/2)-\frac{N}{2} (c_2+ c_{3}) - &o_{\delta }(1)N +O\Big(\frac{\sqrt{N}}{\delta} \log \frac{\delta}{\sqrt{NK}}\Big)\\
\leq \log \E\Big( \exp & \Big\{{N}(\Sigma_{1,3}(g,g) + 2\Sigma_{2,3}(g,g) + \Sigma_{2,2}(g,g)) \Big\}\Car_{\mathcal{A}_{c_1,c_2c_3}^{\delta} \cap \mathcal{I}_{k,l}}\Big) \\
&\leq S_{c_1,c_2,c_3}(\delta)-\frac{N}{2} (c_2+ c_{3}) + o_{\delta}(1) N
+ O\Big(\frac{\sqrt{N}}{\delta^2}\log(1/\delta)\Big).
\end{align*}}

\end{lemma}

\begin{proof}
Integrating on $g_i, i\leq k+l$, we find
\begin{align*}
&\E\Big( \exp \Big\{N(\Sigma_{1,3}(g,g) + 2\Sigma_{2,3}(g,g) + \Sigma_{2,2}(g,g) )\Big\}\Car_{\mathcal{A}_{c_1,c_2c_3}^{\delta} \cap \mathcal{I}_{k,l}}\Big) \\
 &\leq \frac{1}{(2\pi)^{\frac{k+l}{2}}} \int \Car_{\{ |\frac{1}{N}\sum_{i=1}^{k+l} g_i^2 -(c_2+c_3) | \leq 2\delta \} }  e^{- \frac{1}{2} \sum_{i=1}^{k+l} g_i^2} \prod_{i=1}^{k+l} d g_i \\
&\times    \max_{ t \in I_2^l \atop | \sum_i t_i^2 - c_2N | \leq \delta N}  \max_{s \in I_3^k \atop |\sum_i s_i^2 - c_3N | \leq \delta N }\E\Big( \exp \Big\{N(\Sigma_{1,3}(g,s) +2\Sigma_{2,3}(t,s) + \Sigma_{2,2}(t,t)) \Big\}\mathds{1}_{\mathcal{A}_{c_1}^\delta } \Big),
\end{align*}
But, using the fact that $c\le K$, we get
\begin{align*}
 \int \Car_{\{ |\frac{1}{N}\sum_{i=1}^{k+l} g_i^2 -(c_2+c_3) | \leq 2\delta \} }  e^{- \frac{1}{2} \sum_{i=1}^{k+l} g_i^2} \prod_{i=1}^{k+l} d g_i
&  \quad\leq e^{- \frac{1}{2} (1-\delta)(c_2+c_3 - 2\delta) N} \int  e^{- \frac{\delta}{2} \sum_{i=1}^{k+l} g_i^2} \prod_{i=1}^{k+l} d g_i\\
&\quad \leq  (2\pi)^{\frac{k+l}{2} }e^{- \frac{N}{2}(c_2+c_3) +(O(\delta) + O(\delta K)) N}\delta^{-\frac{k+l}{2} }.
\end{align*}
By Lemma \ref{fixnorm} we have $k+l = O(\sqrt{N}/\delta)$.Therefore,
\begin{align*}
&\log\E \Big( \exp \Big\{ N(\Sigma_{1,3}(g,g) + 2\Sigma_{2,3}(g,g) + \Sigma_{2,2}(g,g) )\Big\}\Car_{\mathcal{A}_{c_1,c_2,c_3}^{\delta} \cap \mathcal{I}_{k,l}}\Big)\\
&\quad\leq-\frac{N}{2} (c_2+ c_{3}) + O(\delta )N +O(\delta K) N
+ O\Big(\frac{\sqrt{N}}{\delta}\log(1/\delta)\Big)
 \nonumber\\
&\quad+ \max_{ t_i \in I_2, i\leq l \atop  |\sum_i t_i^2 -c_2N|\leq \delta N}  \max_{s_i \in I_3, i\leq k \atop |\sum_i s_i^2 -c_3N |\leq \delta N } \log \E\Big( \exp \Big\{N(\Sigma_{1,3}(g,s) +2\Sigma_{2,3}(t,s) + \Sigma_{2,2}(t,t) )\Big\} \mathds{1}_{\mathcal{A}_{c_1}^\delta } \Big).
\end{align*}

To get the converse bound, we take $t\in  I^{l}_{2}, s\in I_{3}^{k}$, {optimizing the  above maximum where $\delta$ is replaced by $\delta/2$. 
We next  localize the integral on the set $\mathcal B_{\delta}$ where $|g_{i}-s_{i}|\le \delta/8\sqrt{NK}$, $1\le i\le k$, $|g_{i}-t_{i}|\le \delta/4$, $k+1 \le i\le k+l$. One can check that $\mathcal B_{\delta} \times \mathcal A^{{\delta}}_{c_{1}}\subset \mathcal{A}_{c_1,c_2c_3}^{\delta} \cap \mathcal{I}_{k,l}$}, and because $L\circ {\sqrt{.}}$ is Lipschitz, we have on this event
$$|\Sigma_{1,3}(g,s) +2\Sigma_{2,3}(t,s) + \Sigma_{2,2}(t,t)-(\Sigma_{1,3}(g,g) +2\Sigma_{2,3}(g,g) + \Sigma_{2,2}(g,g))|\le C\theta^{2}\delta,
$$
where $C$ is some positive constant.
Hence
\begin{eqnarray*}
&&\E\Big( \exp \Big\{N(\Sigma_{1,3}(g,g) + 2\Sigma_{2,3}(g,g) + \Sigma_{2,2}(g,g) )\Big\}\Car_{\mathcal{A}_{c_1,c_2c_3}^{\delta} \cap \mathcal{I}_{k,l}}\Big) \nonumber\\
&&\geq \frac{1}{(2\pi)^{\frac{k+l}{2}}}\prod_{1\le i\le k}\int_{g \in I_3 }\Car_{|g - s_i|\leq \delta/8\sqrt{NK}}e^{-\frac{1}{2}g^{2}} dg\prod_{1\le i\le l}\int_{g \in I_2} \Car_{|g-t_j| \leq \delta/4}e^{-\frac{1}{2}g^{2}} dg 
 \nonumber\\
&&\times  e^{-C N\theta^{2}\delta}   \E\Big( \exp \Big\{N(\Sigma_{1,3}(g,s) +2\Sigma_{2,3}(t,s) + \Sigma_{2,2}(t,t)) \Big\}\mathds{1}_{\mathcal{A}_{c_1}^\delta } \Big),\nonumber\\
&&\geq  e^{-\frac{N}{2} (c_2+ c_{3}) + O(\delta )N }\Big(\frac{\delta}{8\sqrt{2\pi NK}}\Big)^{k}\Big(\frac{\delta}{4\sqrt{2\pi}}\Big)^{l}
 \E\Big( \exp \Big\{N(\Sigma_{1,3}(g,s) +2\Sigma_{2,3}(t,s) + \Sigma_{2,2}(t,t)) \Big\}\mathds{1}_{\mathcal{A}_{c_1}^\delta } \Big)
\end{eqnarray*}
which completes the proof of Lemma \ref{expan} as $k+l= O(\sqrt{N}/\delta)$ and $K^{-1} \geq \delta$.

\end{proof}

Hence, we are left to estimate
$$\Lambda_{1}^{N}=\mathbb E\big[ \exp\big( N \Sigma_{1,3}(g,s)\big) \Car_{\mathcal{A}_{c_1}^{\delta}}\big],$$
where $s \in I_3^k$ satisfies $|\sum_i s_i^2-N c_{3}| \leq \delta N$ and $\mathcal{A}_{c_1}^\delta$ is defined in \eqref{defAc1}. Let $\hat \mu_{N}=\frac{1}{|J_{1}|}\sum_{i \in J_1}\delta_{g_{i}}$.
We can write 
$$\Sigma_{1,3}(g,s)=\frac{|J_{1}|}{N}\sum_{i=1}^{k}\int L\left(\frac{2 \theta x s_{i}}{\sqrt{N}c}\right) d\hat\mu_{N}(x).$$
The first difficulty in estimating $\Lambda_1^N$ lies in the fact that the function
$$ x \mapsto \sum_{i=1}^k L\Big( \frac{2\theta x s_i}{\sqrt{N} c} \Big),$$
is not bounded so that Varadhan's lemma (see \cite[Theorem 4.3.1]{DZ}) cannot be applied directly. The second issue is that we need a large deviation estimate which is uniform in the choice of $s \in I_3^k$ such that $|\sum_{i=1}^k s_i^2 -N c_3|\leq \delta N$. In the next lemma, we prove a uniform large deviation estimate of the type of Varadhan's lemma. The proof is postponed to the appendix \ref{proofVaradhan}.

\begin{lemma}
\label{Varadhan}
Let $f : \R \to \R$ such that $f(0) = 0$ and $f(\sqrt{ . })$ is a $L$-Lipschitz function. Let $M_N, m_N$ be sequences  such that $M_N = o(\sqrt{N})$ and $m_N\sim N$. Let $g_1,\ldots,g_{m_N}$ be independent Gaussian random variables conditioned to belong to  $[-M_N,M_N]$. Let $\delta\in(0,1)$ and $c>0$ such that $K^{-1} < c <K$ and $2\delta <K^{-1}$. Then,
\begin{align*}\Big|   \frac{1}{N} \log \E e^{\sum_{i=1}^{m_N} f\big( \frac{g_i}{\sqrt{c}}\big) } \Car_{| \sum_{i=1}^{m_N} g_i^2 - cN | \leq \delta N}& - \sup_{\nu \in \mathcal{P}([-M_N, M_N]) \atop \int x^2 d \nu =c} \Big\{ \int f\Big( \frac{x}{\sqrt{c}} \Big) d\nu(x) - H(\nu| \gamma) \Big\}\Big| \\
& \leq   \varepsilon_{L,K}(N) +\varepsilon_{L}(\delta K),\end{align*}
where $\varepsilon_{L,K}(N)$  (resp. $\varepsilon_{L}(x)$) goes to zero  as $N\to +\infty$ (resp. as $x \to 0$), while $L,K$ are fixed.
\end{lemma}

Let $s\in I_3^l$ such that $|\sum_i s_i^2 - Nc_3|\leq \delta N$. We consider the function 
 $$f : x \mapsto \sum_{i=1}^k L\Big( \frac{2\theta x s_i\sqrt{c_1}}{\sqrt{N
}c} \Big).$$
Using the fact that $\delta \leq c$, one can observe that $f(\sqrt{ . })$ is $8\theta^2L$-Lipschitz.
Using the previous lemma with $M_{N}=\delta N^{1/4}$, we deduce that for any $c_1 \geq K^{-1}$,
\begin{equation}\label{jhL}
\Big|\frac{1}{N}\log \Lambda_{1}^{N} - \sup_{\nu \in \mathcal{P}(I_1)  \atop \int x^2 d\nu(x) = c_1}\big\{ \sum_{i=1}^{k}\int L\Big(\frac{ 2\theta x s_{i}}{\sqrt{N} c}\Big) d\nu(x)- H(\nu|\gamma)\big\}\Big| \leq \varepsilon_{K}(N) + o_{\delta K}(1),
\end{equation}
where $\varepsilon_{K}(N)\to 0$ as $N\to+\infty$.  Putting together \eqref{decomp}, Lemma \ref{explem} and \eqref{jhL}, we obtain
$$
\frac{1}{N} \log \E\Big( \exp(\Sigma(g)) \Car_{\mathcal{A}_{c_1,c_2,c_3} \cap \mathcal{I}_{k,l}}\Big) - \Psi_{k,l}^{{(\delta)}}(c_1,c_2,c_3)  \leq \varepsilon_{\delta,K}(N) +  o_{\delta K}(1) +o_K(1)+{o_\delta(1)},$$
and {similarly for the lower bound where where $\Psi_{k,l}^{(\delta)}(c_1,c_2,c_3)$ is replaced by $\Psi_{k,l}^{(\delta/2)}(c_1,c_2,c_3) $;} with $\varepsilon_{K}(N)\to 0$ as $N\to+\infty$, whereas $\delta$ and $K$ are fixed, 
\begin{eqnarray*}
&& \Psi_{k,l}^{(\delta)}(c_1,c_2,c_3) = Q( c_1,  c_2,c_3)- \frac{1}{2}\big( c_2+c_3\big) \\
&&+\max_{ t_i \in I_2, i\leq l \atop  |\sum_i t_i^2 - c_2N |\leq \delta N }  \max_{s_i \in I_3, i\leq k \atop |\sum_i s_i^2 - c_3N| \leq \delta N} \Big\{\Sigma_{2,3}(t,s) + \Sigma_{2,2}(t,t) + \sup_{\nu \in \mathcal{P}(I_1) \atop \int x^2 d\nu_1(x) = c_1 } \{ \Phi(\nu, s) - H(\nu|\gamma_1)\} \Big\},
\end{eqnarray*}
 with
\begin{equation}\label{phi} Q(x,y,z) = \theta^2\frac{x^2 +2 xy + B z^2}{(x+y+z)^2}, \ \mbox{ and } \  \Phi_{c}(\nu,s) =  \sum_{i=1}^{k}\int L\Big(\frac{ 2\theta x s_{i}}{\sqrt{N}c}\Big) d\nu(x).\end{equation}
By Lemma \ref{fixconfig}, we obtain
\begin{equation} \label{total}  \frac{1}{N}\log F^N_{c_1,c_2,c_3}(\theta,\delta)- \max_{ k\leq 3\sqrt{N}/K  \atop l \leq 3\sqrt{N}/\delta} \Psi_{k,l}^{(\delta)}(c_1,c_2,c_3) \leq  \varepsilon_{\delta,K}(N) +  o_{\delta K}(1) +o_K(1)+o_\delta(1),\end{equation}
and similarly for the lower bound, {where $\Psi_{k,l}^{(\delta)}(c_1,c_2,c_3)$ is replaced by $\Psi_{k,l}^{(\delta/2)}(c_1,c_2,c_3) $}. To prove the lower bound of  Lemma \ref{WigGen2}, we can first write
$$ \liminf_{N\to +\infty} F_N(\theta) \geq \liminf_{N\to +\infty} \frac{1}{N}\log F^N_{c_1,c_2,c_3}(\theta,{2}\delta).$$
Using \eqref{total}, we deduce
 $$ \liminf_{N\to +\infty} F_N(\theta) \geq  \liminf_{N\to +\infty} \max_{ k\leq 3\sqrt{N}/K  \atop l \leq 3\sqrt{N}/\delta} \Psi_{k,l}^{(\delta)}(c_1,c_2,c_3) -  o_{\delta K}(1) -o_K(1)-o_\delta(1).$$
To complete the proof of the lower bound, one can observe that 
$$ \Sigma_{2,2}(t,t) = \frac{1}{N}  \sum_{i,j}  L\Big( 2\theta \frac{t_i t_j}{\sqrt{N} c}\Big) + o_N(1),$$
uniformly in $t_i \in I_2$ such that $|\sum_{i} t_i^2 - c_2 N|\leq \delta N$. Indeed, the diagonal terms are negligible in this case since 
$$ \frac{1}{N} \sum_{t_i\in J_2}\frac{t_i^4}{Nc^2}\leq {\frac{K}{c\sqrt{N}}} \frac{1}{N} \sum_{i\in J_2} t_i^2 =\eps_{K}(N).$$
To complete the proof of the upper bound, we will use the exponential tightness of $|| g||^2$ and of $|| g||_{J_1}^{-2}$. More precisely, we claim that
\begin{equation} \label{tensionexponorm}\lim_{K \to +\infty} \limsup_{N\to +\infty} \frac{1}{N} \Pp( || g||^2\geq KN, || g||_{J_1}^2 \leq K^{-1} N) =-\infty.\end{equation}
Indeed, it is clear by Chernoff's inequality that for {$N$ large enough}
$$ \Pp( || g||^2\geq K N) \leq Ce^{-CN K},$$
where $C$ is a positive numerical constant. Whereas, using Lemma \ref{fixnorm} and a union bound,
$$ \Pp( ||g||_{J_1}^2 \leq K^{-1} N) \leq \sum_{m \leq 3 \sqrt{N}/\delta } \binom{N}{m}\Pp\big( \sum_{i=1}^{N-m}g_i^2\leq K^{-1} N\big).$$
By Chernoff's inequality, we have for any $m\leq N$,
$$ \Pp\big( \sum_{i=1}^{N-m}g_i^2\leq K^{-1} N\big) \leq e^{-(N-m)\Lambda^*(K^{-1})},$$
where $\Lambda^* (x) = \frac{x}{2} - \frac{1}{4}-\frac{1}{2} \log (2x)$ for any $x>0$. Since $\Lambda^*(K^{-1}) \to +\infty$ as $K \to  +\infty $ and for any $m\leq 2 \sqrt{N}/\delta$, the binomial 
$ \binom{N}{m}$ is negligible in the exponential scale, 
we deduce the claim \eqref{tensionexponorm}.

Using that $L(x)\leq Ax^2/2$ for any $x\in \R$, we have 
$$ \Sigma(g) \leq A\theta^2 N.$$
From the exponential tightness \eqref{tensionexponorm}, we deduce that for $K$ large enough
\begin{equation} \label{trunc} \E \exp(\Sigma(g)) \Car_{\{ \frac{1}{N}|| g||^2 \geq K, \frac{1}{N}||g||_{J_1}^2 \leq K^{-1}\}  } \leq 1.\end{equation}
Let $\mathcal{E}_K =\{ \frac{1}{N}|| g||^2 \geq K, \frac{1}{N}||g||_{J_1}^2 \leq K^{-1} \} $. We have
$$0 \leq  \limsup_{N\to +\infty} F_N(\theta) \leq \max\Big( \limsup_{N\to +\infty}\frac{1}{N}\log \E \big(e^{\Sigma(g)} \Car_{ \mathcal{E}_K }\big),\limsup_{N\to +\infty}\frac{1}{N}\log \E\big( e^{\Sigma(g)} \Car_{ \mathcal{E}_K^c}\big)\Big)$$
Since we took $K$ so that \eqref{trunc} holds and $F_{N}(\theta)\ge 0$ {as $X_N$ is centered}, we deduce that
\begin{equation}\label{truncation} \limsup_{N\to +\infty} F_N(\theta) \leq  \limsup_{N\to +\infty}\frac{1}{N}\log \E\big( e^{\Sigma(g)} \Car_{\mathcal{E}_K }\big).\end{equation}
Let now $\mathcal{C}_{\delta}$ be a $\delta$-net for the $\ell^{\infty}$-norm of the set 
$$ \{ (c_1,c_2,c_3) \in \R_+^3 : c_1+c_2+c_3 \leq K , c_1\geq K^{-1}\}.$$
As $| \mathcal{C}_\delta| =O( K/\delta^{ 3})$ is independent of $N$, we have
$$\limsup_{N\to +\infty} F_N(\theta) \leq \max_{(c_1,c_2,c_3) \in \mathcal{C}_\delta} \limsup_{N\to +\infty }\frac{1}{N} \log F^N_{c_1,c_2,c_3}(\theta,\delta).$$
From \eqref{total}, we deduce
$$ \limsup_{N\to +\infty}F_N(\theta) \leq  \max_{c = c_1+c_2+c_3\atop c_1 \geq K^{-1}  } \limsup_{N\to +\infty} \max_{ k\leq 3\sqrt{N}/K \atop l \leq 3\sqrt{N}/\delta } \Psi_{k,l}^{(\delta)}(c_1,c_2,c_3) + o_{\delta K}(1) +o_K(1)+o_\delta(1).$$
Taking now the limit as $\delta \to 0$ and $K\to +\infty$ such that $\delta K\to 0$, we obtain the upper bound of Lemma \ref{WigGen2}.
\end{proof}

We are now ready to prove Proposition \ref{simpleFint}. Building on Lemma \ref{WigGen2}, we show that we can optimize on the total norm $c$ in order to simplify the variational problem.

\begin{proof}[Proof of Proposition \ref{simpleFint} ] \label{proofsimple} We use the notation of Lemma \ref{WigGen2}. {Observe that $ F^N_{0,c_2,c_3}(\delta,K) =-\infty$ for any $c_2,c_3\geq 0$}. Therefore, we may assume that $c =c_1+c_2+c_3>0$. 
We make the following changes of variables. For $ i=1,2,3$, we set $\alpha_i = \frac{c_i}{c}$, and if $h_{c} : x\mapsto x/\sqrt{c}$, and $\nu \in \mathcal{P}(\R)$, {we set $\nu_1 = h_c \# \nu$ to be the push-forward of $\nu$ by $h_c$, defined for any bounded continuous function $f$ to satisfy}
\begin{equation}\label{pushf} \int f(x) d\nu_{1}(x)=\int f\Big(\frac{x}{\sqrt{c}}\Big) d\nu(x).\end{equation}
For any $\nu \in \mathcal{P}(I_1)$ such that $\int x^2 d\nu(x) = c_1$, we have $\int x^{2} d\nu_{1}(x)=\alpha_{1}$ and
 \begin{equation}\label{pushforw} H(\nu | \gamma) = H(\nu) + \frac{1}{2} c_1 +\frac{1}{2} \log (2\pi) \mbox{ and }H(\nu)=\int\log \frac{d\nu}{dx} d\nu(x)=H(\nu_{1})+ \frac{1}{2}\log c.\end{equation}
We obtain
$$\overline{F}(\theta)=  \limsup_{ \delta \to 0, K\to +\infty \atop \delta K \to 0} \sup_{ \alpha_1 + \alpha_2 +\alpha_3=1\atop \alpha_i \geq 0}\sup_{c\geq 0}  \limsup_{N\ra+\infty} \mathcal{F}^N_{\alpha_1,\alpha_2,\alpha_3,c}(\delta,K),$$
and similarly for $\underline{F}(\theta)$,
where
\begin{align*}&\mathcal{ F}^N_{\alpha_1,\alpha_2,\alpha_3,c}(\delta,K)  = \sup_{s_i\in I_3 \atop
|\sum s_i^2-N\alpha_3|\leq N\delta } \sup_{t_i \in I_2\atop
|\sum t_i^2- N \alpha_2|\leq N\delta }\Big\{  \theta^2\big( \alpha_1^2 +2\alpha_1 \alpha_2 +B\alpha_3^2 \big) - \frac{1}{2}c +\frac{1}{2}\log c  \\
&\quad+ \frac{1}{N}\sum_{i,j} L\Big(\frac{2 \theta  s_i t_j}{\sqrt{N}}\Big)  
 +\frac{1}{2N}\sum_{i, j} L\Big(\frac{ 2\theta  t_i t_j}{\sqrt{N}}\Big) + \sup_{\nu_{1} \in \mathcal{P}(h_{c}(I_1)) \atop \int x^2 d\nu_1(x) = \alpha_1 }\big\{ \Phi_{1}(\nu_{1}, s) -H(\nu_{1}) \big\}\Big\}  - \frac{1}{2}\log (2\pi),
\end{align*}
Note that $h_c(I_1) =[-\sqrt{ \delta} N^{1/4},\sqrt{\delta} N^{1/4}]$ is independent of $c$. Optimizing over $c$, and see that the maximum is achieved at $c=1$ {for $N$ large enough}.


\end{proof}

 





\section{The large deviations close to the bulk}\label{closebulk}
 We prove in this section Proposition  \ref{Propsmall}. 
By Theorem \ref{maintheo}, the large deviation lower bound holds at every $x >2$ such that $I(x)=\bar I(x) \neq 0$ 
so that there exists $\theta\in \Theta_{x}$ which does not belong to any $\Theta_{y}$ for $y\neq x$. 
In the following lemma, we prove that if $\overline{F}(\theta) =\underline{F}(\theta)=\theta^2$ on a interval $(0,b)$ with $b>1/2$, then the large deviation lower bound holds in  a neighborhood of $2$ and the rate function $I$ is equal to the one of the GOE.

\begin{lemma}\label{smallem}
If for $\theta \in \big(0,\frac{1}{2\eps}\big)$, for some $\eps \in (0,1)$, $  \overline{F}(\theta)=\underline{F}(\theta) = \theta^2,$
then for any $x \in [2,\eps + \frac{1}{\eps})$, 
$$\bar I(x)=I(x)=  I_{GOE}(x).$$
As a consequence, $\bar I(x) >0$ for any $x>2$. Moreover, for $x \in [2,\eps + \frac{1}{\eps})$ the optimizer in $I$ is taken at $\theta_{x}=1/2G(x)$ and $\theta_{x}\notin \Theta_{y}$ for all $y\in [2,+\infty)\backslash \{x\}$.

\end{lemma}

\begin{proof}
As $\underline F(\theta) \geq \theta^2$ for any $\theta \geq 0$,  we have that
$$\sup_{ \theta \in [0,1/2\eps)} \big\{ J(\theta,x) - \theta^2 \big\} \leq I(x)\le \bar  I(x) \leq \sup_{\theta \geq 0} \big\{ J(\theta,x) - \theta^2 \big\}.$$
But if $x \in [2,\eps+\frac{1}{\eps})$,
$$ I_{GOE}(x) = \sup_{\theta \geq 0} \big\{J(\theta,x) - \theta^{2} \big\}$$
is achieved at $\theta = 1/2G(x) \in (0,1/2\eps)$ since $G^{-1}(\eps) =\eps +1/\eps$. Therefore, if $x \in [2,2\eps+\frac{1}{2\eps})$, then we obtain
$$ I(x)=\bar I(x) = I_{GOE}(x).$$
The consequences are obvious as $G$ is invertible on $[2,+\infty)$.

\end{proof}

The result of Proposition \ref{Propsmall} then follows from Lemma \ref{smallem} and Lemma \ref{F0}. 
We now study the  convergence of the annealed spherical integrals for large values of $\theta$, for which we need to make additional assumptions on  $\mu$.

\section{Case where $\psi$ is an increasing function}\label{incsec}
In this section we make the additional assumption that
$\psi$ is non-decreasing. This assumption is in particular satisfied in the sparse Gaussian case below. 

\begin{example}[Sparse Gaussian distribution] Let $\mu$ be the law of $\xi \Gamma$ where $\xi$ is a Bernoulli variable of parameter $p\in(0,1)$ and $\Gamma$ is a standard Gaussian random variable.  In that case we have for any $ x \in \R$,
\[ \psi(x)  = \frac{\log[ (1-p) + p \exp( x^2/2p)]}{x^2}={\int_0^1 \frac{  t}{ (1-p)  \exp( -(xt)^2/2p)+ p } dt}\] 
is  increasing in $x$ as the integral of increasing functions on $\mathbb R^{+}$.
\end{example}
\subsection{Simplification of the variational problem}
We prove in this section that when $\psi$ is increasing on $\R^+$, $\mathcal C_\mu=\mathbb R^+$ and we can simplify the limit $F(\theta)$ as follows.
\begin{proposition}\label{simplif}
For any $\theta \geq 0$, $\overline{F}(\theta)=\underline{F}(\theta)=F(\theta)$ where
$$F(\theta)=
\sup_{\alpha \in [0,1]} 
\sup_{\nu\in \mathcal{P}(\R)\atop \int x^2 d\nu(x)=\alpha}\bigg\{ \theta^2\alpha^2+ B\theta^2(1-\alpha)^2+ 
\int L(2\theta \sqrt{1-\alpha} x) d\nu(x)- H(\nu) - \frac{1}{2} \log (2\pi)-\frac{1}{2}\bigg\}
$$
\end{proposition}

\begin{proof}With the notation of Proposition \ref{simpleFint}, we need to bound, for any $\delta, K>0$, and $\alpha_1+\alpha_2+\alpha_3 = 1$,
the quantity $ \mathcal{F}^N_{\alpha_1,\alpha_2,\alpha_3}(\delta, K)$.
Since $\psi$ is non-decreasing {on $\R^+$ and symmetric}, we have  for any $s \in I_3^k$ such that $|\sum_i s_i^2 -\alpha_3 N|\leq N\delta $, that for any $i \in\{1,\ldots ,k\}$, $s_{i}\le   \sqrt{(\alpha_{3}+\delta)N}$ so that for any $x\in\R$,
$\psi(2\theta s_i x/\sqrt{N})\le \psi(2\theta \sqrt{(\alpha_{3}+\delta)}x)$. Thus,
\begin{align*}
  \sum_{i=1}^k \int L \Big( \frac{2\theta s_i x}{\sqrt{N}}\Big) d\nu_{1}(x) & = \frac{4\theta^2}{N} \sum_{i=1}^k s_i^2  \int x^2\psi\Big(\frac{2\theta s_i x}{\sqrt{N} }\Big) d\nu_1(x) \\
&\leq  \int L\Big(2\theta \sqrt{\alpha_3+\delta} x\big) d\nu_1(x)= \int L\Big(2\theta \sqrt{\alpha_3} x\big) d\nu_1(x)+o_{\delta}(1)
\end{align*}
where we finally use that $L\circ\sqrt{.}$ is Lipschitz.
On the other hand, since $L(x) \leq Bx^2/2$ for any $x\geq 0$,
$$\frac{1}{N}\sum_{j=1}^{l}\sum_{i=1}^k  L \Big( \frac{2\theta s_i t_{j}}{\sqrt{N}}\Big)   +\frac{1}{2N}  \sum_{i,j=1}^{l}L\Big(\frac{2\theta t_{i} t_{j}}{\sqrt{N}}\Big)  \leq \theta^2 B (\alpha_2^2+2\alpha_3\alpha_2)+o_{\delta}(1).$$
Therefore, we have the upper bound,
\begin{align*}
 \overline{F}(\theta) &\leq \sup_{\alpha_1+\alpha_2+\alpha_3 = 1} \sup_{ \nu_1 \in \mathcal{P}(\R) \atop \int x^2 d\nu_1(x) = \alpha_1} \Big\{ \theta^2 \big( \alpha_1^2 +2\alpha_1 \alpha_2 \big)+ \theta^2 B (\alpha_3+\alpha_2)^2\\
&+ \int L(2\theta \sqrt{\alpha_3} x) d\nu_1(x)-H(\nu_1)- \frac{1}{2} \log (2\pi) - \frac{1}{2}\Big\}.
\end{align*}
We can further simplify this optimization problem by showing that the assumption on the monotonicity of $\psi$ entails that we can take $\alpha_2 = 0$ in the above RHS. Indeed, note that $\psi(0) = 1/2$. Therefore, $\psi$ is bounded below by $1/2$ everywhere. Hence, we deduce that
\begin{align*}
2 \theta^2 \alpha_1 \alpha_2 + \int L(2\theta \sqrt{\alpha_3} x) d\nu_1(x)  & = 2\theta^2 \alpha_1\alpha_2 + 4\theta^2 \alpha_3 \int x^2 \psi(2\theta \sqrt{\alpha_3} x) d\nu_1(x)\\
& \leq 4\theta^2 (\alpha_2+\alpha_3) \int x^2 \psi(2\theta \sqrt{\alpha_2+\alpha_3}x ) d\nu_1(x).
\end{align*}
Thus, with the change of variables $\alpha_3+\alpha_2\ra\alpha_3$, we conclude that
\begin{align*}
 \overline{F}(\theta) &\leq \sup_{\alpha_1+\alpha_3= 1} \sup_{ \nu_1 \in \mathcal{P}(\R) \atop \int x^2 d\nu_1(x) = \alpha_1} \Big\{ \theta^2  \alpha_1^2 + \theta^2 B \alpha_3^2+ \int L(2\theta \sqrt{\alpha_3} x) d\nu_1(x)-H(\nu_1)- \frac{1}{2} \log (2\pi) - \frac{1}{2}\Big\}.
\end{align*}
To prove that $\underline{F}(\theta)$ is bounded from below by the same quantity, we fix $\alpha_1,\alpha_2,\alpha_3$ such that $\alpha_1 +\alpha_2+\alpha_3 = 1,\alpha_2=0$, and $\nu \in \mathcal{P}(\R)$ such that $\int x^2 d\nu(x) = \alpha_1$. We take in the definition of $\mathcal{F}_{\alpha_1,\alpha_2,\alpha_3}^N(K,\delta)$, $k=1$, $s = \sqrt{\alpha_3 N }$,  and $\nu_1=h_{\lambda}\# \frac{1}{\nu(I_1)} \nu(.\cap I_1) ,$ with the notations of \eqref{pushf}. We take $\lambda$  such that $\int x^2 d\nu_1(x) = \alpha_1$, which goes to $1$ as $N$ goes to infinity. We have
\begin{align*} \mathcal{F}_{\alpha_1,\alpha_2,\alpha_3}^N(K,\delta)& \geq  \theta^2\big(\alpha_1^2+ B\alpha_3^{2}\big)  + \int L(2\theta \sqrt{\alpha_3} x) d \nu_1(x) - H( \nu_1)- \frac{1}{2} \log (2\pi)-\frac{1}{2}.
\end{align*}
We deduce by  monotone convergence and the fact that $L(\sqrt{.})$ is Lipschitz that
\begin{align*}   \underline{F}(\theta) \geq  \theta^2\big(\alpha_1^2+ B\alpha_3\big)  + \int L(2\theta \sqrt{\alpha_3} x) d \nu(x) - H( \nu)- \frac{1}{2} \log (2\pi)-\frac{1}{2}.\end{align*}
\end{proof}

We next compute the supremum over $\nu$ in the definition of $F$ in Proposition \ref{simplif}. To this end, we denote by $G : [B/2, + \infty) \to \R \cup\{+\infty\}$ the function given by
\begin{equation}\label{defG} \forall \zeta \in [B/2,+\infty), \ G( \zeta) = \log \int \exp( L(x) - \zeta x^2)dx\,. \end{equation}

\begin{lemma}\label{lemmin}
Let 
\begin{equation} \label{defl} l= - \lim_{\zeta \to B/2} G'(\zeta) \in (0,+\infty]. \end{equation}
For any $C \in (0,l)$, there exists a unique $\zeta \in (B/2,+\infty)$ solution to the equation
\[ G'( \zeta) = -C\,. \]
It is denoted by $\zeta_C$. For $C\geq l$, we set $\zeta_C = B/2$.
Then,
\begin{align*}
 \sup_{\nu \in \mathcal{P}(\R) \atop \int x^2 d \nu(x) = C} \Big [ \int L(x) d\nu(x) - H(\nu) \Big] & =
\sup_{\nu \in \mathcal{P}(\R) \atop \int x^2 d \nu(x) \leq C} \Big[ \frac{B C}{2} + \int \Big( L(x) -\frac{B }{2}x^2\Big)  d\nu(x) - H(\nu) \Big]\\
& = C \zeta_C + G( \zeta_C) 
\end{align*}

\end{lemma}
Note that if $l$ is finite, then so is $G(B/2)$ since $G$ is a convex function.
\begin{proof}
Define the function 
$$ \forall \nu \in \mathcal{P}(\R), \ E (\nu)= H(\nu) + \int \Big( \frac{B }{2}x^2  - L(x) \Big)d\nu(x).$$
We first show that 
\begin{equation}\label{tot}\inf_{\nu \in \mathcal{P}(\R) \atop \int x^2 d \nu(x) = C} E(\nu)  = \inf_{\nu \in \mathcal{P}(\R) \atop \int x^2 d \nu(x) \leq C} E(\nu).\end{equation}
Clearly the RHS is bounded above by the LHS.
To prove the equality, we therefore need to show that for any $\nu \in \mathcal{P}(\R) $ such that  $\int x^2 d \nu(x) \leq  C$, there exists $\nu_{\eps}$ such that $\int x^2 d \nu_{\eps}(x) = C$, and 
$$\lim_{\eps \to 0} \nu_{\eps} = \nu,  \ \mbox{ and } \  \lim_{\eps \to 0} E( \nu_{\eps}) \ge E(\nu).$$
 We set $\nu_{\eps} = (1 - \eps) \nu + \eps \gamma_{\eps}$ where $\gamma_\eps$ is a Gaussian measure of variance $1$ and mean $m_{\eps}$, given, if $D=\int x^{2} d\nu(x)$ by
\[ m_{\eps} = \sqrt{ \frac{ C - (1- \eps)D}{\eps} -1 }\,. \]
With this choice of $m_\eps$, one can check that  $\int x^2 d \nu_{\eps}(x) = C$. Moreover,
$ \nu_{\eps}$ converges weakly to $ \nu$ as $\eps$ goes to zero.
 As $\int x^2 d\nu_{\eps}(x) \leq C$, and  $( \frac{B}{2}x^2 -  L(x))/x^{2}$ goes to zero as $x$ goes to infinity, we deduce that 
\begin{equation} \label{conv} \lim_{\eps \to 0^+} \int \Big( \frac{B }{2}x^2 - L(x) \Big) d\nu_{\eps}(x) = \int \Big( \frac{B }{2}x^2 - L(x) \Big)d\nu(x).\end{equation}
Besides, as $H$ is lower semi-continuous,
$$ \liminf_{\eps \to 0^+} H(\nu_\eps) \geq H(\nu).$$We  conclude together with \eqref{conv} that,
$ \lim_{\eps \to 0^+} E(\nu_\eps) \ge  E(\nu),$
which ends the proof of the claim \eqref{tot}. 
Observe that $E$ is a lower semi-continuous function for the weak topology since $H$ is lower semi-continuous and $x\mapsto Bx^2/2 - L(x)$ is non-negative and continuous. Moreover, the set 
$$ \big\{ \nu \in\mathcal{P}(\R) : \int x^2 d\nu(x) \leq C\big\},$$
is a compact set. Thus, the supremum of $E$ over the set above is achieved. We will identify the maximizer.
For any $\zeta \in [B/2, + \infty)$ such that $G(\zeta) < +\infty$,  we let  $\nu_\zeta$  be  the probability measure given by
\[ d\nu_{\zeta} = \frac{\exp( L(x) - \zeta x^2)}{\int \exp( L(y) - \zeta y^2)dy} dx. \]
We will show that 
\begin{equation} \label{min}  \inf_{\nu \in \mathcal{P}(\R) \atop \int x^2 d \nu(x) \leq C} E(\nu) = E(\nu_{\zeta_C}).\end{equation} 
Let $\mu$ be a probability measure such that $H( \mu ) < + \infty$ and $\int x^2 d \mu(x)  \leq C$. As $H(\mu) <+\infty$, we can write,
$$ \mu = (1+\phi) d\nu_{\zeta_C},$$
where $\phi$ is some measurable function such that $\phi \geq -1$ $\nu_{\zeta_C}$-a.s. and $\int \phi d\nu_{\zeta_C} =0$. We have
\begin{align*}
E(\mu) &= E(\nu_{\zeta_C}) + \int \Big( \frac{B}{2}x^2 -L(x)\Big) \phi(x) d\nu_{\zeta_C}(x) \\
& +\int (1+\phi) \log(1+\phi) d\nu_{\zeta_C}   +\int \log \frac{d\nu_{\zeta_C}}{dx} \phi d\nu_{\zeta_C}.
\end{align*}
 By convexity of $x\mapsto x\log x$, we know that
$$ \int  (1+\phi)\log(1+\phi) d\nu_{\zeta_C}  \geq  (1+\int \phi d \nu_{\zeta_C})\log (1+\int \phi d \nu_{\zeta_C})  =0.$$ 
Therefore, using again that $\int \phi  d\nu_{\zeta_C}=0$ to cancel the contribution from the partition function of $\nu_{\zeta_{C}}$, we get
\begin{eqnarray*}
&&E(\mu)-  E(\nu_{\zeta_C})\geq  \int\left( \frac{B}{2}x^2 -L(x)+\log \frac{d\nu_{\zeta_C}}{dx}  \right) \phi(x) d\nu_{\zeta_C}(x) \\
&&\qquad  =\Big( \frac{B}{2} -\zeta_C\Big) \int x^2 \phi(x) d\nu_{\zeta_C}(x) =  \Big( \frac{B}{2} -\zeta_C\Big) \Big( \int x^2 d\mu(x) - \int x^2 d\nu_{\zeta_C}(x)\Big).
\end{eqnarray*}
If $ C<l$, then $\int x^2 d\nu_{\zeta_C}(x) = C$. Since $\zeta_C \leq B/2$ and $\int x^2 d\mu(x) \leq C$ so that the RHS is non negative.
If $C\geq l$, then $\zeta_C = B/2$, and we also get $E(\mu) \geq  E(\nu_{\zeta_C})$.
This shows that $\nu_{\zeta_C}$ achieves the infimum in \eqref{min}, and ends the proof of Lemma \ref{lemmin}.
\end{proof}

\subsection{Differentiability of the limit of the annealed spherical integral}
This section is devoted to the proof of the following proposition. 

\begin{proposition}\label{diffF}
$F$ is continuously differentiable on $(1/\sqrt{B-1},+\infty)$ except possibly at the point $\theta_0$ such that 
$$ \theta_0 = \inf \big\{ \theta : F(\theta) > \theta^2 \big\}.$$  
 Moreover, for any $\theta \leq 1/2\sqrt{B-1}$,
\begin{equation}\label{theta2}
F(\theta) = \theta^2.\end{equation}
\end{proposition}

The second part of the claim of the above proposition \eqref{theta2} is due to Proposition \ref{F0} and the fact that $A=B$. From now on, we assume that $\theta^2(B-1) > 1$ and wish to prove the first part of Proposition \ref{diffF}. 
We define for any $\alpha\in[0,1]$, and $\nu\in \mathcal{P}(\R)$,
\begin{equation} \label{defHtheta} H_{\theta}(\alpha,\nu)=
\theta^2\alpha^2+ \theta^2B(1-\alpha)^2+
 \int L(2\theta \sqrt{1-\alpha} x) d\nu(x)- H(\nu)- \frac{1}{2} \log (2\pi) - \frac{1}{2}.
\end{equation}
By Proposition \ref{simplif}, we have
\begin{equation} \label{defFsimpl} F(\theta) = \sup_{(\alpha,\nu) \in S} H_\theta(\alpha,\nu),\end{equation}
where 
$S = \{ (\alpha,\nu) \in [0,1] \times \mathcal{P}(\R) :  \ \int x^2 d\nu(x) = \alpha\}.$
We first show  that we can can restrict the parameter $\alpha$ to the set $[0,\frac{1}{2}]\cup\{1\}$, as described in the following lemma. \begin{lemma}\label{bound} If $\theta^2 (B-1) \geq 1$, then
$$ F(\theta) = \max\Big(\sup_{(\alpha,\nu) \in S \atop  \alpha \leq \frac{1}{2} } H_\theta(\alpha,\nu), \theta^2\Big).$$
\end{lemma}
\begin{proof} Up to replace $\nu$ by $h_{{\alpha}}\#\nu$, {where $h_\alpha : x \mapsto x /\sqrt{\alpha}$,} $$ \sup_{(\alpha,\nu) \in S} H_\theta(\alpha,\nu) = \sup_{ \alpha \in (0,1] \atop \int x^2 d\nu =1} \{ \tilde H_\theta(\alpha,\nu)-H(\nu)\},$$
where for any $\alpha \in (0,1]$, and $\nu \in \mathcal{P}(\R)$, 
$$ \tilde H_\theta(\alpha,\nu) = \theta^2\alpha^2+ \theta^2B(1-\alpha)^2+\frac{1}{2} \log \alpha+
 \int L(2 \theta \sqrt{(1-\alpha)\alpha } x) d\nu(x)- \frac{1}{2} \log (2\pi) - \frac{1}{2}.$$
We claim that for any $\nu \in \mathcal{P}(\R)$ such that $\int x^2 d\nu(x) =1$,
\begin{equation} \label{majH}  \max_{\alpha\in [1/2,1]}\tilde H_\theta(\alpha,\nu)\leq  \max\Big( \tilde H_\theta\Big(\frac{1}{2},\nu\Big), \tilde H_\theta(1,\nu)\Big).\end{equation}
Indeed, first notice that since $\psi$ is increasing, for all $\alpha\in [0,1]$ we have,
\begin{eqnarray*}
  \int L(2 \theta \sqrt{(1-\alpha)\alpha } x) d\nu(x) &=& 4\theta^2 {\alpha(1-\alpha)} \int x^2 \psi(2\theta\sqrt{\alpha(1-\alpha)} x) d\nu(x)\\
  & \le &4\theta^2  {\alpha(1-\alpha)}\int x^2 \psi(\theta x) d\nu(x)\,.\end{eqnarray*}
Denote by $m =2 \int x^2 \psi(\theta x) d\nu(x)\in [1,B]$. For any $\alpha \in (0,1]$,
$$ \tilde H_\theta(\alpha,\nu) \leq \theta^2\alpha^2+ \theta^2B(1-\alpha)^2+ \frac{1}{2} \log \alpha+
 2\theta^2 \alpha(1-\alpha) m- \frac{1}{2} \log (2\pi) - \frac{1}{2}=:f_{\theta,m}(\alpha)\,.$$ 
 We find that 
$$ f_{\theta,m}'\Big( \frac{1}{2}\Big) = \theta^2(1 -B)+1,\quad f_{\theta,m}''(\alpha) = 2\theta^2(B+1-2m)-\frac{1}{2\alpha^2}.$${
Since $f_{\theta,m}''$ is increasing and $f_{\theta,m}''(0)=-\infty$, we deduce that $f'_{\theta,m}$ is either   decreasing or decreasing and then increasing.  Since $f_{\theta,m}'(1/2) \leq 0$, we conclude that $f_{\theta,m}$ is either decreasing or decreasing and then increasing on $[1/2,1]$. Therefore, 
$$ \max_{ \alpha \in [\frac{1}{2},1]} f_{\theta,m}(\alpha) = \max\Big( f_{\theta,m}\Big( \frac{1}{2}\Big), f_{\theta,m}(1)\Big)\,,$$
which yields the claim  \eqref{majH} since $f_{\theta,m}(\alpha)= \tilde H_\theta(\alpha,\nu)$ at the two points $\alpha=1/2$ and $1$. To conclude the proof we observe that since $\tilde H_\theta(1,\nu) = \theta^2$ for any $\nu\in\mathcal{P}(\R)$, we have
$$ \sup_{ \int x^2 d \nu(x) = 1} \{\tilde H_\theta (1,\nu) -H(\nu)\}= \theta^2 +\frac{1}{2} + \frac{1}{2} \log (2\pi)-\inf_{ \int x^2 d \nu(x) = 1} H(\nu)=\theta^2 .$$}
\end{proof}
Due to Lemma \ref{lemmin}, we can further simplify the optimization problem defining $F(\theta)$ in \eqref{defFsimpl} by optimizing on $\nu \in\mathcal{P}(\R)$ such that $\int x^2 d\nu(x) = \alpha$, given $\alpha \in(0,1)$.
\begin{corollary}\label{corsimplif}
Let $R$ be the function
$$ R : C \in (0,+\infty) \mapsto C \zeta_C + G(\zeta_C),$$
where $\zeta_C$ is defined as in Lemma \ref{lemmin}. 
Denote for any $\alpha\in(0,1)$,
$$ K_\theta(\alpha) = \theta^{2}(\alpha^{2}+B(1-\alpha)^{2})+R(4\theta^2 \alpha(1-\alpha))-\frac{1}{2}\log (1-\alpha)-\log (2\theta)
- \frac{1}{2} \log (2\pi) - \frac{1}{2},$$
and $K_\theta(1) =\theta^2 $. Then, for any $\theta \geq 1/\sqrt{B-1}$,
$$F(\theta)=
\sup_{\alpha\in (0,1]} K_\theta(\alpha)\,.
$$
\end{corollary}
\begin{proof}
When $\alpha<1$, we make the following change of variables which consists in replacing $\nu$ by its pushforward by $x\mapsto  2\theta\sqrt{1-\alpha}x$. Using \eqref{pushforw}, we find that
$$H(\nu)=\int\log \frac{d\nu}{dx} d\nu=H(\nu_1)-\frac{1}{2}\log (1-\alpha)-\log (2\theta)$$
and $\int x^2 d\nu(x) =4\alpha(1-\alpha)\theta^2$.
Thus,
\begin{equation} \label{optimF} F(\theta) = \max\Big( \theta^2, \sup_{ (\alpha,\nu) \in S'} K_\theta(\alpha,\nu)\Big),\end{equation}
where
$ S' = \big\{ (\alpha,\nu) \in (0,1)\times \mathcal{P}(\R) :  \int x^2 d\nu(x) = 4\alpha (1-\alpha) \theta^2 \big\},$
and  for any $ \alpha \in (0,1), \nu \in \mathcal{P}(\R)$,
\begin{align*}  K_\theta(\alpha, \nu)&=\theta^{2}(\alpha^{2}+B(1-\alpha)^{2})+\int L(x) d\nu(x)
-H(\nu) -\frac{1}{2}\log (1-\alpha)-\log (2\theta)- \frac{1}{2} \log (2\pi e) .
\end{align*}
By Lemma \ref{lemmin}, we obtain for any $\alpha \in (0,1)$,
\begin{align*} \sup_{\int x^2 d\nu(x) = 4\alpha(1-\alpha)\theta^2} K_\theta(\alpha,\nu) &= \theta^{2}(\alpha^{2}+B(1-\alpha)^{2})+4\theta^2\alpha (1-\alpha) \zeta_{\alpha,\theta}+ G(\zeta_{\alpha,\theta}) \\
&-\frac{1}{2}\log (1-\alpha)-\log (2\theta)- \frac{1}{2} \log (2\pi) - \frac{1}{2},
\end{align*}
where $\zeta_{\alpha,\theta} = \zeta_{4\theta^2\alpha(1-\alpha)}$. Hence, if we set,  for $\alpha \in (0,1)$,
$$ K_\theta(\alpha) = \theta^{2}(\alpha^{2}+B(1-\alpha)^{2})+R(4\theta^2 \alpha(1-\alpha)) -\frac{1}{2}\log (1-\alpha)-\log (2\theta)
- \frac{1}{2} \log (2\pi) - \frac{1}{2},$$
we deduce from \eqref{optimF}  that 
\begin{equation} \label{maxF}  F(\theta) = \max\Big( \theta^2, \sup_{\alpha \in (0,1)} K_\theta(\alpha)\Big).\end{equation}
To study the supremum of $K_\theta$, we will need the following result on the limit of $R$ at $0$, which will allow us to compute the limit of $K_\theta$ at $1$.
\begin{lemma}\label{behavR}When $C \to 0^+$,
$$ R(C)= \frac{1}{2} + \frac{1}{2}\log (2\pi C) + o(1).$$
\end{lemma}
\begin{proof}
For $C<l$, we have
\begin{equation} \label{defzeta} G'(\zeta_C) = -C.\end{equation}
Since $G'(C)$ goes to zero when $C$ goes to infinity,
and $G'$ is invertible as $G''(x)>0$ on $(-\infty,l)$, we find that
$$ \lim_{C\to0} \zeta_C = +\infty.$$
From the inequalities, $L(x)/x^{2} \in [0,B/2]$, we deduce that by  \eqref{defG} of $G$  we have the bounds
$$ \frac{1}{2} \log \frac{\pi}{\zeta} \leq G(\zeta) \leq \frac{1}{2} \log \frac{\pi}{\zeta-\frac{B}{2}},$$
which implies that
\begin{equation} \label{ineqG} G(\zeta)\sim_{+\infty} \frac{1}{2} \log \frac{\pi}{\zeta }.\end{equation}
On the other hand, inserting the bound $L(x)/x^{2} \in [0,B/2]$ and \eqref{ineqG}  in the numerator and the denominator of the derivative, we obtain
$$ \sqrt{\frac{\zeta - \frac{B}{2}}{\zeta}}\frac{1}{2\zeta} \leq -G'(\zeta) \leq\sqrt{ \frac{\zeta}{\zeta-\frac{B}{2}}} \frac{1}{2(\zeta-\frac{B}{2})}.$$
We deduce,  since $\zeta_C \to +\infty$ as $C\to 0$, that
$$ G'(\zeta_C) = -\frac{1}{2\zeta_C} + o\Big( \frac{1}{\zeta_C}\Big).$$ 
Therefore, we get from the definition of $\zeta_C$  \eqref{defzeta} that
$ \zeta_C$ is equivalent to $ \frac{1}{2C}$ when $C$ goes to zero.  
Using \eqref{ineqG}, we can conclude that 
$$R(C) = \frac{1}{2} + o(1) + \frac{1}{2} \log \frac{\pi}{ \frac{1}{2C} + o\Big( \frac{1}{2C}\Big)} = \frac{1}{2} + \frac{1}{2} \log (2\pi C) +o(1).$$
\end{proof}
From  Lemma \ref{behavR}, we deduce that 
\begin{equation}\label{Keq} \lim_{\alpha \to 1} K_\theta(\alpha ) = \theta^2 ,\end{equation}
so that we can continuously extend $K_\theta$ to $1$. Therefore, \eqref{optimF} gives
$$F(\theta) = \sup_{\alpha \in (0,1]} K_\theta(\alpha).$$
\end{proof}
We now study  $K_\theta$  and show that it is continuously differentiable on $(0,1)$. This amounts to prove that $R$ is continuously differentiable on $(0,1)$. 
On $(0,l)$, it is clear that $R$ is continuously differentiable due to the implicit function theorem. Indeed, $\zeta_C$ is by definition the unique solution of the equation
$$ G'(\zeta) = -C,$$
and $G$ is strictly convex. On $(l,+\infty)$, $R(C)=B C/2+G(B/2)$ is an affine function, therefore it is sufficient to prove that 
\begin{equation} \label{contdiff} \lim_{C \to l^-} R'(C) = \frac{B}{2}.\end{equation}
This is a consequence of the fact that for any $C <l$,
$$ R'(C) = C \partial \zeta_C + \zeta_C + \partial \zeta_C G'(\zeta_C)\\ 
 = \zeta_C,$$
which gives \eqref{contdiff}. We deduce that $K_\theta$ is continuously differentiable on $(0,1)$ and 
$$\forall \alpha \in (0,1), \  K_{\theta}'(\alpha) = 2\theta^2 (\alpha +B(\alpha-1))+4\theta^2 \zeta_{\alpha,\theta} (1-2\alpha) +\frac{1}{2(1-\alpha)}.$$
From Lemma \ref{behavR}, we know that $K_{\theta}$ goes to $-\infty$ when $\alpha$ goes to zero so that the
 supremum of $K_\theta$ on $(0,1]$ is achieved either at $1$ or on $(0,1)$.  From Lemma \ref{bound}, we find for $\theta^{2}(B-1)\ge 1$,
$$ F(\theta) = \max\big( K_\theta(1),\sup_{\alpha\leq \frac{1}{2} } K_\theta(\alpha) \big).$$
Let us assume that the maximum of $K_\theta$ is achieved on $(0,1)$. We deduce that the maximum of $K_\theta$ is achieved on $(0,\frac{1}{2}]$ at a critical point since $K_\theta$ is differentiable.
 The critical points $\alpha$ of $K_\theta$ satisfy the equation,
\begin{equation}\label{supcrit}
2\theta^2 (\alpha +B(\alpha-1))+4\theta^2 \zeta_{\alpha,\theta} (1-2\alpha) +\frac{1}{2(1-\alpha)}=0,
\end{equation}
As $\theta^2(B-1) >1$, $1/2$ does not satisfy the above equation so that the critical points of $K_\theta$ are the $\alpha \neq 1/2$, such that   
\begin{equation}\label{eqzeta}\zeta_{\alpha,\theta} =\frac{ 2\theta^2 (\alpha +B(\alpha-1))+\frac{1}{2(1-\alpha)}}{4\theta^2(2\alpha-1)}:=\varphi(\alpha).\end{equation}
We find that 
$$ \phi(\alpha)\geq \frac{B}{2} \Longleftrightarrow  \begin{cases}
P(\alpha) \geq 0 & \text{if } \alpha \geq \frac{1}{2}\\
P(\alpha) \leq 0 & \text{if } \alpha \leq \frac{1}{2},
\end{cases}$$
with $P(\alpha) = 4\theta^2(B-1) \alpha^2 - 4\theta^2 (B-1) \alpha +1$. As $\zeta_{\alpha,\theta} \geq B/2$, we obtain that the maximum of $K_\theta$ is achieved  at $\alpha \in (0,1/2)$ such that $P(\alpha) \leq 0$. The roots of $P$ are 
\begin{equation} \label{defalpha}\alpha_{\pm} = \frac{ 1 \pm \sqrt{ 1 - [\theta^2 (B-1)]^{-1}}}{2}.\end{equation}
Thus, the maximum of $K_\theta$ is achieved on $[\alpha_-,1/2]$.
We will show that $K_\theta$ is strictly concave on $(0,\frac{1}{2})$. Note that,
$$ 4 \theta^2 \alpha(1-\alpha ) \geq l \Longleftrightarrow \alpha \in [\beta_-, \beta_+],$$
with 
$$ \beta_\pm = \frac{1 \pm \sqrt{1-l\theta^{-2} }}{2}.$$ 
For any $\alpha \in (\beta_-,\frac{1}{2})$, we must have $\zeta_{ \alpha,\theta}=B/2$ and therefore 
$$ K_\theta(\alpha) = \theta^{2}(\alpha^{2}+B(1-\alpha)^{2})+2B\theta^2\alpha (1-\alpha) + G\Big(\frac{B}{2}\Big)\\
-\frac{1}{2}\log (1-\alpha)+C_\theta,$$
where $C_\theta$ is some constant depending on $\theta$. Thus, for $\alpha \in (\beta_-,\frac{1}{2})$,
$$K_\theta''(\alpha) = 2\theta^2(1-B) +\frac{1}{2(1-\alpha)^2}< 2\theta^2(1-B) + 2 <0.$$
For any $\alpha \in (0,\beta_-)$, we have
\begin{align*} K_\theta(\alpha) &= \theta^{2}(\alpha^{2}+B(1-\alpha)^{2})+4\theta^2\alpha (1-\alpha) \zeta_{\alpha,\theta} + G(\zeta_{\alpha,\theta})\\
&-\frac{1}{2}\log (1-\alpha)+C_\theta,
\end{align*}
where $\zeta_{\alpha,\theta}$ is such that 
$$ G'(\zeta_{\alpha,\theta}) = -4\theta^2 \alpha(1-\alpha).$$ 
As $G$ is strictly convex, we deduce by the implicit function theorem that $\alpha \in (0,\beta_-) \mapsto \zeta_{\alpha,\theta}$ is differentiable, and we have
$$ \partial_{\alpha} \zeta_{\alpha,\theta} G''(\zeta_{\alpha,\theta}) =-4\theta^2(1-2\alpha).$$
We deduce that $\partial_{\alpha} \zeta_{\alpha,\theta} < 0$,
for any $\alpha \in (0,\beta_-)$. Therefore, for $\alpha \in (0,\beta_-)$, we obtain
$$K''_\theta(\alpha) = 2\theta^2(B+1) - 8\theta^2\zeta_{\alpha,\theta} +4\theta^2 \partial_\alpha \zeta_{\alpha,\theta}(1-2\alpha) + \frac{1}{2(1-\alpha)^2}.$$
Using that $\zeta_{\alpha,\theta} > B/2$ and that $\partial_\alpha \zeta_{\alpha,\theta} < 0$ for $\alpha \in (0,\beta_-)$, we find that 
\begin{align*}
\forall \alpha \in (0,\beta_-), \ K''_\theta(\alpha)& \leq 2\theta^2(B+1) - 4\theta^2 B +\frac{1}{2(1-\alpha)^2},2\theta^2(1-B) + 2 \leq 0.
\end{align*}
Thus, $K_\theta'$ is decreasing on $(0,\beta_-)$ and $(\beta_-,\frac{1}{2})$. Since $K_\theta'$ is continuous, we deduce that $K_\theta'$ is decreasing on $(0,\frac{1}{2})$ and $K_\theta$ is strictly concave on $(0,\frac{1}{2})$. Therefore, the maximum is achieved at the unique critical point of $K_\theta$ on $(0,\frac{1}{2})$ which we denote by  $\alpha_\theta$. 
 We distinguish two cases.
\subsubsection*{$1^{\text{st}}$ case:  $l\leq \frac{1}{B-1}$.} We  then have
$ \beta_- \leq \alpha_- \leq \alpha_+\leq \beta_+.$
We know that on one hand $P(\alpha_-) = 0$, so that 
$ \phi(\alpha_-) = \frac{B}{2}.$
On the other hand $\zeta_{\alpha_-,\theta} = B/2$ since $\alpha_- \in [\beta_-,\beta_+]$. We deduce by \eqref{eqzeta} that $\alpha_-$ is a critical point of $K_\theta$ which lies in $(0,\frac{1}{2})$. Therefore 
$ \alpha_\theta = \alpha_-.$
\subsubsection*{$2^{\text{nd}}$ case:  $l>\frac{1}{B-1}$.} We have,
$ \alpha_- < \beta_- < \beta_+ < \alpha_+.$
Note that $0 \leq \alpha_- < \frac{1}{2} < \alpha_+\leq 1$. Since $\phi(\alpha) \neq B/2$ for any $\alpha \in [\beta_-,\beta_+]^c$, we deduce that $\alpha_\theta \in [\alpha_-,\beta_-)$, and in particular $K_\theta''(\alpha_\theta) <0$.  We deduce by the implicit function theorem that $\theta \mapsto \alpha_\theta$ is $C^1$, and therefore $\theta \mapsto K_\theta(\alpha_\theta)$ is continuously differentiable on $(1/\sqrt{B-1},+\infty)$.

In conclusion, we have shown that for any $\theta^2(B-1) \geq1$, if $l\leq \frac{1}{B-1}$,
$$ F(\theta) = \max\big( \theta^2, K_\theta(\alpha_-)\big),$$
where $\alpha_-$ is defined in \eqref{defalpha}, whereas if $l\geq \frac{1}{B-1}$,
$$ F(\theta) = \max\big(\theta^2, K_\theta(\alpha_\theta)\big),$$
where $\alpha_\theta$ is the unique solution in $(0, \beta_-)$ such that 
$ G'(\zeta_{\alpha,\theta}) = -4\theta^2 \alpha(1-\alpha).$

To conclude that $F$ is continuously differentiable on $( 1/ \sqrt{B-1},+\infty)$ except at most at one point, we show that there exists $\theta_0$ such that 
$$ \forall \theta \leq \theta_0, \ F(\theta) = \theta^2, \text{ and } \forall \theta >\theta_0, \ F(\theta) > \theta^2.$$
Since $F(\theta) \geq \theta^2$ for any $\theta \geq 0$, it suffices to prove that $\theta \mapsto F(\theta) - \theta^2$ is non-decreasing. Recall that
$$ F(\theta) = \lim_{N\to +\infty} F_N(\theta),$$
where 
$$ F_N(\theta) = \frac{1}{N} \log \E_e \exp\Big( \sum_i L\big( \sqrt{2N} \theta e_i^2\big) + \sum_{i<j} L\big( 2\sqrt{N}\theta e_ie_j\big)\Big),$$
and $e$ is uniformly sampled on $\mathbb{S}^{N-1}$. Therefore,
$$ F_N(\theta)- \theta^2 = \frac{1}{N} \log \E_e \exp\Big( \sum_i 2N\theta^2 \Big(\psi\big( \sqrt{2N} \theta e_i^2\big)-\frac{1}{2}\Big)e_i^4  + \sum_{i<j} 4N \theta^2 \Big(\psi\big( 2\sqrt{N}\theta e_ie_j\big)-\frac{1}{2}\Big)e_i^2 e_j^2\Big).$$
As $\psi$ is increasing and $\psi(0)=1/2$,  $\theta \mapsto F_N(\theta) -\theta^2$ is non-decreasing, and therefore $\theta \mapsto F(\theta)-\theta^2$ is non-decreasing as well.

%
For the sake of completeness, we show the following Proposition which indicates that it is unlikely we could prove the large deviation principle for all values of $x$ by following our strategy because $F$ is in general not differentiable everywhere.
\begin{proposition}
Assume $\theta_0 = \inf \{\theta \in \R^+: F(\theta) > \theta^2 \} >1/ \sqrt{B-1}$.  Then,  $F$ is not differentiable at $\theta_0$. 
\end{proposition}
\begin{proof}
Let $\theta >\theta_0$.
Lemma \ref{bound} shows that with $H_{\theta}$ defined in \eqref{defHtheta} , we have
$$ F(\theta) = \max_{ \int x^2 d\nu(x)\leq \alpha \atop \alpha \leq \frac{1}{2}} H_\theta(\alpha,\nu).$$
Since $\theta_0 \geq 1/\sqrt{B-1}$, we know from the proof of Proposition \ref{diffF} that there exists $\alpha_0 \leq 1/2$ and $\nu_0 \in \mathcal{P}(\R)$ such that
$$ H_{\theta_0}(\alpha_0,\nu_0) = F(\theta_0).$$
Define $g(\theta) = H_\theta(\alpha_0,\nu_0)$ for any $\theta \geq \theta_0$. Let  $F'_+$ denote the right derivative of $F$. We have as $F \geq g$ and $F(\theta_0) = g(\theta_0)$,
$$ F'_+(\theta_0) \geq g'(\theta_0).$$
We find
$$g'(\theta_0)= 2 \theta_0(\alpha^2_{0} + B(1- \alpha_{0})^2) + 2 \sqrt{1-\alpha_0}\int x L' (2\theta_0 \sqrt{1-\alpha_{0}} x) d\nu_{0}(x).$$
Since $\psi$ is increasing, $xL'(x) \geq 2L(x)$, and $L(x) \geq x^2/2$ for any $x\geq 0$. Therefore, $xL'(x) \geq x^2$ and we deduce
\begin{align*} g'(\theta_0)& \geq  2 \theta_0(\alpha^2_{0} + B(1- \alpha_{0})^2) +  4\theta_0 \alpha_0 (1-\alpha_0)\\
& \geq 2\theta_0 + 2\theta_0 (1-\alpha)^2 (B-1).
\end{align*}
This shows that $g'(\theta_0)>2\theta_0$ and therefore $F'_+(\theta) >2\theta_0$. It yields that $F$ is not differentiable at $\theta_0$.

\end{proof}

\subsection{Proof of Proposition \ref{increasingprop}}
By Proposition \ref{diffF}, we know that $F$ is differentiable on $(1/\sqrt{B-1},+\infty)$ except possibly at $\theta_0$. Using Proposition \ref{lb} we deduce that there exists $x_\mu$  finite such that the lower large deviation lower bound holds with rate function $I(x) = \overline{I}(x)$ for any $x \geq x_\mu$.

\section{The case $B < A$}\label{seccompact}
We consider in this section the case where the  following assumption holds.
\begin{assum}\label{As-compact}
 $B$ exists and is strictly smaller than $A$. Moreover, we assume that $\psi$ achieves its  maximum $A$ at a unique point $m_{*}$ such that $\psi''(m_{*}) <0$.
 \end{assum}
  The first condition  includes in particular the case where the law of the entries have a compact support (since in this case $B=0$) and we believe the second condition is true quite generically, as we check in the following example.
\begin{example}
Let 
$$\mu=\frac{p}{2}(\delta_{-1/\sqrt{p}}+\delta_{1/\sqrt{p}})+(1-p)\delta_{0},\qquad \psi(x)=\frac{1}{x^{2}}\log( p(\cosh(\frac{x}{\sqrt{p}})-1)+1)\,.$$
Then,   for $p<1/3$, $\mu$ satisfies Assumption \ref{As-compact} (but for $p>1/3$ $\mu$ has a sharp sub-Gaussian tail).
Indeed, we have
$$ \forall x\geq 0, \  \psi'(x) = \frac{L'(x)}{x^2} - \frac{2L(x)}{x^3},  \ \psi''(x) = \frac{L''(x)}{x^2} - \frac{4L'(x)}{x^3}+\frac{6L(x)}{x^4}.$$
We claim that $h : x\mapsto xL'(x) - 2L(x)$ is increasing and then decreasing on $\R_+$. Indeed, 
$$ \forall x \geq 0, \  h'(x) = xL''(x) -L'(x), \ h''(x) = xL^{(3)}(x),$$
and we have,
$$ L(x) = \log\big( p\cosh\Big( \frac{x}{\sqrt{p}}\Big) +1-p\big), \ L'(x) = \frac{\sqrt{p} \sinh\big( \frac{x}{\sqrt{p}}\big) }{p\cosh\big( \frac{x}{\sqrt{p}}\big) +1-p}.$$
$$ L''(x) = \frac{p+(1-p)\cosh\big( \frac{x}{\sqrt{p}}\big)}{(p\cosh\big(\frac{x}{\sqrt{p}}\big) + 1-p)^2}, \ L^{(3)}(x) = \frac{\frac{(1-p)^2}{\sqrt{p}} - 2p\sqrt{p} - \sqrt{p}(1-p) \cosh\big( \frac{x}{\sqrt{p}}\big)}{(p\cosh\big( \frac{x}{\sqrt{p}}\big) + 1-p)^3 }\sinh\big( \frac{x}{\sqrt{p}}\big).$$ 
We have, for $p>p_*=1/3$ 
$$\frac{(1-p)^2}{\sqrt{p}} - 2p\sqrt{p}  < \sqrt{p}(1-p) .$$
Therefore, 
 $L^{(3)}$ is negative and therefore $h'$ is decreasing. Since $h'(0)=0$, we deduce that $h'$ is negative and $\psi$ is decreasing. If $p>p_*$, we have that $h''$ is positive and then negative. Therefore, $h'$ is increasing on $[0, x_0]$ and then decreasing on $[x_0,+\infty)$, with $x_0=\sqrt{p}\cosh^{-1} (\frac{1-2p-p^2}{p(1-p)})$. But,
$$h'(0 ) = 0, \   \lim_{x\to +\infty} h'(x) = - \frac{1}{\sqrt{p}},$$
as $L'(x) \sim_{+\infty} 1/\sqrt{p}$ and $L''(x) \sim_{+\infty}2(1-1/p)e^{-x/\sqrt{p}}$. Therefore, there exists $m_{*}>x_0$ such that $h'$ is positive on $(0,m_{*})$ and negative on $(m_{*},+\infty)$. We deduce that $\psi$ is increasing on $(0,m_{*})$ and decreasing on $(m_{*},+\infty)$ so that $\psi$ achieves its unique maximum at $m_{*}$. Moreover, $\psi''(m_{*})<0$. Indeed, otherwise we have
\begin{align*} \psi'(m_{*}) = 0, \psi''(m_{*}) = 0 & \Longleftrightarrow m_{*}L'(m_{*}) = 2L(m_{*}), \ m_{*}^2 L''(m_{*})= 4m_{*} L'(m_{*}) -6L(m_{*})\\
& \Longleftrightarrow L'(m_{*}) = m_{*}L''(m_{*}), \ m_{*}L'(m_{*}) = 2L(m_{*}).
\end{align*}
This implies that $h'(m_{*})=0$ which contradicts $m_{*}>x_{0}$ (and $L^{(3)}\neq 0$ on $[x_{0},m_{*}]$.
As $m_{*}>x_0$, we have that $h'(m_{*})<0$ and therefore $\psi''(m_{*}) <0$.


\end{example}

Studying the variational problem arising from the limit of the annealed spherical integral $\overline{F}(\theta)$ and $\underline{F}(\theta)$ defined in Proposition \ref{simpleFint},  we will show that for $\theta$ large enough we can give an explicit formula as stated in the following proposition.

\begin{proposition}\label{simpleFcompact}
There exists $\theta_0>1/\sqrt{A-1}$ such that for any $\theta\geq \theta_0$, $\overline{F}(\theta) = \underline{F}(\theta) =F(\theta)$ where
$$F(\theta) = \sup_{\alpha \in (0,1] } V(\alpha),$$
with
$$ \forall \alpha>0, \ V(\alpha) = \theta^2 (A-1) \alpha^2 +\theta^2 +\frac{1}{2}\log (1-\alpha).$$
More explicitly, 
\begin{align*} F(\theta)& = \frac{\theta^2}{4} (A-1) \Big(1 +\sqrt{1-\frac{1}{\theta^2 (A-1)}}\Big)^2 +\theta^2 +\frac{1}{2} \log \Big( 1 - \sqrt{1- \frac{1}{\theta^2 (A-1)}}\Big) -\frac{1}{2} \log 2.
\end{align*}
\end{proposition}
Given the above proposition is true, the result of Proposition \ref{compactprop}  immediately follows from Proposition \ref{lb}.

We prove this proposition by first showing that $\underline F(\theta)\ge F(\theta)$ for all $\theta$ and then that, for large $\theta$,
$\overline{F}(\theta)\le F(\theta)$.
\subsection{Proof of the lower bound}
Recall that
by   Proposition \ref{simpleFint}, 
we have the following formulation of the limit $\underline{F}(\theta)$. 
$$\underline{F}(\theta)=\sup_{\alpha_1 + \alpha_2 +\alpha_3 =1 \atop \alpha_i \geq  0} \liminf_{ \delta \to 0, K\to +\infty \atop \delta K \to 0} \limsup_{N\ra+\infty} \mathcal{F}^N_{\alpha_1,\alpha_2,\alpha_3}(\delta, K)\,,$$
where 
\begin{align*} \mathcal{F}^N_{\alpha_1,\alpha_2,\alpha_3}&(\delta, K)  =  \theta^2 \big( \alpha_1^2 +2\alpha_1 \alpha_2 + B \alpha_3^2\big) \\
&+ \sup_{t_i  \in I_2 , i\leq l \atop |\sum_i t_i^2 - N\alpha_2|\leq \delta N}
\sup_{ s_i  \in I_3 , i\leq k\atop |\sum_i s_i^2 - N\alpha_3|\leq \delta N} \Big\{\frac{1}{N}\sum_{i=1}^k\sum_{j=1}^{l} L\Big( \frac{2\theta s_i t_{j}}{\sqrt{N}}\Big)  + \frac{1}{2N} \sum_{i,j=1}^{l} L\Big(\frac{2\theta t_{i}t_{j}}{\sqrt{N}} \Big) \\
&+ \sup_{ \nu_1 \in \mathcal{P}(I_1) \atop \int x^2 d\nu_1(x) = \alpha_1} \Big\{\sum_{i=1}^k \int L\Big(\frac{2\theta s_ix}{\sqrt{N}}\Big) d\nu_1(x)-H(\nu_1)\Big\} - \frac{1}{2} \log (2\pi) - \frac{1}{2}\Big\},
\end{align*}

Our goal is to show that  we can take $\alpha_{3}=0$ and in the supremum defining $ \mathcal{F}^N_{\alpha_1,\alpha_2,\alpha_3}(\delta,K)$ we can take all the $t_{i}$'s equal. 
In fact we first prove the lower bound:
\begin{lemma} \label{lemlowerbound}
For any $\theta\geq 0$, 
$$\underline{F}(\theta) \ge \sup_{\alpha \in (0,1] } V(\alpha),$$
where $V$ is defined in Proposition \ref{simpleFcompact}.
\end{lemma}
\begin{proof}
Indeed, if we take $\alpha_{3}=0$ and $t_{j}=N^{1/4}\sqrt{\frac{m_{*}}{2\theta}}, 1\le j\le l$, $\alpha_{2}\in [l m_{*}/2\theta \sqrt N-\delta, l m_{*}/2\theta \sqrt N+\delta]$, $\alpha_{1}=1-\alpha_{2}$, $\nu_{1}$ to be the Gaussian law restricted to $I_{1}$ with variance $\alpha_{1}$,  then we get the lower bound
$$\mathcal F^{N}_{\alpha_{1},\alpha_{2},0}(\delta,K)
\ge \theta^{2}(\alpha_{1}^{2}+2\alpha_{1}\alpha_{2}+\alpha_{2}^{2}A)+\frac{1}{2}\log\alpha_{1}= V(\alpha_{2})\,.$$
Hence, to derive the lower bound it is enough to remark that we can achieve any possible value of  $\alpha_{2}$ in  $[0,1]$ as  some large $N$ limit of
$l_{N} m_{*}/2\theta \sqrt N$  for some sequence of integer numbers $l_{N}$, which is obvious.
\end{proof}
\subsection{Proof of the upper bound}
The rest of this section is devoted to prove that the previous lower bound is sharp when $\theta$ is large enough.
To this end, recall that
by   Proposition \ref{simpleFint}, 
we have the following formulation of the limit $\overline{F}(\theta)$. 
$$\overline{F}(\theta)=\sup_{\alpha_1 + \alpha_2 +\alpha_3 =1 \atop \alpha_i \geq  0} \limsup_{ \delta \to 0, K\to +\infty \atop \delta K \to 0} \limsup_{N\ra+\infty} \mathcal{F}^N_{\alpha_1,\alpha_2,\alpha_3}(\delta, K)\,.$$
We first reformulate the supremum in $\mathcal{F}^N_{\alpha_1,\alpha_2,\alpha_3}(\delta, K)$ by 
denoting for $t\in I_{2}^{l}$ so that $|\sum t_{i}^{2}-N\alpha_{2}|\le\delta N$, 
$$\mu_{2}=\frac{1}{\alpha_{2 }N }\sum_{i=1}^{l}t_{i}^{2}\delta_{\frac{\sqrt{2\theta} t_{i}}{N^{1/4}}}\,.$$
$\mu_{2}$ is a positive measure on $S_2 = \{ x : \sqrt{2 \delta  \theta} \leq |x|\leq \sqrt{2
K\theta}\}$ whose total mass belongs to 
$[1-\frac{\delta}{\alpha_{2}},1+\frac{\delta}{\alpha_{2}}]$. We also denote by $S_3 = \{ x: \sqrt{K} \leq |x| \leq N^{1/4} \sqrt{\alpha_3} \}$.
Then it is not hard to see  that 
for any $\theta \geq 0$, 
\begin{equation} \label{dominFbar} \overline{F}(\theta) \leq \hat F(\theta),\end{equation}
where 
$\hat F(\theta)$ is defined by
$$\hat F(\theta)= \sup_{ \alpha_1 + \alpha_2 +\alpha_3 = 1\atop \alpha_i \geq 0} \limsup_{ \delta \to 0, K\to +\infty \atop \delta K \to 0} \limsup_{N\ra+\infty} \sup_{ \mu_2 \in \mathcal{P}(S_2) } \sup_{ s \in S_3} \mathcal G^{N}_{\alpha_1,\alpha_2,\alpha_3}(\delta, K,s,\mu_{2})
$$
if
\begin{align*} {\mathcal{G}}^N_{\alpha_1,\alpha_2,\alpha_3}(\delta, K,s,\mu_2) & =   \theta^2 \big( \alpha_1^2 +2\alpha_1 \alpha_2 + B \alpha_3^2 \big) \\
&+ 4\theta^2  \alpha_3\alpha_2 \int \psi\big( \sqrt{2\theta} s x) d\mu_2(x) + 2\theta^2 \alpha_2^2 \int \psi(xy) d\mu_2(x)d\mu_2(y) \\
&+ \sup_{ \nu_1 \in \mathcal{P}(I_1) \atop \int x^2 d\nu_1(x) = \alpha_1} \Big\{4\theta^2\alpha_3 \int x^2\psi\Big(\frac{2\theta sx}{N^{\frac{1}{4}}}\Big) d\nu_1(x)-H(\nu_1)\Big\} - \frac{1}{2} \log (2\pi) - \frac{1}{2}.
\end{align*}
Indeed, the upper bound  proceeds in two steps: first we take  the supremum over all  measures $\mu_{2}$ on  $S_{2}$ with mass in $[1-\frac{\delta}{\alpha_{2}},1+\frac{\delta}{\alpha_{2}}]$, and then restrict ourselves to probability measures as $\delta$ goes to zero (since $\psi$ is bounded). Then, we observe that
 for any  $\mu_2 \in\mathcal{P}(S_2)$, $\nu_1 \in \mathcal{P}(I_1)$, and $s \in S_3^k$ such that $|\sum_{i} s_i^2 -\alpha_3 \sqrt{N}|\leq \delta \sqrt{N}$,
\begin{align*} \frac{\alpha_2}{\sqrt{N}} \sum_{i=1}^k s_i^2 \int \psi\big( \sqrt{2\theta} s_i x) d\mu_2(x)&+\frac{1}{\sqrt{N}}\sum_{i=1}^k s_i^2 \int x^2\psi\Big(\frac{2\theta s_ix}{N^{\frac{1}{4}}}\Big) d\nu_1(x)\\
& \leq \alpha_3 \int \psi\big( \sqrt{2\theta} s x) d\mu_2(x)+ \int x^2 \psi\Big( \frac{2\theta sx}{N^{1/4}}\Big)d\nu_1(x)+o_{\delta}(1),
\end{align*}
where $s$ is a maximizer of the function 
$$s \in S_3 \mapsto  \alpha_2 \int \psi\big( \sqrt{2\theta} s x) d\mu_2(x) + \int x^2 \psi\Big( \frac{2\theta sx}{N^{1/4}}\Big)d\nu_1(x),$$
which ends the proof of the claim \eqref{dominFbar}. We will see that under our assumptions that $B<A$ and that the maximum of $\psi$ is uniquely achieved at $m_{*}$ such that $\psi''(m_{*})<0$, the upper bound $\hat F(\theta)$ is sharp when $\theta$ is large.

The starting point of our analysis of the variational problem defining $\hat F(\theta)$ in the regime where $\theta$ is large is the fact that $\underline{F}(\theta)$ and $\hat{F}(\theta)$ behave like $A\theta^2$. More precisely, we know from \eqref{lbFcompact} that there exists $\theta_0>0$ (depending on $A$) such that  for all $\theta\ge \theta_0,$
\begin{equation}\label{lbF} \hat{F}(\theta)\ge\underline{F}(\theta)\ge A\theta^{2}-\kappa\log\theta,\end{equation}
for some constant $\kappa>0$.

As a consequence, we can localize the suprema over $(\alpha_1,\alpha_2,\alpha_3)$ and $\mu_2$ in the definitions of  $\hat{F}(\theta)$ in some subset of the constraint set, denoted by $\mathcal{S}$, and defined as follow,
$$
 \mathcal{S} = \big\{  (\underline{\alpha}, \mu_2) \in [0,1]^3  \times  \mathcal{P}(S_2) :  \alpha_1 + \alpha_2 +\alpha_3 = 1\Big\}.$$

\begin{lemma}\label{cond} There exists a constant $\theta_0>0$ depending on $A$ such that for any $\theta \geq \theta_0$,  the
 suprema defining $\hat F(\theta)$ can be restricted to the set $\mathcal{A}_\theta \times \mathcal{B}_\theta \subset \mathcal{S}$ defined by,
$$  \underline{\alpha} \in \mathcal{A}_\theta \Longleftrightarrow 
 \alpha_{2}\ge 1-\frac{C\sqrt{\log\theta}}{\theta},\ \alpha_{1}\le \frac{C\log \theta}{\theta^2}, \ \alpha_{3}\le \frac{C\sqrt{\log\theta}}{\theta},$$
and
\begin{equation}\label{bmu2}\mu_2 \in \mathcal{B}_\theta \Longleftrightarrow
 \int  \Big(\frac{A}{2}-\psi(xy)\Big)d\mu_{2}(x)d\mu_{2}(y)\le \frac{C\log \theta}{\theta^{2}}.
\end{equation}
where $C$ is a some positive constant depending also on $A$.
\end{lemma}
\begin{proof} 
From \eqref{lbF} we deduce that we can restrict the suprema in the definitions of $\hat F(\theta)$ to the parameters $\underline{\alpha}, s,\nu_1, \mu_2$ with $\alpha_1+\alpha_2+\alpha_3=1$, $s \in S_3$, $\int x^2 d\mu_1(x) =\alpha_1$ such that,
\begin{align*}& (A-1)(\alpha_{1}^{2}+2\alpha_{1}\alpha_{2})+(A-B)\alpha_{3}^{2}
+4\alpha_2\alpha_{3}\int \Big(\frac{A}{2}-\psi(2\theta sy)\Big) d\mu_{2}(y)\\
&+2 \alpha_2^2\int\Big(\frac{A}{2}-\psi(2\theta xy)\Big) d\mu_{2}(y) d\mu_2(x)+4\alpha_{3}\int y^{2 }\Big(\frac{A}{2}-\psi(2\theta sy)\Big) d\nu_{1}(y)\\
& +\frac{1}{\theta^2} \big( H(\nu_{1}) + \log \sqrt{2\pi}\big) \le \frac{2\kappa\log \theta}{\theta^{2}}.\end{align*}
But
$$ 4\alpha_{3}\int y^{2 }\Big(\frac{A}{2}-\psi(2\theta sy)\Big) d\nu_{1}(y)+\frac{1}{\theta^2}\big(H(\nu_{1}) + \log \sqrt{2\pi}\big)   \geq \frac{1}{2} \log \frac{1}{\alpha_1}\geq 0.$$
Therefore,
\begin{align*} (A-1)(\alpha_{1}^{2}+2\alpha_{1}\alpha_{2})&+(A-B)\alpha_{3}^{2}
+4\alpha_2\alpha_{3}\int y^{2 }\Big(\frac{A}{2}-\psi(2\theta sy)\Big) d\mu_{2}(y)\\
&+2\alpha_2^2 \int\Big(\frac{A}{2}-\psi(xy)\Big) d\mu_{2}(y) d\mu_2(x)
\le \frac{2\kappa \log \theta}{\theta^{2}}.\end{align*}
Since each term is non-negative, they are all bounded by $2\kappa \log \theta/\theta^{2}$.
Note that this already yields with $C={4}\kappa /\min\{ (A-B),A-1 \}$,
\begin{equation}\label{jh}\alpha_{1}^{2}\le \frac{C\log \theta}{\theta^{2}}, \ \alpha_{3}^{2}\le \frac{C\log \theta}{\theta^{2}},  \ \alpha_{1}\alpha_{2}\le  \frac{C\log \theta}{\theta^{2}}.\end{equation}
The two first estimates imply since $\alpha_2 = 1- \alpha_1-\alpha_3$,
$$\alpha_{2}\ge 1-2\frac{\sqrt{C\log \theta}}{\theta}.$$
We can finally plug back this estimate into the last inequality of \eqref{jh} to improve the estimate on $\alpha_1$ as announced.
\end{proof}
 {Next, note that because $\psi$ is bounded continuous, the function $ \mathcal G^{N}_{\alpha_1,\alpha_2,\alpha_3}(\delta, K,s,.)$ we are optimizing over $\mu_{2}$, is bounded continuous in  $\mu_{2}$ and therefore it achieves its maximal value. We denote by $\mu_{2}$ such an optimizer. In the next lemma, we prove  that the  optimizers of $ \mathcal G^{N}_{\alpha_1,\alpha_2,\alpha_3}(\delta, K,s,.)$ are concentrated around $\sqrt{m_{*}}$ if $\psi$ takes its  maximum value  at $m_{*}$ only.}
\begin{lemma}\label{concmu2} Assume that $\psi$ achieves its maximum value at $m_{*}$ only and that it is strictly concave in an open neighborhood of this point. Let $\mu_{2}$ be an optimizer of  $\mathcal G^{N}_{\alpha_1,\alpha_2,\alpha_3}(\delta, K,s,.)$.
There exists $\eps_0>0$ such that  for any  $\mu_2 \in \mathcal{B}_\theta$,
$$ \forall 0< \eps <\eps_0, \ \mu_2\big(|x-\sqrt{m_{*}}|\ge \eps\big)\le \frac{C \sqrt{\log \theta}}{\theta \eps},$$
where $C$ is a positive constant depending on $\psi$.
\end{lemma}
\begin{proof}
Let  $\mu_2 \in \mathcal{B}_\theta$. By Lemma \ref{cond} we have,
\begin{equation} \label{bmu2bis}\int \Big(\frac{A}{2} -\psi(xy)\Big)d\mu_2(x) d\mu_2(y)\le \frac{C\log \theta}{\theta^{2}}.\end{equation}
Since $\psi$ is strictly concave in a neighborhood of $m_{*}$, and $m_{*}$ is its unique maximizer, we deduce that there exists $\eta_0>0$ such that for all $0<\eta < \eta_0$,
$$ \forall |x - m_{*}| \in [ \sqrt{\eta}, \sqrt{\eta_0}], \quad \frac{A}{2} -\psi(x) \geq \eta/c,$$
for some constant $c>0$. As $\psi$ is analytic,  it admits a finite number of local maxima. Therefore, we can find $\eta_0>0$ such that for all $0<\eta<\eta_0$,
$$ \forall |x - m_{*}| \geq  \sqrt{\eta}, \quad \frac{A}{2} -\psi(x) \geq \eta/c,$$ Since $\frac{A}{2} -\psi$ is non-negative, we deduce from \eqref{bmu2bis} that
$$\forall \eta <\eta_0, \ \mu_2^{\otimes 2}\big(|xy-m_{*}|\ge \sqrt{\eta}\big)\le \frac{C'\log \theta}{\eta \theta^{2}},$$
where $C '\geq 1$ is a constant depending on $\psi$. But  for $\eps$ small enough, we have 
$$ \mu_2( [0, \sqrt{m_{*}}-\eps])^2 \leq \mu_2^{\otimes 2}\big( xy \leq m_{*} - \sqrt{m_{*}}\eps \big) \mbox{ and }  
\mu_2\big([\sqrt{m_{*}} +\eps,+\infty)\big)^2 \leq \mu_2^{\otimes 2}\big( xy \geq m_{*} + \sqrt{m_{*}} \eps\big)$$
from which the result follows by a union bound. \end{proof}

Using Lemma \ref{concmu2}, we will show that the optimization problem over $\mu_2$  is asymptotically solved by $\delta_{\sqrt{m_{*}}}$, with an error which vanishes when $K$, and therefore the lower boundary point  of $S_3$,  goes to $+\infty$.
\begin{lemma}\label{optimmu2}
There exists $\theta_{0}$ depending on $\psi$ such that for any $\theta\ge \theta_{0}$,  $\underline{\alpha} \in \mathcal{A}_\theta$ and $s\in S_3$,
\begin{align*} \sup_{ \mu_2 \in \mathcal{B}_\theta} \Big\{2 \alpha_3 \int \psi\big( \sqrt{2\theta} s x) d\mu_2(x) + \alpha_2 \int \psi(xy) d\mu_2(x)d\mu_2(y) \Big\} = \frac{A\alpha_2}{2}+ \frac{B\alpha_3}{2} + o_K(1).
\end{align*}
\end{lemma}
\begin{proof}
Letting $\overline{\psi}(x)= \psi(x) - \frac{B}{2}$, it is equivalent to show that: 
\begin{align*} \sup_{ \mu_2 \in \mathcal{B}_\theta} \Big\{2 \alpha_3 \int \overline{\psi}\big( \sqrt{2\theta} s x) d\mu_2(x) + \alpha_2 \int \overline{\psi}(xy) d\mu_2(x)d\mu_2(y) \Big\} = \frac{(A-B )\alpha_2}{2} + o_K(1).
\end{align*}

Let us fix $\theta \geq \theta_0$ where $\theta_0$ is given by Lemma \ref{cond}. Observe that since $\psi$ is bounded continuous,
$$ Z: \mu \in \mathcal{P}(S_2) \mapsto 2 \alpha_3 \int \overline{\psi}\big( \sqrt{2\theta} s x) d\mu(x) + \alpha_2 \int \overline{\psi}(xy) d\mu(x)d\mu(y)$$
achieves its maximum value in the closed set $\mathcal B_\theta$. Let $\mu_2$ be an optimizer, and therefore a critical point of this function. Writing that $Z(\mu_2)\ge Z(\mu_2+\eps \nu)$ for all {signed} measures $\nu$ on $S_2$ such that $\mu_2+\eps \nu$ is a probability measure for small $\eps$, 
we deduce that there exists a constant $C>0$ such that,
\begin{equation} \label{critmu2} \forall x \in S_2,  \ \alpha_{3}\overline{\psi}(\sqrt{ {2} \theta} s x)+\alpha_{2}\int \overline{\psi}(xy) d\mu_{2}(y)\leq C,\end{equation}
with equality $\mu_2$-almost surely. 
Using Lemma \ref{concmu2}, we get for any $\eps$ small enough,
$$\int \overline{\psi}(xy) d\mu_2(y) = \overline{\psi}(\sqrt{m_{*}} x) + \int_{[\sqrt{m_{*}}-\eps,\sqrt{m_{*}}+\eps]} \big( \overline{\psi}(xy) - \overline{\psi}(\sqrt{m_{*}}x)\big) d\mu_2(y) + O\Big( \frac{\sqrt{\log \theta}}{\theta \eps}\Big).$$
Where we notice that our $O( \sqrt{\log \theta}/\theta \eps)$ is a function that does not depend on $\delta,K$ or $N$. 
As $L$ is the log-Laplace transform of a  sub-Gaussian  distribution, we have that ${x\mapsto | L'(x)/x|}$ is bounded. In particular, $|\psi'|$ is bounded and thus $\psi$ is Lipschitz. Therefore,  for any $x \leq M$,
$$ \int \overline{\psi}(xy) d\mu_2(y)=\overline{\psi}(\sqrt{m_{*}}x)+O\Big(\eps M+\frac{\sqrt{\log \theta}}{\theta\eps}\Big).$$
Again, $O\Big(\eps M+\frac{\sqrt{\log \theta}}{\theta\eps}\Big)$ does not depend on $\delta,K$ or $N$.
We choose $\eps=\theta^{-1/2}$ and $M = \theta^{1/4}$ so that the two error term above goes to zero when $\theta$ goes to $\infty$, so that we have for any $x \geq 0$,
\begin{equation} \label{concpsi}  \int \overline{\psi}(xy) d\mu_2(y) = \overline{\psi}(\sqrt{m_{*}} x) + o_\theta(1).\end{equation}
In particular,
$$ \alpha_{3} \overline{\psi}(\sqrt{ 2\theta} s x)+\alpha_{2}\int \overline{\psi}(xy) d\mu_2(y) = \overline{\psi}(\sqrt{m_{*}} x) + o_\theta(1).$$
Taking $x = \sqrt{m_{*}}$ in \eqref{critmu2}, we get
\begin{equation} \label{borneinf} C \geq \frac{A - B }{2}-o_\theta(1),\end{equation}
since $s\geq K$ and $1 -\alpha_2 \leq O(\frac{\sqrt{\log \theta}}{\theta})$. The term $o_{\theta}(1)$ above do not depend on $K, \delta$ or $N$.  We claim that there exists $\theta_0$ such that for any $\theta \geq \theta_0$, 
$$ \mu_2([0,\sqrt{m_{*}}/2]) =0.$$
 Indeed, if $x \leq \sqrt{m_{*}}/2$, we have by \eqref{concpsi} and the fact that $\alpha_2$ goes to $1$ as $\theta$ goes to infinity,
\begin{align*}
 \alpha_{3}\overline{\psi}(\sqrt{2\theta} s x)+\alpha_{2}\int \overline{\psi}(xy) d\mu(y)& \leq \sup_{t\leq \sqrt{m_{*}}/2} \overline{\psi}(\sqrt{m_{*}} t) + o_\theta(1),
\end{align*}
with $\sup_{t\leq \sqrt{m_{*}}/2} \overline{\psi}(\sqrt{m_{*}} t) < (A-B)/2$ since the maximum of $\psi$ is uniquely achieved at $m_{*}$. From \eqref{borneinf} and the fact that equality in \eqref{critmu2} holds  $\mu_2$-a.s, we deduce that for $\theta$ large enough (and not depending on $\delta,K$ or $N$) $[0,\sqrt{m_{*}}/2] \cap \mathrm{supp}(\mu_2) = \emptyset$. Therefore,
\begin{align*}
 2 \alpha_3 \int \overline{\psi}\big( \sqrt{2\theta} s x) d\mu_2(x) + \alpha_2 \int \overline{\psi}(xy) d\mu_2(x)d\mu_2(y) & \leq \frac{(A-B)\alpha_2}{2}+ 2\sup_{y \geq  K\frac{\sqrt{m_{*}\theta }}{2}} \overline{\psi}(y) \\
& = \frac{(A -B)\alpha_2}{2} + o_K(1).
\end{align*}Thus,
\begin{align*}
\sup_{\mu_2 \in \mathcal{B}_\theta}\Big\{ 2 \alpha_3 \int \overline{\psi}\big( \sqrt{2\theta} s x) d\mu_2(x) + \alpha_2 \int \overline{\psi}(xy) d\mu_2(x)d\mu_2(y)\Big\} \leq \frac{(A-B)\alpha_2}{2} + o_K(1).
\end{align*}
The reverse inequality is achieved by taking  $\mu_2= \delta_{\sqrt{m_{*}}}$,  which completes the proof.
\end{proof}

We deduce that  taking $\delta$ to $0$ and $K$ to $+ \infty$, we can simplify the expression of $\hat{F}(\theta)$. 
\begin{proposition}\label{simpleFstep}
There exists $\theta_0$ depending on $\psi$ such that for any $\theta \geq \theta_0$, $\overline{F}(\theta) \le \hat F(\theta)$, where
$$ \hat F(\theta) = \sup_{ (\underline{\alpha},s,\nu)\in \mathcal{S}' } \mathcal{F}(\underline{\alpha}, s,\nu),$$
with
\begin{align*}\mathcal{F}(\underline{\alpha},s,\nu) 
 & =  \theta^2 \big( \alpha_1^2 +2\alpha_1 \alpha_2 \big) + \theta^2 A \alpha_2^2 + \theta^2 B (\alpha_3^2 + 2 \alpha_3 \alpha_2) \\ 
&+ 4\theta^2\alpha_3 \int x^2\psi(2\theta s\sqrt{\alpha_3} x) d\nu(x)-H(\nu)- \frac{1}{2} \log (2\pi) - \frac{1}{2},
\end{align*}
and 
$$ \mathcal{S}' = \Big\{ (\underline{\alpha}, s,\nu) \in [0,1]^3\times [0,1]\times \mathcal{P}(\R) : \alpha_1+\alpha_2+\alpha_3 = 1, \int x^2 d\nu(x) = \alpha_1\Big\}.$$
\end{proposition}
\begin{proof}
By Lemmas \ref{cond} and \ref{optimmu2}, we know that 
$$\hat {F}(\theta) = \sup_{\underline{\alpha} \in \mathcal{A}_\theta}  \limsup_{ \delta \to 0, K\to +\infty \atop \delta K \to 0} \limsup_{N\ra+\infty} \hat{\mathcal{F}}^N_{\alpha_1,\alpha_2,\alpha_3}(\delta, K),$$
where 
\begin{align*} \hat{\mathcal{F}}^N_{\alpha_1,\alpha_2,\alpha_3}(\delta, K) & =  \sup_{ s \in S_3}  \sup_{ \nu_1 \in \mathcal{P}(I_1) \atop \int x^2 d\nu_1(x) = \alpha_1} \Big\{4\theta^2\alpha_3 \int x^2\psi\Big(\frac{2\theta sx}{N^{\frac{1}{4}}}\Big) d\nu_1(x)-H(\nu_1)\Big\} \\
&+ \theta^2 \big( \alpha_1^2 +2\alpha_1 \alpha_2 \big) + A\alpha_2^2 + B(\alpha_3^2 +2 \alpha_2 \alpha_3) - \frac{1}{2} \log (2\pi) - \frac{1}{2},
\end{align*} 
$S_3 =[ K,N^{1/4} \sqrt{\alpha_3}]$. Using the change of variable $s\mapsto sN^{-1/4}$ we have the upper bound,
$$ \hat{F}(\theta) \leq  \sup_{ (\underline{\alpha},s,\nu)\in \mathcal{S}' } \mathcal{F}(\underline{\alpha}, s,\nu).$$
\end{proof}

We finally prove that the supremum is taken at $\alpha_3=0$.
\begin{proposition}\label{simpleFstep2}
There exists $\theta_0$ depending on $A$ such that for any $\theta \geq \theta_0$, 
$$ \sup_{ (\underline{\alpha},s, \nu) \in \mathcal{S}'} \mathcal{F}(\underline{\alpha},s,\nu) 
 =  \sup_{ (\alpha_1,\alpha_2,0,s, \nu) \in \mathcal{S}'} \mathcal{F}((\alpha_1,\alpha_2,0),s,\nu).$$
\end{proposition}
\begin{proof}
We claim that for any $((\alpha_1,\alpha_2,\alpha_3),s,\nu) \in \mathcal{S}'$ such that $\alpha_2 \geq \frac{A-1}{2A -B -1}$, we have
\begin{equation} \label{claim} \mathcal{F}((\alpha_1,\alpha_2,\alpha_3),s,\nu) \leq \sup_{\nu \in \mathcal{P}(\R)} \mathcal{F}((\alpha_1,\alpha_2+\alpha_3,0),\nu).\end{equation}
Note that

$$ \sup_{\nu \in \mathcal{P}(\R)} \mathcal{F}((\alpha_1,\alpha_2+\alpha_3,0),\nu) = \theta^2(\alpha_1+2\alpha_1\alpha_2 +2 \alpha_1\alpha_3) +\theta^2 A (\alpha_2+\alpha_3)^2 + \frac{1}{2} \log \alpha_1.$$
Now, for any $((\alpha_1,\alpha_2,\alpha_3),s,\nu) \in \mathcal{S}'$, using the fact that $\psi(x) \leq A/2$ for any $x\in \R$, we have
$$ \mathcal{F}((\alpha_1,\alpha_2,\alpha_3),s,\nu) \leq  \theta^2(\alpha_1 + 2\alpha_1\alpha_2) + \theta^2A\alpha_2^2 +2\theta^2 A\alpha_1 \alpha_3 + \theta^2B(\alpha_3^2 +2 \alpha_2 \alpha_3)+ \frac{1}{2} \log \alpha_1.$$
Therefore, it suffices to prove that for $\alpha_2$ sufficiently near 1:
$$ (A-B) ( 2 \alpha_2 \alpha_3 + \alpha_3^2) \geq 2 (A-1) \alpha_1 \alpha_3$$
This is true if 
$$
 2(A-1) \alpha_1 \leq (A -B) (\alpha_3 + 2 \alpha_2).$$
A sufficient condition for the  inequality to be true is that $(A-1) (1-\alpha_2)  \leq (A-B)\alpha_2$, which ends the proof of the claim \eqref{claim}. 
By Lemma \ref{cond}, we know that for $\theta \geq \theta_0$,
$$ \sup_{ (\underline{\alpha},s, \nu) \in \mathcal{S}'} \mathcal{F}(\underline{\alpha},s,\nu)  =  \sup_{ (\underline{\alpha},s, \nu) \in \mathcal{S}' \atop \alpha_1,\alpha_3 \leq C\sqrt{\log \theta} /\theta} \mathcal{F}(\underline{\alpha},s,\nu).$$
Hence, for $\theta$ such that 
$$1-2C \frac{\sqrt{\log \theta}}{\theta} \geq \frac{A-1}{2A -B -1},$$
we obtain  \eqref{claim}. \end{proof}
We can now conclude  from the last two Propositions  \ref{simpleFstep} and \ref{simpleFstep2},   that
for $\theta\ge \theta_{0}$
$$\overline{F}(\theta)\le  \sup_{ ((\alpha_1,\alpha_2,0),s, \nu) \in S} \mathcal{F}(\alpha_1,\alpha_2,0,s,\nu)= \sup_{\alpha \in [0,1) } V(\alpha)$$
where we  optimized over $\nu$ (at the centered Gaussian law with covariance $\alpha_1$). This completes the proof of the proof of Proposition \ref{simpleFcompact}  with Lemma \ref{lemlowerbound}.

\section{Delocalization and localization of the eigenvector of the largest eigenvalue}\label{section:loc}
In this section we consider a unit eigenvector $u_{X_{N}}$ associated to the largest eigenvalue of $X_N$, conditioned to deviate towards a large value. We assume hereafter that $\mu$ is compactly supported, allowing us to use the sub-Gaussian concentration property of the titled measure $\Pp^{(e,\theta)}$, as defined in \eqref{deftilt}.

We first show when $X_N$ has sharp sub-Gaussian tails, $u_{X_N}$ stays close to the set of  delocalized vectors.  Then, we show that in the case where $\mu$ is not sharp sub-Gaussian,  $u_{X_{N}}$ is close to a set of localized vectors in the sense that it contains about $\sqrt{N}$ entries of order $N^{{-1/4}}$, the other being much smaller. It should be possible to consider as well the case where $\psi$ is increasing, and we then expect that the eigenvector would localize over one entry. However, this would require more effort to obtain the required exponential estimates and we postpone this research to further investigations.

We denote by $d_{2}$ the Euclidean distance in $\mathbb R^{N}$: for a  subset $A$ of $\mathbb R^{N}$ and $u\in \mathbb R^{N}$ we set
$$d_{2}(u, A)=\inf\{\|u-v\|_{2}:v\in A\}\,.$$
\begin{proposition}\label{delocEV}
Assume that $\mu$ has sharp sub-Gaussian tail and  is compactly supported.  Let $\varepsilon > 0$ and define the set of delocalized vectors $D_{\eps}$  by:
$$D_{\eps} := \{ e \in \mathbb S^{N-1}: \forall i \in \{1,\ldots, N\}, \  |e_i| \leq \eps N^{1/4}\}.$$  There exists  a function $\eta(x)$ that goes to zero when $x$ goes  to $+\infty$ such that
$$\lim_{\eps\rightarrow 0}\lim_{\delta\ra 0}\lim_{N\rightarrow\infty} \mathbb P\left( d_{2}(u_{X_N},B_{\eps}) \leq \eta(x)\big| |\lambda_{X_N}-x|\le \delta\right)=1\,.$$
\end{proposition}
For any  symmetric matrix $X$, we denote by $u_{X}$ a unit eigenvector associated to the largest eigenvalue. For any $\chi \in(0,1)$, we set:
$${A_{\chi}=\{ \sup_{e\in  D_{\eps} } |\langle u_{X_N},e \rangle| \leq 1 - \chi\}}\,.$$
Let $x>2$ and $\theta_x \geq 1/2$ such that $x = 2\theta_x + 1/2\theta_x$. We know from \cite[Section 5]{HuGu}  that  under the measure $\Pp_{\theta_x,N}$ defined by,
$$ d\Pp_{\theta_x,N} = \frac{I_N(\theta_x,X_N)}{\E I_N(\theta_x, X_N)} d\Pp(X),$$
 for  $\delta,\gamma>0$, $N,M$ large enough,  if $V_{\delta,x}^M = \{  |\lambda_{X_N} -x|<\delta, d(\hat\mu_{X_N},\sigma) <N^{-\gamma}, ||X_N||\leq M \}$

\begin{equation} \label{convtilt} {\Pp_{\theta_x,N}(V_{\delta,x}^M) \geq \frac{1}{2}}.\end{equation}
By \eqref{lbtilt} we know that
$$ \Pp\big(   |\lambda_{X_N}-x|\le \delta\big) \geq e^{-N(J(\theta_x,x) -F(\theta_x)+o_\delta(1))} \Pp_{\theta_x,N}( V_{\delta,x}^M).$$
 Similarly, we have
$$ \mathbb P\big(  V_{\delta,x}^M \cap A_\chi\big) \leq e^{-N(J(\theta_x,x) -F(\theta_x) -o_{\delta}(1))} \Pp_{\theta_x,N}( A_\chi \cap V^M_{\delta,x}).$$
Using \eqref{convtilt} and assumption \ref{ass} we find $f(M) \to +\infty$ when $M\to +\infty$ so that
$$ \mathbb P\left(  A_\chi\big | |\lambda_{X_N}-x|\le \delta\right) \leq 2 \Pp_{\theta_x,N}\big(  A_\chi\cap V^M_{\delta,x}\big) + e^{-N f(M)}.$$
 Using the lower bound $\log \E I_N(X_N,\theta_x) \geq N\theta_x^2 -o(N)$, we deduce
\begin{equation} \label{eqdeloc} \mathbb P\left(  A_\chi\big | |\lambda_{X_N}-x|\le \delta\right) \leq 2e^{ -N \theta_x^2} \mathbb E_e[ \mathbb E_X[\Car_{\{ A_\chi\cap V^M_{\delta,x} \}}e^{N\theta_x\langle e,X_{N} e\rangle}]] + e^{-Nf(M)},\end{equation}
Let $\kappa\in(0,1)$ and define the set $D_{\varepsilon,\kappa}=\{ e \in\mathbb{S}^{n-1}: \sum_{i,j} e_i^2 e_j^2 \Car_{\sqrt{N} |e_i e_j|>\varepsilon^2/4}> \kappa\}$. Since we assumed that $X_N$ has sharp sub-Gaussian tails, we have that  $r_\varepsilon=\inf_{y\ge \varepsilon^2/4}(\psi(0)-\psi(y)) >0$. Therefore, for any ${e\in D_{\eps,\kappa}}$,
$$\sum_{i=1}^N L(\sqrt{2N} \theta_x e_i^2) + \sum_{i< j} L(2 \sqrt{N} \theta_x e_i e_j) -N\theta_x^2\le -\theta_x^2 r_\varepsilon \kappa N.$$
We deduce that
\begin{equation}\label{tiltDe} e^{-\theta_x^2 N}  \mathbb E_e[ \Car_{e\in D_{\varepsilon,\kappa}}\mathbb E_X[e^{N\theta_x\langle e,X_{N} e\rangle}]]\le e^{-\theta_x^2r_\varepsilon \kappa N}\,.\end{equation}
On the other hand, observe that for $e\in D_{\varepsilon,\kappa}^c$,
$$\Big(\sum_{i=1}^N e_i^2 1_{|e_i|\ge \varepsilon N^{-\frac{1}{4}}/2}\Big)^2\le \sum_{i,j =1}^N e_i^2 e_j^2 1_{\sqrt{N} |e_i e_j|>\varepsilon^2/4}\le \kappa\,.$$
Therefore if we let $\bar e_i=\mbox{sgn}(e_i)\min\{|e_i|, \varepsilon  N^{-1/4}/2\}$, we have that
$$\left|\langle e, X e\rangle-\langle  \bar e, X \bar e\rangle\right|\le 2\|X\| \sqrt{\kappa}\,.$$
Thus, we can write 
$$ \mathbb E_e[ \Car_{e\in D_{\varepsilon,\kappa}^c}\mathbb E_X[\Car_{A_{\chi}\cap V_{\delta,x}^M}e^{N\theta_x\langle e,X_{N} e\rangle}]] \le 
e^{2\theta_xM\sqrt{\kappa} N}  \mathbb E_e[\Car_{e\in D_{\varepsilon,\kappa}^c} \mathbb E_X[\Car_{\{A_\chi\cap V^M_{\delta,x}\}}e^{N\theta_x\langle \bar e,X \bar e\rangle}]].$$
But for $e\in D_{\varepsilon,\kappa}^c$,
$$\mathbb E_X[e^{N\theta_x\langle \bar e,X \bar e\rangle}]\leq e^{\theta_x^2 N\|\bar e\|_2^2}\leq e^{\theta_x^2 N} , $$
which implies that
\begin{equation}\label{zs}
 \mathbb E_e[ \Car_{e\in D_{\varepsilon,\kappa}^c}\mathbb E_X[\Car_{A_{\chi}\cap V_{\delta,x}^M}e^{N\theta_x\langle e,X_{N} e\rangle}]\le  e^{(2M\sqrt{\kappa} -\theta_x^2)N}
  \mathbb E_e[\Car_{e\in D_{\varepsilon,\kappa}^c} \Pp^{(\bar e ,\theta_x)}(A_\chi)]\end{equation}
where 
\begin{equation} \label{defmesuretilt} \Pp^{(\bar e ,\theta_x)} =   \frac{ e^{N\theta_x\langle \bar e,X \bar e\rangle}}{\mathbb E_X[e^{N\theta_x\langle \bar e,X \bar e\rangle}]}d\mathbb P(X).\end{equation}
We can conclude from \eqref{eqdeloc}, \eqref{tiltDe} and \eqref{zs} that
\begin{align*}
\mathbb P\left(  A_\chi\big | |\lambda_{X_N}-x|\le \delta\right) &\leq 2e^{2N\theta_xM\sqrt{\kappa}}   \mathbb E_e[\Car_{e\in D_{\varepsilon,\kappa}^c} \Pp^{(\bar e ,\theta_x)}(A_\chi)] \nonumber \\
&+ 2e^{- N\theta_x^2r_\eps \kappa}+e^{-Nf(M)}.\label{eqdeloc}
\end{align*}
Hence, it is sufficient to  complete the proof of Proposition \ref{delocEV} to prove the following:
\begin{lemma}\label{deloc} There exists a numerical constant $C>0$  and a positive function $h$, such that for $\kappa$ small enough, $N$ large enough, $\chi \geq C\eps^2$ and any $e \in D_{\eps,\kappa}^c$,  
$$ \Pp^{(\bar e ,\theta_x)}(A_{\chi})\leq e^{-Nh(\chi)}.$$
\end{lemma}
\begin{proof}We can proceed as in \cite[section 5.1]{HuGu} and observe that under   $\mathbb P^{(\bar e,\theta_x)}$, $X_{N}$ is symmetric and has independent entries with distribution
$$\bigotimes_{i\le j} dP_{N}^{2^{i\neq j}\theta\sqrt{N} \bar e_{i }\bar e_{j}}(X_{ij})$$
where $P_{N}^{\gamma}$ is the law of $x/\sqrt{N}$ under $e^{\gamma x} d\mu(x)/\int e^{\gamma y} d\mu(y)$. Using the fact that $ \bar e_{i }=O(\eps N^{-1/4})$, we see that 
we can write $X_N=W+M_{\bar e,\theta_x}$ where $M_{\bar e,\theta_x}=\mathbb E_{(\bar e ,\theta_x)} [X_{N}]=2\theta_x \bar e \bar e^T+r_{N}$,  with {$\|r_{N}\| = O(\eps^2)$}. We  denote $\hat W=W+r_{N}$. 
{Provided that $\lambda_{X_N} >\lambda_{\hat W}$}, a unit eigenvector $u$ of $X_N$ associated to $\lambda_{X_N}$ satisfies the equation:
$$(\hat W+2\theta_x \bar e\bar e^T )u=\lambda_{X_N} u\Rightarrow (\hat W-\lambda_{X_N})u=2\theta_x \langle u,\bar e\rangle \bar e \Rightarrow u=\frac{(\hat W-\lambda_{X_N})^{-1}  \bar e}{\|(\hat W-\lambda_{X_N})^{-1}  \bar e\|_2}.$$
Therefore 

$$|\langle u, \bar e\rangle |^2= \frac{ \langle \bar e, (\lambda_{X_N}-\hat W)^{-1}\bar e\rangle ^2}{  \langle \bar e, (\lambda_{X_N}-\hat W)^{-2}\bar e\rangle}\,.$$
\begin{lemma}\label{deloctilt}
For $N$ large enough, and for any $e \in D_{\eps,\kappa}^c$, $\delta>C \eps^2$, and $K\geq C$ for some constant $C>0$,
 \begin{equation} \label{conclamdamax} \Pp^{(\bar e,\theta_x)}\big( ||\hat W|| >K \big) \leq e^{-c(K)N},\qquad \Pp^{(\bar e,\theta_x)}\big(| \lambda_{X_{N}} -x|\ge \delta  \big)\le e^{-c(\delta) N},\end{equation}
where $c$ is a positive function increasing to infinity.
Furthermore, for any $\chi \geq C\eps^2$,
$$ \Pp^{(\bar e,\theta_x)}\Big(  |\langle u_{X_N}, \bar e\rangle |^2 \leq 1-\chi\Big) \leq e^{-h(\chi)N},$$
where $h$ is a positive function.
\end{lemma}
\begin{proof} The first statement follows from Remark \ref{tensionexporadius}.  The second  claim is the consequence of  Talagrand's concentration inequality for convex Lipschitz functions (see \cite[Corollary 4.10]{Ledouxmono}) and the fact that  $\E^{(\bar e, \theta_x)}  \lambda_{X_N} = x + O(\eps^2)$. Indeed, note that $W$ is a centered random symmetric matrix with independent entries above the diagonal with variance close to $1/N$. It is known as the BBP transition, (see \cite{BBP}, \cite{BGM} for example), that $\lambda_{W + 2\theta_x \bar e \bar e^T}$ converges to $x$, almost surely and in expectation. Since $|| r_N|| =O(\eps^2)$, we deduce that for $N$ large enough, ${\E^{(\bar e, \theta_x)} \lambda_{X_N} = x + O(\eps^2)}$.

Let now $x>2$, $K\geq 1$ and $\delta\in(0,1)$ such that $x-\delta>2K$. We have on the event $V=\{ ||\hat W||\leq K, \ |\lambda_{X_N} - x| \leq \delta\}$, we have
$$\Big| |\langle u_{X_N}, \bar e\rangle |^2-  v_{N}\big| \leq C \delta, \text{ where } v_{N}=\frac{ \langle \bar e, (x-\hat W)^{-1}\bar e\rangle ^2}{  \langle \bar e, (x-\hat W)^{-2}\bar e\rangle}\,,$$
and $C$ is a numerical constant. 
Moreover, one can check that on the event $V$,
$$ | v_N- 1|\leq \frac{C'K}{x},$$
where $C'>0$ is a numerical constant.
Therefore, for $\chi \leq CK/x+C'\delta$,
$$ \Pp^{(\bar e,\theta_x)}\Big( \big| |\langle u_{X_N}, \bar e\rangle |^2 -1  \big|>\chi\Big) \leq \Pp^{(\bar e,\theta_x)}(W^c),$$
which gives the claim by an appropriate choice of $K$ and $\delta$.
%
\end{proof}
To conclude the proof of Lemma \ref{deloc}, note that for $e \in D_{\eps, \kappa}$, we have $\bar e/||\bar e|| \in D_{\eps/2\sqrt{1-\kappa}}$. For $\kappa$ small enough, $\bar e/||\bar e|| \in D_{\eps}$ and we have:
$$ \Pp^{(\bar e ,\theta_x)}(A_{\chi})\leq \Pp^{(\bar e ,\theta_x)}\Big(  |\langle u_{X_N}, \bar e\rangle |^2 -1 \big|>\chi \Big),$$
which, using  Lemma \ref{deloctilt}, ends the proof.
\end{proof}

We next consider what happens when  $\mu$ is compactly supported and is not sharp sub-Gaussian. We shall prove that in this case, at least when we condition by deviations of the largest eigenvalue close to $x$ large, the associated eigenvector becomes close to the set
$${\mbox{Loc}_{r_{1},r_{2},\varepsilon}=\Big\{ e\in \mathbb S^{N-1}: |I_{r_2,\eps}(e)| \in [(r_1 -\eps)\sqrt{N}, (r_{1}+\eps)\sqrt{N}], \ \forall i \notin I_{r_2,\eps}(e), \ |e_i|\leq \eps N^{-\frac{1}{4}}\Big\}},$$
where for any $e\in \mathbb{S}^{N-1}$, $I_{r_2,\eps}(e)=\big\{ i: \frac{N^{-\frac{1}{4}}|e_i|}{\sqrt{r_2}} \in  [1-\varepsilon, 1+\varepsilon]\big\}$.

\begin{proposition}
Assume that $\mu$ is compactly supported and $\psi$ achieves its maximum  at a unique point $m_{*}$ where it is strictly concave. {For any $x\geq x_\mu$, let $v_{x}=2\theta_{x}/m_{*}$, where $x= 1/2\theta_x + F'(\theta_x)$}. 
$$\lim_{\varepsilon\rightarrow 0}\lim_{\delta\rightarrow 0}\lim_{N\rightarrow\infty} \mathbb P\big( d_{2}(u_{X_{N}},\mbox{Loc}_{v_{x},1/v_x,\varepsilon}) \leq c(x)\big| |\lambda_{X_N}-x|\le \delta\big)=1\,,$$
where $c$ is a function going to $0$ as $x$ goes to $+\infty$.
\end{proposition}
\begin{proof}
As in  \eqref{jh} we can replace the conditioning by  the tilt by spherical integrals of parameter $\theta=\theta_{x}$ large. We then can use Lemmas \ref{cond} and \ref{concmu2} to see that up to exponentially small probability we can restrict the integration  over a $\delta$ neighborhood of 
$$  
 \alpha_{2}^{\theta}\ge 1-\frac{C\sqrt{\log\theta}}{\theta^{2}},\ \alpha_{1}^{\theta}\le \frac{C\log \theta}{\theta^2}, \ \alpha_{3}=0,$$
and
$ \mu_2 =\delta_{\sqrt{m_{*}}}\,.$
Here, $\alpha_{j}=\sum_{i:\sqrt{N} e_{i}\in I_{j}} e_{i}^{2}$ and 
$\mu_{2}=\frac{1}{\alpha_{2 }N }\sum_{\sqrt{N}e_{i}\in I_{2}}^{l}e_{i}^{2}\delta_{\sqrt{2\theta} N^{1/4} e_{i}}\,.$ Indeed, this is a  simple consequence of the estimate of Lemma \ref{cond} which implies that we can restrict ourselves  to the space $\mathcal A_{\theta}\times \mathcal B_{\theta}$ when we estimate  the annealed free energy up to exponentially small errors. 
Note then that  $\mu_{2}$ is compactly supported  and the rate function  is smooth in $\mu_{2}$ so that we can cover the integration over $\mu_{2}$ by finitely many balls: this implies that we can restrict ourselves to a neighborhood of the minimizer $\mu_{2}=\delta_{m_{*}}$ and $\alpha_{3}=0$. Note that this implies that with exponentially large probability, the uniform vector in the annealed spherical integral belongs to $\mbox{Loc}_{r_{1},r_{2},\varepsilon}$ with $r_{2}=m_{*}/2\theta, r_{1 }r_{2}=\alpha_{2}$, hence $r_{1}=2\theta \alpha_{2}/m_{*}$. Since $\alpha_{2}$ goes to one as $\theta$ (hence $x$) goes  to infinity, we retrieve the fact that $r_{1}r_{2}$ goes to one.
Hence, as in the delocalized case we can write for every $x \geq x_0$ that for every event $E$

\[ \Pp[ E | |\lambda_{X_N}- x| \leq \delta] = \exp( (o(\varepsilon) + o(\kappa) + o(\delta))N) \E_{e}[ \mathds{1}_{ e \in {\mbox{Loc}_{r_1,r_{2},\varepsilon}} }\Pp^{(\bar{e},\theta)}[E]] \]
where for $e \in \mbox{Loc}_{r_1,r_{2},\varepsilon}$, we denoted $\bar{e}$ the vector such that all the entries which are close to $r_{2}N^{-1/4}$ are equal to this value and all the others are smaller than $\varepsilon N^{-1/4}$. Again, note  that under $\Pp^{(\bar{e},\theta)}$, we have
$$X_{N}=W+M_{\bar e, \theta}$$
where $M_{\bar e,\theta}=\mathbb{E}^{(\bar e,\theta)}[X_{N}]=2\theta (\bar e\bar e^{T}-\bar e_{1}\bar e_{1}^{T})+L_{N}+R_{N}$. $R_{N}$ is a matrix with negligible spectral radius, $\bar e_{1}$ is the restriction of $\bar e$ to the entries of order $\sqrt{r_{2}}N^{-1/4}$  and ${L_{N}=(L'( \theta \bar e_1(i) \bar e_1(j))_{i,j}}$. We may assume without loss of generality that they are the first $l=r_{1}\sqrt{N}$ indices and then $L_{N}$ is a $l\times l$ matrix with constant entry $\tilde m/\sqrt{N}$ where
$$\tilde m=\frac{\mathbb E [ xe^{2\theta r_{2}x}]}{\mathbb E [ e^{2\theta r_{2}x}]}=L'(2\theta r_{2})=L'(m_{*})\,.$$
$L_{N}$ has rank one, with non-zero eigenvalue equal to $\gamma=\tilde m r_{1}=2\theta L'(m_{*})\alpha_{2}/m_{*}$ and corresponding eigenvector $v=\frac{1}{\sqrt{\alpha_{2}}}\bar e_{1}$. Recalling that $m_{*}$ is a critical point of $\psi$, we find that
$$\frac{L'(m_{*})}{m_{*}^{2}}=\frac{2L(m_{*})}{m_{*}^{3}}\Rightarrow \frac{L'(m_{*})}{m_{*}}=2\psi(m_{*})= A\,,$$
so that $\gamma=2\theta A \alpha_{2}$.
Note that $\bar e= \sqrt{\alpha_{2}} v+ \sqrt{1-\alpha_{2}}w$ with $w$ a unit vector orthogonal to $v$. We then have up to a small error
$$M_{\bar e, \theta}=2\theta ((1-\alpha_{2})ww^{T}+\sqrt{\alpha_{2}(1-\alpha_{2})}(vw^{T}+w v^{T}))+2\theta A\alpha_{2} vv^{T}$$
We see that as $x$ goes to infinity, $\alpha_{2}$ goes to one and the largest eigenvalue of $M_{\bar e,\theta}$ goes to $2\theta A$. More generally, the largest eigenvalue of $M_{\bar e,\theta}$ is given by
$$\lambda_{2}^{\theta}=\theta\left( (A+1)\alpha_{2}+\sqrt{(A+1)^{2}\alpha_{2}^{2}-4(A-1)\alpha_{2}(1-\alpha_{2})}\right)$$ 
and eigenvector $v^{\theta}_{2}$ converging to $v$ when $\theta$ goes to infinity. 
Then, by the BBP transition \cite{BBP, rao} the largest eigenvalue of $X_{N}$ is given by $K_{\sigma}(\lambda_{2}^{\theta})$ with $K_{\sigma}(x)=x+x^{{-1}}$. We therefore conclude that the optimal coefficients $\theta,\alpha_{2}$ must satisfy
$x= K_{\sigma}(\lambda_{2}^{\theta})$.
At the same time, denoting by $v_{1}^{\theta}$ the second eigenvector of $M_{e,\theta}$, we find \cite{rao}:

\[  M_{e,\theta} = \lambda^{\theta}_1 v_1^{\theta}(v_1^{\theta})^T+ \lambda^{\theta}_2 v_2^{\theta}(v_2^{\theta})^T +o_{N}(1)\]
 where $ \lambda_1^{\theta} < \lambda_2^{\theta}$ and $v_1^{\theta},v_2^{\theta}$ are orthogonal unit vectors. 
Denoting $u_{X_{N}}$ the eigenvector associated to the largest eigenvalue $\lambda_{X_N}$ of $X_{N}$ which is close to $x$, we deduce  \cite{rao} that 

\[ u_{X_{N}} =C\left( \lambda_1^{\theta}\langle v_1^{\theta}, u \rangle  ( \lambda_{X_N}-W)^{-1} v_1^{\theta}  + \lambda_2^{\theta} \langle v_2^{\theta}, u_{X_{N}} \rangle( \lambda_{X_N} -W)^{-1}  v_2^{\theta} \right)\]
where $C$ is the constant such that $u$ is a unit vector.
In expectation, the isotropic law shows that $\langle v_{2}^{\theta},  ( \lambda_{X_N} -W)^{-1} v_1^{\theta}\rangle$ goes to zero. Hence, again by concentration  of measure arguments we see that up to events with exponentially small probability 
$$|\langle u, v_{2}^{\theta}\rangle|^{2}= -\frac{G_{\sigma}^{2}(x)}{G'_{\sigma}(x)}+o(1)$$
Notice that the right hand side goes to one as $x$ goes to infinity.
Since when $x$ goes to infinity, $\theta_{x}$ goes to infinity, $\alpha_{2}$, $v_{2}^{\theta}$ goes to $v$ which is the renormalized vector with $r_{1}\sqrt{N}$ non vanishing entries,  and 
$r_{1}= 2\theta \alpha_{2}/m_{*}$ the conclusion follows.

\end{proof}

\section{Appendix}
\subsection{Concentration for Wigner matrices with sub-Gaussian log-concave entries}\label{concappendix}
In this section we show that Assumption \ref{ass} do not require to have compact support or log-Sobolev inequality as assumed in \cite{HuGu}. This hypothesis for instance would not include sparse Gaussian variables, whereas the following proposition handles this case.
\begin{proposition}\label{concWig}\cite{GZ,AGZ} Let $\mu$ be a symmetric probability measure on $\R$ which has log-concave tails in the sense that $t \mapsto \mu( x : |x| \geq t)$ is concave, and which is sub-Gaussian in the sense that \eqref{subGauss} holds.
Let $X_N$ be a symmetric random matrix of size $N$ such that $(X_{i,j})_{i\leq j}$ are independent random variables. Assume $\sqrt{N} X_{i,j}$ and $\sqrt{N/2} X_{i,i}$ have law $\mu$ for any $i\neq j$. There exists a numerical constant $\kappa >0$ such that for any convex $1$-Lipschitz function $f : \R \to \R$, and $t\geq 0$,
\begin{equation}\label{concstatline} \Pp\Big( \big| \frac{1}{n} \tr f(X_N) - \frac{1}{n} \E \tr f(X_N) \big|>t \Big) \leq 2e^{-\frac{\kappa}{A} N^2 t^2}.\end{equation}
Moreover, for any $t>0$,
\begin{equation}\label{concvalp}  \Pp\big( | \lambda_{X_N} - \E \lambda_{X_N} |>t \big) \leq 2e^{-\frac{\kappa}{A} N t^2}.\end{equation}
One can take $\kappa = 1/8\beta^2$ with $\beta = 1680e$.
\end{proposition}
From these concentration inequalities, one can deduce  as in the Appendix of \cite{HuGu} that a Wigner matrix with entries having sub-Gaussian and log-concave laws satisfy Assumptions \ref{ass}. 
\begin{corollary}
Assume  $\mu$ satisfies the assumptions of Proposition \ref{concWig} and has variance $1$. Then the matrix $X_N$ satisfies Assumptions \ref{ass}.
\end{corollary}
We now prove Proposition \ref{concWig}. It will be a direct consequence of Klein's lemma (see \cite[Lemma 4.4.12]{AGZ}) and the following concentration of convex Lipschitz functions under $\mu^n$.
\begin{proposition}
Let $\mu$ be a symmetric probability measure on $\R$ which has log-concave tails in the sense that $t \mapsto \mu( x : |x| \geq t)$ is concave, and which is sub-Gaussian in the sense that \eqref{subGauss} holds. For any lower-bounded convex $1$-Lipschitz function $f : \R\to \R$ such that $\int f d\mu^n =0$ and any $t>0$,
$$ \mu^n \big(  x : |f(x)| >t  \big) \leq 2e^{-\frac{t^2}{4\beta^2 A}},$$
where $\beta$ is numerical constant. One can take $\beta = 1680 e$.
\end{proposition}
\begin{proof}
By \cite[Corollary 2.2]{SST}, we know that there exists a numerical constant $\beta$ such that $\mu^n$ satisfies a convex infimum convolution inequality with cost function $\Lambda^*(./\beta)$, where $\Lambda^*$ is the Legendre transfom of $\Lambda$ defined by,
$$ \forall \theta \in \R^n, \ \Lambda(\theta) = \log \int e^{\langle \theta, x\rangle } d\mu^n(x).$$ 
Moreover, $\beta$ can be taken to be $1680 e$. More precisely, for any convex lower-bounded function $f : \R \to \R$,
\begin{equation}\label{infconv} \Big( \int e^{f \Box \Lambda^*(./\beta)} d\mu^n \Big) \Big( \int e^{-f} d\mu^n\Big) \leq 1,\end{equation}
where $\Box$ denotes the infimum convolution operator, defined by
$$ f \Box \Lambda^*(./\beta) (x) = \inf_{ y  \in \R^n} \big\{ f(y) + \Lambda^*\Big( \frac{y-x}{\beta}\Big)\big\}.$$
Since $\mu$ is sub-Gaussian in the sense of \eqref{subGauss}, for any $x \in \R^n$,
$$ \Lambda^*(x) \geq \frac{1}{2A} ||x||^2,$$
where $||\ ||$ denotes the Euclidean norm in $\R^n$. Therefore,
$$ f \Box \Lambda^*(./\beta) (x)\geq  \inf_{ y  \in \R^n} \big\{ f(y) +  \frac{1}{2\beta^2 A} ||y-x||^2\big\}.$$
Assume $f$ is $L$-Lipschitz for some $L>0$. Reproducing the arguments of \cite[section 1.9, p19]{Ledouxmono} we have for any $x\in \R^n$,
\begin{align*}
 f \Box \Lambda^*(./\beta) (x)& \geq f(x) + \inf_{y \in\R^n} \big\{ - L||y-x|| + \frac{1}{2\beta^2 A} ||y-x||^2\big\}\\
& \geq f(x) - \frac{1}{2}\beta^2 AL^2.
\end{align*}
Thus, by \eqref{infconv} we deduce that
\begin{equation} \label{Laplacebound} \Big( \int e^{f} d\mu^n \Big) \Big( \int e^{-f} d\mu^n\Big) \leq e^{\frac{1}{2}\beta^2 AL^2}.\end{equation}
Assume now that $f$ is $1$-Lipschitz and $\int f d\mu^n = 0$. Using Jensen's inequality, we get for any $\lambda>0$,
$$\int e^{\lambda f} d\mu^n  \leq e^{\frac{1}{2}\beta^2 A\lambda^2}.$$
Using  Chernoff inequality we deduce that for any $t>0$, 
$$ \mu^n\big( x : f(x) \geq t \big) \leq e^{- \frac{t^2}{4\beta^2 A}}.$$
Using the symmetry in \eqref{Laplacebound} between $f$ and $-f$, we complete  the proof. \end{proof}

\subsection{A Uniform Varadhan's lemma}\label{proofVaradhan}
We prove a quantitative version of  Varadhan's lemma which  is of  independent interest.
\begin{lemma}\label{unifVaradhan}

Let $f : \R \to \R$ such that $f(0) = 0$ and $f(\sqrt{ . })$ is $L$-Lipschitz for some $L>0$. Let $M_N, m_N$ be sequences  such that $M_N = o(\sqrt{N})$ and $m_N=(1+o(1)) N$. Let $g_1,\ldots,g_{m_N}$ be independent Gaussian random variables conditioned to belong to  $[-M_N,M_N]$. Let $\delta\in(0,1)$ and $c>0$ such that $K^{-1} < c <K$ and $2\delta <K^{-1}$. Then,
\begin{align*}\Big|   \frac{1}{N} \log \E e^{\sum_{i=1}^{m_N} f\big( \frac{g_i}{\sqrt{c}}\big) } \Car_{| \sum_{i=1}^{m_N} g_i^2 - cN | \leq \delta N}& - \sup_{\nu \in \mathcal{P}([-M_N, M_N]) \atop \int x^2 d \nu =c} \Big\{ \int f\Big( \frac{x}{\sqrt{c}} \Big) d\nu(x) - H(\nu| \gamma) \Big\}\Big| \\
& \leq   \varepsilon_{L,K}(N) +\varepsilon_{L}(\delta K),\end{align*}
where $\varepsilon_{L,K}(N)\to +\infty$ as $N\to +\infty$ and $\varepsilon_L(x) \to 0$ as $x \to 0$.
\end{lemma}
Let $\eps = 1/N$ and $l_0$ be the smallest integer such that $(1+\eps)^{-l_0} \leq \eps$. Define 
$$I_{l_0} = [- (1+\eps)^{-l_0},  (1+\eps)^{-l_0}], \text{ and } B_{l_0} = \{ i : g_i \in I_{l_0} \}.$$
For any $k > -  l_0$, we set\begin{equation} \label{defIB} I_k =  \{ x \in \R : (1+\eps)^{k-1} \leq  |x| \leq   (1+\eps)^{k} \} \text{ and }B_k = \big\{ i : g_i \in  I_k \big\}.\end{equation}
Let  $\mu_k = |B_k|/m_N$. Let $k_0$ be the smallest integer such that $(1+\eps)^{k_0} \geq M_N$. Since $g_i \in [-M_N,M_N]$ for all $i$, we obtain that for any $k>k_0$, $B_k = \emptyset$.
\begin{lemma}\label{cont} Let $\delta,\eps=\frac{1}{N} \in [0,1]$ and $N/m_{N}\le 2$. On the event $\big\{ | \sum_{i=1}^{m_N} g_i^2 - cN | \leq \delta N\big\}$,
$$\Big| \frac{1}{m_N}\sum_{i=1}^{m_N} f\Big( \frac{g_i}{\sqrt{c}}\Big) - \sum_{ k=-l_0}^{k_0} \mu_k f\Big( \frac{(1+\eps)^k}{\sqrt{c}}\Big) \Big| \leq  {30(c+1) C LK \eps} \,.$$

\end{lemma}

\begin{proof}
As $f(\sqrt{.})$ is $L$-Lipschitz, we have
\begin{align*}
&\big|\frac{1}{m_N}\sum_{i=1}^{m_N} f\Big( \frac{g_i}{\sqrt{c}}\Big) - \sum_{ k=-l_0}^{k_0} \mu_k f\Big( \frac{(1+\eps)^k}{\sqrt{c}}\Big) \big|  \leq  \frac{1}{m_N} \sum_{k=-l_0}^{k_0} \sum_{i\in B_k } \Big| f\Big( \frac{g_i}{\sqrt{c}}\Big) - f\Big( \frac{(1+\eps)^k}{\sqrt{c}}\Big)\Big|\\
&\qquad \leq \frac{L}{c}
\sum_{ k=-l_0+1}^{k_0}\mu_k (1+\eps)^{2k} \big(1-(1+\eps)^{-2}\big)+ \frac{L}{c} \mu_{-l_0} (1+\eps)^{-2l_0}\,.
\end{align*}
Using the fact that $(1+\eps)^{-l_0} \leq\eps$, we deduce
$$\big|\frac{1}{m_N}\sum_{i=1}^{m_N} f\Big( \frac{g_i}{\sqrt{c}}\Big) - \sum_{ k=-l_0}^{k_0} \mu_k f\Big( \frac{(1+\eps)^k}{\sqrt{c}}\Big) \big| \leq  \frac{3\eps L}{c}\Big( \sum_{k=-l_0+1}^{k_0} \mu_k (1+\eps)^{2k} + 1\Big)\,.$$  
But, on the other hand, on the event $\big\{ | \sum_{i=1}^{m_N} g_i^2 - cN | \leq \delta N\big\}$,
$$ \sum_{k=-l_0+1}^{k_0} \mu_k (1+\eps)^{2(k-1)}\leq \frac{1}{m_N}\sum_{i=1}^{m_N} g_i^2 \leq \frac{N}{m_{N}}(c+\delta)\le 2(c+1).$$
 Thus, we  conclude that
$$\big|  \frac{1}{m_N}\sum_{i=1}^{m_N} f\Big( \frac{g_i}{\sqrt{c}}\Big) - \sum_{ k=-l_0}^{k_0} \mu_k f\Big( \frac{(1+\eps)^k}{\sqrt{c}}\Big) \big|  
\le \frac{3\eps L}{c} \Big( 2(1+\eps)^2\big(c+1\big)+1\Big)\,.$$
\end{proof}
Let $I = \{-l_0,\ldots, k_0\}$ and $\mathcal{L}_N= \big\{ y \in \R_+^I : \sum_{k\in I} y_k =1,  \ \forall k\in I, m_Ny_k\in \N \big\}.$
We know from \cite[Lemma 2.1.6]{DZ}, that for any $y \in \mathcal{L}_N$,
\begin{equation} \label{probtype} (m_N+1)^{-n} e^{ - m_NH( y | \gamma_{M_N}) } \leq \Pp( \mu_k = y_k, \  \forall k \in I) \leq  e^{ - m_NH( y | \gamma_{M_N}) } ,\end{equation}
where $I_k = [(1+\eps)^{k-1},(1+\eps)^k]$, $n= |I| = l_0+k_0+1$,  and with $\gamma$  the standard Gaussian law
$$ H( y | \gamma_{M_N}) = \sum_{ k \in I} y_k \log \frac{y_k}{\gamma_{M_N} (k)},  \quad \text{with } \gamma_{M_N}(k ) =\frac{\gamma(I_k)}{\gamma([-M_N,M_N])}.$$

Let $\mu = (\mu_k)_{k \in I}$ and denote $$\mathcal A_{C_{1},C_{2}}= \big\{ y \in \mathcal{L}_N: c-\delta +C_{1}\eps \leq \sum_{k = -l_0}^{k_0} (1+\eps)^{2k} y_k \leq c+\delta + C_{2}\eps\big\}$$ 
Then, by the previous lemma we see that there exists a finite constant $C=O(KL)$ such that if we denote $\mathcal A_{\eta}=\mathcal A_{-\eta C, \eta C}$, for $N$ large enough,
$$ \big\{ \mu\in \mathcal A_{-}\big\}\subset  \big\{ |\sum_{i=1}^{m_N} g_i^2 - cN | \leq \delta N \big\}
  \subset  \big\{ \mu\in \mathcal A_{+}\big\}\,.$$
  We used here the fact that the $g_{i}$ belong to $[-M_{N},M_{N}]$ and that $(I_{k})_{k}$ is a partition of this set.
 By \eqref{probtype}, we get the upper bound, 
\begin{equation}\label{sup} \E e^{m_N\sum_{ k=-l_0}^{k_0} \mu_k f\big( \frac{(1+\eps)^k}{\sqrt{c}}\big) } 
\Car_{ \mu \in \mathcal{A}_+} \leq  \sum_{ y \in \mathcal{A}_+}  e^{m_N \sum_{ k=-l_0}^{k_0} y_k f\big( \frac{(1+\eps)^k}{\sqrt{c}}\big)} e^{-m_N H(y| \gamma_{M_N})},\end{equation}
whereas for the lower bound,
\begin{equation}\label{inf} \E e^{m_N\sum_{ k=-l_0}^{k_0} \mu_k f\big( \frac{(1+\eps)^k}{\sqrt{c}}\big) } 
\Car_{ \mu \in \mathcal{A}_-} \geq   (m_N+1)^{-n} \sum_{ y \in \mathcal{A}_-}  e^{m_N\sum_{ k=-l_0}^{k_0} y_k f\big( \frac{(1+\eps)^k}{\sqrt{c}}\big)} e^{-m_N H(y| \gamma_{M_N})}.\end{equation}

%
%
%
%
%
%
%
%
%
Let $y\in\mathcal{A}_+$ and define $\nu \in \mathcal{P}(\R)$ by $d\nu(x) = \phi (x)d\gamma (x)$, where
$$ \phi(x) = \sum_{k =-l_0}^{k_0} \Car_{x \in I_k} \frac{y_k}{\gamma(I_k)}.$$
With this notation, we have
$$ H(y|\gamma_{M_N}) = H(\nu|\gamma) - \log \gamma([-M_N,M_N]).$$
With the same argument as in Lemma \ref{cont}, we also have  for $y\in\mathcal{A}_+$,
\begin{equation} \label{contf} \big| \sum_{k=-l_0}^{k_0} y_k f\Big( \frac{(1+\eps)^k}{\sqrt{c}}\Big) - \int f \Big(\frac{x}{\sqrt{c}}\Big) d\nu(x) \big| \leq  C\eps ,
|\int x^2 d\nu(x)-c|\le \delta+ C\eps  \end{equation}
where $C$ only depends on $c$. 
From \eqref{sup} and Lemma \ref{cont}, we deduce that
$$ \frac{1}{N} \log \E e^{\sum_{i=1}^{m_N} f\big( \frac{g_i}{\sqrt{c}}\big) } \Car_{| \sum_{i=1}^{m_N} g_i^2 - cN | \leq \delta N}\leq  \frac{m_N}{N}\sup_{ \nu\in \mathcal{P}([-M_N,M_N] \atop  | \int x^2 d\nu(x) - c| \leq \delta+ C\eps} \big\{ \int f\Big(\frac{x}{\sqrt{c}}\Big) d\nu(x)- H(\nu|\gamma) \big\} +O(\eps)
$$
To complete the proof of the upper bound, we show the following result.

\begin{lemma}
Let $K,L,\delta>0$ such that $\delta <2K^{-1}$ and $\eps=\frac{1}{N}$. There exists a function $s_{L,K}$ depending on $K$ and $L$ such that for any function $f : \R \to \R$ such that $f(0) = 0$ and $f(\sqrt{.})$ is $L$-Lipschitz, and any $K^{-1} <c<K$,
\begin{align*} \sup_{ \nu\in \mathcal{P}([-M_N,M_N] \atop | \int x^2 d\nu(x) - c| \leq \delta+C\eps} \big\{ \int & f\Big(\frac{x}{\sqrt{c}}\Big) d\nu(x) - H(\nu|\gamma) \big\} \\
& \leq \sup_{ \nu\in \mathcal{P}([-M_N,M_N] \atop  \int x^2 d\nu(x) = c} \big\{ \int f\Big(\frac{x}{\sqrt{c}}\Big) d\nu(x)- H(\nu|\gamma) \big\} + s_{L}((\delta+\eps) K),
\end{align*}
where $s_{L}(x) \to 0$ as $x \to 0$.
\end{lemma}

\begin{proof}
Let $\nu \in \mathcal{P}([-M_N,M_N])$ such that $|\int x^2 d\nu(x) - c|\leq \delta$.  Let $\tilde\nu=h_{\lambda}\#\nu$ with the notations of \eqref{pushf}.With the same arguments that below  \eqref{pushf}, we easily see that
$$
 H(\tilde \nu|\gamma) \leq H(\nu|\gamma) + \frac{1}{2} \delta +\frac{L}{2} \delta K,$$ 
which ends the proof.

\end{proof}

For the lower bound, fix $\nu$ a probability measure on $[-M_N,M_N]$ such that $\nu \ll \gamma$. We set  $\eps = \eps_N$ such that $M_N^2/ m_N\eps_N \to 0$, and we define $I_k$ and $B_k$ as in \eqref{defIB}. Define, for $k\in\{-l_0+1,\ldots, k_0\}$,
$$ y_k = \frac{1}{m_N} \lfloor m_N \nu(I_k) \rfloor,$$ 
and $y_{-l_0} =1- \sum_{k= -l_0+1}^{k_0} y_k$. We claim that for $N$ large enough and independent of $\nu$,
$$ \int x^2 d\nu(x) =c \Longrightarrow y \in \mathcal{A}_-.$$
Indeed, one can check that on one hand
$$ \int x^2 d\nu(x) - \frac{(1+\eps)^2 M_N^2}{m_N \eps}\le  \sum_{k=-l_0}^{k_0} y_k (1+\eps)^{2k} \leq (1+\eps)^2 \int x^2 d\nu(x) + \eps^2\,.$$
 We obtain from \eqref{inf},
\begin{equation}\label{lowerboundV} \log \E e^{m_N\sum_{ k=-l_0}^{k_0} \mu_k f\big( \frac{(1+\eps)^k}{\sqrt{c}}\big) } 
\Car_{ \mu \in \mathcal{A}_-} \geq   (m_N+1)^{-n}   e^{m_N\sum_{ k=-l_0}^{k_0} y_k f\big( \frac{(1+\eps)^k}{\sqrt{c}}\big)} e^{-m_N H(y| \gamma_{M_N})}.\end{equation}
In the next lemma we compare $H(y|\gamma_{M_N})$ and $H(\nu|\gamma)$.

\begin{lemma}\label{approxentro}
$$H(y| \gamma_{M_N}) \leq H(\nu|\gamma) + o_N(1).$$

\end{lemma}

\begin{proof}By definition we have,
\begin{equation}\label{defentrop}  H(y| \gamma_{M_N})=  \sum_{k=-l_0}^{k_0} y_k\log \frac{y_k}{\gamma(I_k)} + \log \gamma([-M_N, M_N]).\end{equation}
Let $f(x) = x\log x$ for $x>0$ and $f(0)=0$. We claim that
\begin{equation}\label{conventropy} \forall 0\leq x<y, \ f(x) \leq f(y) +(y-x).\end{equation}
Indeed, either $x>e^{-1}$ and $f(x)\leq f(y)$ since $f$ is increasing on $[e^{-1}, +\infty)$. Or $x<e^{-1}$ and by convexity,
$$ f(x) \leq f(y) + f'(x)(x-y).$$
Since  $|f'(x)| \leq 1$ we get the claim.
Note that we have for any $k>-l_0$,
$$ \nu(I_k)-\frac{1}{m_N}< y_k\leq \nu(I_k),
\mbox{ and  }
 \nu(I_{-l_0}) \leq y_{-l_0} < \nu(I_{-l_0}) +\frac{k_0+l_0}{m_N}.$$
Thus we deduce from \eqref{conventropy} that
\begin{align*}
 H(y| \gamma_{M_N}) & \leq \sum_{k=-l_0}^{k_0} \nu(I_k) \log \frac{\nu(I_k)}{\gamma(I_k)} + \sum_{k= -l_0+1}^{k_0} \gamma(I_k) \frac{1}{\gamma(I_k) m_N}  + \gamma(I_{l_0}) \frac{k_0+l_0}{\gamma(I_{-l_0}) m_N }+o_N(1)\\
& \leq \sum_{k=-l_0}^{k_0} \nu(I_k) \log \frac{\nu(I_k)}{\gamma(I_k)} +  \frac{2(k_0+l_0)}{m_N} +o_N(1).
\end{align*}
We have $k_0 = O(\log(M_N)/\eps_N)$ and $l_0 = O(\log (1/\eps_N)/\eps_N)$. Since $M_N^2/m_N\eps_N \to0$, we get
$$  H(y| \gamma_{M_N})   \leq \sum_{k=-l_0}^{k_0} \nu(I_k) \log \frac{\nu(I_k)}{\gamma(I_k)} +  \frac{2(k_0+l_0)}{m_N} +o_N(1).$$
Since $f : x \mapsto x\log x$ is convex, we complete the proof  by  using Jensen's inequality which yields
$$
 \sum_{k=-l_0}^{k_0} \nu(I_k) \log \frac{\nu(I_k)}{\gamma(I_k)}  = \sum_{k=-l_0}^{k_0} \gamma(I_k)  f\Big (\frac{1}{\gamma(I_k)}\int_{I_k} \frac{d \nu}{d\gamma } d \gamma\Big) \leq \sum_{k=-l_0}^{k_0} \int_{I_k}  \frac{d \nu}{d\gamma }  \log  \frac{d \nu}{d\gamma }  d\gamma.
$$

\end{proof}
Moreover, we can compare $\int f(x/\sqrt{c}) d\nu(x)$ and $\sum_{k=-l_0} y_k f((1+\eps)^k/\sqrt{c})$.

\begin{lemma}
\label{contfinf}
$$\Big| \int f\Big( \frac{x}{\sqrt{c}}\Big) d\nu(x) - \sum_{k=-l_0}^{k_0} y_k f\Big( \frac{(1+\eps)^k}{\sqrt{c}}\Big)\Big |\leq \varepsilon_{L,K}(N),$$
where $\varepsilon_{L,K}(N)\to0$ as $N\to +\infty$.
\end{lemma}
\begin{proof}
As $f(\sqrt{.})$ is $L$-Lipschitz, we have on one hand using the same argument as in the proof of Lemma \ref{cont},
$$\Big| \int f\Big( \frac{x}{\sqrt{c}}\Big) d\nu(x) - \sum_{k=-l_0}^{k_0} \nu(I_k) f\Big( \frac{(1+\eps)^k}{\sqrt{c}}\Big)\Big| \leq \frac{3L\eps}{c}\Big( \int x^2 d\nu(x) +1\Big).$$
Therefore,
\begin{equation}\label{contnu} \Big| \int f\Big( \frac{x}{\sqrt{c}}\Big) d\nu(x) - \sum_{k=-l_0}^{k_0} \nu(I_k) f\Big( \frac{(1+\eps)^k}{\sqrt{c}}\Big)\Big| \leq \varepsilon_{L,K}(N),\end{equation}
where $\varepsilon_{L,K}(N) \to 0$ as $N\to +\infty$.
On the other hand, as $|\nu(I_k)-y_k|\leq 1/m_N$ for any $k>-l_0$ and $|\nu(I_{-l_0})-y_{-l_0}|\leq (k_0+l_0)/m_N$, we get since $f(0)=0$ and $f(\sqrt{.})$ is $L$-Lipschitz
\begin{eqnarray*}
\sum_{k=-l_0}^{k_0} \Big|(y_k -\nu(I_k))  f\Big(\frac{(1+\eps)^k}{\sqrt{c}}\Big)\Big| & \leq& \frac{L}{c m_N} \sum_{k=-l_0+1}^{k_0} (1+\eps)^{2k} +\frac{L(k_0+l_0)}{c m_N} (1+\eps)^{-2l_0},\\
&\leq &\frac{\kappa L}{c m_N}\Big( \frac{M_N^2}{\eps_N} +(k_0+l_0)\eps_N^2 \Big)\end{eqnarray*}
As $M_N^2/m_N  \to 0$ and $k_0 = O(\log(M_N)/\eps_N)$ and $l_0 = O(\log (1/\eps_N)/\eps_N)$, we deduce that
$ \frac{(k_0+l_0)}{m_N}\eps_N = o_N(1).$
Combining the above estimate with \eqref{contf}, we get the claim.
\end{proof}
Coming back to \eqref{lowerboundV}, using the results of Lemmas \ref{approxentro} and \ref{contfinf}, we deduce 
$$   \E e^{m_N\sum_{ k=-l_0}^{k_0} \mu_k f\big( \frac{(1+\eps)^k}{\sqrt{c}}\big) } 
\Car_{ \mu \in \mathcal{A}_-} \geq   (m_N+1)^{-n}   e^{m_N\int f(\frac{x}{\sqrt{c}}) d\nu(x) -m_Ng_{L,\delta}(N)} e^{-m_N (H(\nu| \gamma) +o_N(1) )},$$
which gives at the logarithmic scale,
$$  \frac{1}{N}\log \E e^{m_N\sum_{ k=-l_0}^{k_0} \mu_k f\big( \frac{(1+\eps)^k}{\sqrt{c}}\big) } 
\Car_{ \mu \in \mathcal{A}_-} \geq \frac{m_N}{N}\Big(\int f\Big(\frac{x}{\sqrt{c}}\Big) d\nu(x)  -H(\nu| \gamma)\Big) -\varepsilon_{L,K}(N),$$
We conclude by  optimizing over the choice of  $\nu \in \mathcal{P}([-M_N,M_N])$, such that $ \int x^2 d\nu =c $.

\begin{acknowledgement}
We would like to thank Ofer Zeitouni for stimulating discussions which helped us preparing the present paper. We are also grateful to David  Coulette for the numerous simulations which contributed to build our understanding of the optimization problem arising from our variational formula for the annealed spherical integral.
\end{acknowledgement}

\bibliographystyle{spmpsci} 
\bibliography{LDPsubVecteursCMP}

\end{document}